\documentclass[11pt, a4paper]{article} 
\usepackage[T1]{fontenc}
\usepackage[utf8]{inputenc} 
\usepackage{amsmath,amssymb,amsthm}
\usepackage{mathtools}
\usepackage{fullpage,color}
\usepackage[colorlinks=true,anchorcolor=blue,filecolor=blue,linkcolor=red,urlcolor=blue,citecolor=blue]{hyperref}     

\usepackage{tikz,bm}
\usetikzlibrary{shapes.geometric, arrows.meta, positioning, calc, fit, backgrounds}  
\usepackage{geometry} 
\AtBeginDocument{%
  \setlength{\abovedisplayskip}{7pt plus 2pt minus 3pt}%
  \setlength{\belowdisplayskip}{5pt plus 2pt minus 3pt}%
  \setlength{\abovedisplayshortskip}{3pt plus 2pt minus 1pt}%
  \setlength{\belowdisplayshortskip}{5pt plus 2pt minus 2pt}%
}

\usepackage{booktabs}

\newtheorem{theorem}{Theorem}[section]

\newtheorem{lemma}[theorem]{Lemma}
\newtheorem{problem}[theorem]{Problem}
\newtheorem{proposition}[theorem]{Proposition}
\newtheorem{corollary}[theorem]{Corollary}

\theoremstyle{definition}

\newtheorem{example}[theorem]{Example}

\newtheorem{remark}{Remark}
\theoremstyle{plain}

\newcommand{\tr}{\operatorname{tr}}

\newcommand{\ex}{\mathrm{ex}}

\renewcommand{\leq}{\leqslant}
\renewcommand{\le}{\leqslant}

\renewcommand{\ge}{\geqslant}

\newcommand{\one}{\mathbf{1}}

\newcommand{\EE}{\mathbb{E}}
\newcommand{\Var}{\operatorname{Var}}
\newcommand{\supp}{\operatorname{supp}}
\newcommand{\pmax}{\lVert \bm{p} \rVert_{\infty}}
\newcommand{\PhiG}{\Phi_G}

\definecolor{figblue}{HTML}{1F4E79}
\definecolor{figamber}{HTML}{A85F00}
\definecolor{figteal}{HTML}{1D6F63} 

\newcommand{\srcc}[1]{{\scriptsize\color{black!60}#1}}
\newcommand{\srcn}[1]{{\scriptsize\bfseries\color{figamber!85!black}#1}}

\title{An entropy bridge from weighted to spectral Tur\'{a}n theorems} 

\author{Yongtao Li\thanks{E-mail address: \texttt{ytli0921@hnu.edu.cn} }  \\[2mm]
{\small Yau Mathematical Sciences Center, Tsinghua University, Beijing, China.} 
 }

\date{}

\begin{document}

\maketitle

\vspace{-0.4cm}

\begin{abstract}
A classical result of Wilf [J. Combin. Theory Ser. B (1986)] shows that every $K_{r+1}$-free graph $G$ on $n$ vertices satisfies $\lambda (G)\le (1- \frac{1}{r}) n$. Another well-known result of Nikiforov [Combin. Probab. Comput. (2002)] states that every $K_{r+1}$-free graph with $m$ edges satisfies $\lambda^2(G)\le (1-\frac{1}{r}) 2m$, which extends both Tur\'{a}n's   and Wilf's bounds. 
Let $w_\ell(G)$ be the number of walks on $\ell$ vertices in $G$, so that $w_1(G)=n$ and $w_2(G)=2m$. Nikiforov [Linear Algebra Appl. (2006)] unified the two bounds by showing that $\lambda^{\ell}(G)\le (1- \frac{1}{r}) w_{\ell}(G)$ for every $\ell \ge 1$.

In this paper, we establish an entropy bridge, and then use it to prove that for any color-critical graph $F$ with chromatic number $\chi(F)=r+1\ge 3$, there exists a constant $\lambda_0=\lambda_0(F)$ such that if $G$ is an $F$-free graph with $\lambda (G)\ge \lambda_0$, then for every $\ell\ge 1$ with $(r,\ell )\neq (2,2)$, 
\vspace{-1.5mm}
\[ \lambda^\ell (G) \le
  \Big(1-\frac1r\Big)w_\ell(G), \vspace{-1mm} \]
with equality if and only if $G$ is a regular complete $r$-partite graph; in the case $r=2$ with $\ell$ even, equality holds for every complete bipartite graph. Isolated vertices are permitted when $\ell \ge 2$. 
The pair $(r,\ell )=(2,2)$ must be excluded, since the bound fails for several forbidden graphs, e.g., $C_{2t+1}$ with $t\ge 2$. Moreover,  walks counts may also be replaced by the homomorphism counts of unbalanced trees. 
As further applications,  we extend a spectral supersaturation result of Bollob\'{a}s and Nikiforov [J. Combin. Theory Ser. B (2007)],  revisit the spectral stability results of Nikiforov [J. Graph Theory (2009)] and Li, Liu and Zhang [arXiv:2508.15271 (2025)] in a new way, 
and we also extend the entropic Tur\'{a}n theorem of Chao and Yu [J. London Math. Soc. (2026)] from $K_{r+1}$-free graphs to $F$-free graphs with $F$ color-critical. 

We provide a framework by passing through weighted Tur\'an theorems of independent interest. If $G$ is $F$-free and $\bm{p}$ is a probability vector on $V(G)$ with $\lVert \bm{p}\rVert_\infty$ sufficiently small, then  \vspace{-2mm}
\[ 2\sum_{uv\in E(G)}p_up_v
\le 1-\frac1r + o(1), \vspace{-1mm}  \]
 and the error term $o(1)$ can be
removed if and only if $F$ is color-critical. This is a Motzkin--Straus-type 
inequality in which the clique number of $G$ is replaced by $\chi (F)-1$. The proof of this weighted result combines a blow-up
argument, the graph removal lemma, the
Erd\H{o}s--Simonovits stability theorem, a
probabilistic sampling argument, and an exact
estimate near a complete $r$-partite graph.
 The bridge linking spectral inequalities to weighted inequalities is based on the entropy method for the Markov chain attached to the Perron vector.
\end{abstract}




\section{Introduction}

For a graph $G$, let $V(G)$ and $E(G)$ denote its vertex and edge sets, and write $v(G)=|V(G)|$ and $e(G)=|E(G)|$. Let $N(v)$ denote the neighborhood of a vertex $v$.  Let $A(G)=(a_{uv})$ be the adjacency matrix of $G$, where $a_{uv}=1$ if $uv\in E(G)$, and $0$ otherwise. Let $\lambda(G)$ be the spectral radius of $A(G)$, which is defined as the maximum absolute value of the eigenvalues of $A(G)$.

A graph $G$ is $F$-free if it contains no subgraph isomorphic to $F$. The  classical Tur\'an theorem \cite{Turan1941} implies that every $n$-vertex $K_{r+1}$-free graph $G$ satisfies $e(G)\le (1- \frac{1}{r})\frac{n^2}{2}$. 
The Erd\H{o}s--Stone--Simonovits theorem \cite{ES1946,ES1966} says that for any graph $F$ with chromatic number $\chi (F)=r+1$, if $G$ is an $F$-free graph on $n$ vertices,  then $e(G)\le \left( 1- \frac{1}{r} + o(1)\right)\frac{n^2}{2}$, where $o(1)\to 0$ as $n\to \infty$.  

Both bounds above measure a graph by its number of edges. A finer parameter is the spectral radius.  
Spectral graph theory extracts structural information from the eigenvalues of matrices associated with a graph, and it has settled long-standing problems outside graph theory, such as the sensitivity conjecture \cite{Huang2019} and 
the equiangular lines problem \cite{JTYZZ2021}. 
Extremal spectral graph theory asks how large the spectral radius of a graph can be when certain subgraphs are forbidden. 
In 1986, Wilf \cite{Wilf1986} proved that every $n$-vertex  $K_{r+1}$-free graph $G$ satisfies  
\begin{equation}  \label{eq-Wilf}
\lambda (G) \le \Big( 1- \frac{1}{r}\Big) n,
\end{equation} 
which implies Tur\'{a}n's bound 
since $\lambda (G)\ge 2m/n$. 
Wilf's bound estimates $\lambda(G)$ in terms of the order $n$. One may instead estimate $\lambda (G)$ by the size $m$. We call these two settings 
 the \emph{vertex-spectral} and the 
\emph{edge-spectral} settings, respectively.  
In the latter setting, Nosal \cite{Nosal1970} showed that every triangle-free graph $G$ with $m$ edges has $\lambda (G)\le \sqrt{m}$; see, e.g., \cite{LNW2021,ZLS2021,NingZhai2023,LLZ2026,CLT2026} for recent generalizations. 
 In 2002, Nikiforov \cite{Nik2002} proved that every $K_{r+1}$-free graph with $m$ edges satisfies 
\begin{equation} \label{eq-Nik-2002}
\lambda^2(G)\le \Big( 1- \frac{1}{r}\Big)2m,
\end{equation}  
with equality if and only if $G$ is a complete bipartite graph when $r=2$, or a regular complete $r$-partite graph when $r\ge 3$;  see \cite{Nik2006, Nik2009}.  
The bound \eqref{eq-Nik-2002} implies both
Tur\'{a}n's and Wilf's bounds.

\smallskip 
In 2006, Nikiforov \cite{Nik2006} proved a common generalization of the vertex- and edge-spectral bounds. 
A {\it walk on $\ell$ vertices} is a sequence $(v_1,\ldots,v_\ell)$ of not necessarily distinct vertices such that $v_iv_{i+1}\in E(G)$ for every $1\le i\le \ell-1$. Let $w_{\ell}(G)$ denote the number of walks on $\ell$ vertices in $G$. 
Nikiforov \cite[Theorem 14]{Nik2006} proved that  if $r\ge 2$ and $G$ is a $K_{r+1}$-free graph, then for every $\ell \ge 1$, 
\begin{equation} \label{eq:nikiforov-walk}
\lambda^\ell (G) \leq \Big( 1- \frac{1}{r}\Big) w_\ell(G). 
\end{equation}

The cases $\ell =1, 2$ of \eqref{eq:nikiforov-walk} are exactly
\eqref{eq-Wilf} and \eqref{eq-Nik-2002}. Other values of $\ell$ bring in the whole degree sequence, 
since $w_3(G)=\sum_{v\in V}d_v^{2}$ and
$w_4(G)=2\sum_{uv\in E}d_ud_v$.  
Recently, Chao and Yu \cite{CY2026} proved an  entropic Tur\'{a}n theorem and extended \eqref{eq:nikiforov-walk} by replacing
$w_{\ell}(G)=\hom (P_{\ell},G)$ with $\hom (T,G)$ for an arbitrary tree $T$. The entropic point of view is also the starting point of the present paper.      
 
\smallskip 

All the bounds above forbid a clique, and all of them are exact. For a general forbidden graph $F$, the Erd\H{o}s--Stone--Simonovits theorem says that only $\chi (F)$ matters asymptotically, but its error term cannot be removed in general. 
It can be removed exactly when $F$ belongs to the following class. 
A graph $F$ is called {\it color-critical} if there is an edge $uv\in E(F)$ for which $\chi(F- uv)<\chi(F)$. 
 This family of graphs, which includes cliques, odd cycles and books, plays a central role in the
development of extremal graph theory; see, e.g., \cite{Mub2010,PY2017,ZhaiLin2023}. 
In 1968, Simonovits \cite{Sim1968} extended the Tur\'{a}n theorem to color-critical forbidden graphs. 
 Let $T_{n,r}$ denote the $n$-vertex Tur\'{a}n graph, which is a complete balanced $r$-partite graph. Simonovits showed that for all sufficiently large $n$, if $G$ is an $n$-vertex $F$-free graph, 
 then $e(G)\le e(T_{n,r})$, with equality if and only if $G=T_{n,r}$. 
In 2009,  Nikiforov \cite{Nik2009-critical} proved that $T_{n,r}$ also attains the maximum spectral radius for sufficiently large $n$.  The vertex-spectral Tur\'an problem has been studied for much wider families; see, e.g., \cite{TT2017, CDT2023,  BDT2025,  WKX2023, ZL2022jctb}.   
All of these results fix the order of $G$. Fixing the size instead changes the problem,
since a graph with $m$ edges may have many vertices of low degree and the extremal graphs need not be Tur\'an graphs. 
This is the Brualdi--Hoffman--Tur\'an problem \cite{BH1985, Row1988}, taken up in the spectral setting in recent work \cite{LZS2024, LZZ2025, LLLY2026, LLZ-ESS, LLZ-Critical}. 
Walk counts interpolate between the two: since $w_1(G)=n$ and $w_2(G)=2m$, a single inequality in the variable $w_\ell$ carries both the vertex-spectral and edge-spectral bounds, and other values of $\ell$ go beyond both. This is the form in which we state our results.

 \smallskip 
  In this walk-spectral form, the Erd\H{o}s--Stone--Simonovits theorem has an analogue for every $\ell$. In 2025, Li, Liu and Zhang \cite{LLZ-ESS}  
 proved that if $\chi (F)=r+1\ge 3$, then every $F$-free graph $G$ satisfies  $ \lambda^\ell (G)\le \left( 1 \!-\! \frac{1}{r} \!+\! o(1)\right)w_\ell(G)$ for every $\ell \ge 1$.  
This unifies the Erd\H{o}s--Stone--Simonovits theorem and its vertex-spectral version due to  Nikiforov \cite{Nik2009ESB}. 
 In the case $\ell =2$, Li, Liu and Zhang \cite{LLZ-ESS} also established the edge-spectral stability theorem: if $\chi (F)=r+1\ge 3$ and $G$ is an $F$-free graph
with $m$ edges and $\lambda^{2}(G)\ge \bigl(1-\frac1r-o(1)\bigr)2m$, then $G$ is
within edit distance $o(m)$ of a complete bipartite graph when $r=2$, and of a
Tur\'{a}n graph $T_{C,r}$ on some $C\subseteq V(G)$ when $r\ge 3$. 
In a companion paper,  Li, Liu and Zhang \cite{LLZ-Critical} developed the edge-spectral stability method and extended Nikiforov's bound (\ref{eq-Nik-2002}) to color-critical forbidden graphs, showing that if $F$ is a color-critical graph with $\chi (F)=r + 1\ge 4$, then every $F$-free graph $G$ with large size $m$ satisfies 
\begin{equation} \label{eq-LLZ-critical}
  \lambda^2(G)\le \Big( 1- \frac{1}{r} \Big)2m, 
\end{equation} 
with equality if and only if $G$ is a regular complete $r$-partite graph. The hypothesis $\chi (F) \ge 4$ is necessary, since the above bound fails when $F=C_{2t+1}$ or $K_{t+1,t+1}^+$ for every $t\ge 2$; see  \cite{LZS2024,LLZ-Critical}.   
Table \ref{tab:all-unified} summarizes what is known: its rows bound $\lambda (G)$ by the order, the size, and the walk counts of $G$, and its
columns forbid a clique, a general graph $F$, and a color-critical  graph $F$.

\begin{table}[htbp]
\centering
\begin{tikzpicture}[
  font=\small,
  cell/.style={rounded corners=2.5pt, draw, line width=0.5pt,
               inner xsep=3pt, inner ysep=4.5pt, align=center,
               text width=38mm, minimum height=12.5mm},
  exact/.style={cell, draw=figblue!45, fill=figblue!8},
  asym/.style={cell, draw=black!25, fill=black!4},
  new/.style={cell, draw=figamber!65, fill=figamber!12, line width=0.9pt},
  chead/.style={font=\small\bfseries, text=figblue!85!black, align=center},
  rhead/.style={align=center, font=\small}
]
\def\cL{0}\def\cA{3}\def\cB{7.15}\def\cC{11.30}
\def\yH{3.30}\def\yRule{2.95}\def\yR{2.15}\def\yS{0.7}\def\yW{-0.75}
\begin{scope}[local bounding box=tbl] 
\node[chead] at (\cA,\yH) {$K_{r+1}$-free};
\node[chead] at (\cB,\yH) {$F$-free};
\node[chead] at (\cC,\yH) {color-critical $F$-free};
\draw[line width=0.6pt, black!35] (-0.75,\yRule) -- (13.40,\yRule); 

\node[rhead] at (\cL,\yR) {order\\[2pt]{\footnotesize\color{black!55}$w_1=n$}};
\node[rhead] at (\cL,\yS) {size\\[2pt]{\footnotesize\color{black!55}$w_2=2m$}};
\node[rhead] at (\cL,\yW) {walks\\[2pt]{\footnotesize\color{black!55}$w_\ell$}};

\node[exact] at (\cA,\yR) {$\lambda \le \bigl(1-\tfrac1r\bigr)n$\\[2pt]
   \srcc{Wilf \cite{Wilf1986}, 1986}};
\node[asym]  at (\cB,\yR) {$\lambda\le\bigl(1-\tfrac1r+o(1)\bigr)n$\\[2pt]
   \srcc{Nikiforov \cite{Nik2009ESB}, 2009}};
\node[exact] at (\cC,\yR) {$\lambda\le\bigl(1-\tfrac1r\bigr)n$\\[2pt]
   \srcc{Nikiforov \cite{Nik2009-critical}, 2009}};

\node[exact] at (\cA,\yS) {$\lambda^{2}\le\bigl(1-\tfrac1r\bigr)2m$\\[2pt]
   \srcc{Nikiforov \cite{Nik2002}, 2002}};
\node[asym]  at (\cB,\yS) {$\lambda^{2}\le\bigl(1-\tfrac1r+o(1)\bigr)2m$\\[2pt]
   \srcc{Li--Liu--Zhang \cite{LLZ-ESS}, 2025}};
\node[exact] at (\cC,\yS) {$\lambda^{2}\le\bigl(1-\tfrac1r\bigr)2m$\\[2pt]
   \srcc{Li--Liu--Zhang \cite{LLZ-Critical}, 2025}};

\node[exact] at (\cA,\yW) {$\lambda^{\ell}\le\bigl(1-\tfrac1r\bigr)w_\ell$\\[2pt]
   \srcc{Nikiforov \cite{Nik2006}, 2006}};
\node[asym]  at (\cB,\yW) {$\lambda^{\ell}\le\bigl(1-\tfrac1r+o(1)\bigr)w_\ell$\\[2pt]
   \srcc{Li--Liu--Zhang \cite{LLZ-ESS}, 2025}};
\node[new]   at (\cC,\yW) {$\lambda^{\ell}\le\bigl(1-\tfrac1r\bigr)w_\ell$\\[2pt]
   \srcn{Theorem \ref*{thm:large-lambda}}};
   \end{scope}
\begin{scope}[on background layer]
\node[rounded corners=8pt, draw=figblue!40, fill=black!1,
      line width=0.9pt, inner sep=7pt, fit=(tbl)] {};
\end{scope} 
\end{tikzpicture}
\caption{All entries assume $\chi (F)=r+1$. The entries in the last column require the relevant parameter to be sufficiently large, and the entry for $w_\ell$ excludes the pair $(r,\ell )=(2,2)$.}
\label{tab:all-unified}
\end{table}

\vspace{-3mm}
\subsection{The walk-spectral Tur\'{a}n theorem}

The last column of Table \ref{tab:all-unified},
where $F$ is color-critical, is exact: reading 
it downwards, the first two entries are the vertex-spectral bound of Nikiforov \cite{Nik2009-critical} and the edge-spectral bound of Li, Liu and Zhang \cite{LLZ-Critical}, both with the sharp constant
$1-\frac1r$ and with the regular complete
$r$-partite graphs as the extremal graphs. 
However, the entry in the lower right-hand corner was missing. At the end of \cite{LLZ-Critical}, 
Li, Liu and Zhang proposed the corresponding 
bound for every $\ell$ as a conjecture.  

\smallskip 
We develop a new framework combining the entropy method and the weighted Tur\'{a}n theorems, 
of which the conjectured walk-spectral bound is one principal consequence. 
Recall that the proof of (\ref{eq-LLZ-critical}) by  Li, Liu and Zhang \cite{LLZ-Critical} uses the edge-spectral
stability established in \cite{LLZ-ESS}, but it has no evident analogue in the walk-spectral setting with $\ell \ge 3$. 
We therefore argue differently:  
 an entropy bridge for the Markov chain attached to the Perron vector turns the spectral inequality into a weighted Tur\'an-type theorem, in which the clique number $\omega (G)$ is replaced by $\chi (F)-1$.

\begin{theorem}
\label{thm:large-lambda}
Let $F$ be color-critical with
$\chi (F)=r+1\ge 3$. There is a constant
$\lambda_0=\lambda_0(F)$ such that
if $G$ is an $F$-free
graph with $\lambda (G)\ge \lambda_0$, then for every  $\ell \ge 1$ with
$(r,\ell )\neq (2,2)$, 
\begin{equation}  \label{eq:main-bound}
  \lambda^{\ell}(G)\ \le\
  \Bigl(1-\frac{1}{r}\Bigr)w_{\ell}(G),
\end{equation}
with equality if and only if $G$ is a regular complete
$r$-partite graph, possibly together with isolated
vertices when $\ell \ge 2$. The exception is the
case $\chi (F)=3$ with $\ell$ even, in which equality holds
exactly for all complete bipartite graphs, again possibly
together with isolated vertices. 
\end{theorem}

\begin{remark}
The constant $1-\frac1r$ in Theorem
\ref{thm:large-lambda} is best possible. If $r$
divides $n$, then $T_{n,r}$ is $d$-regular with
$d=(1-\frac1r)n$, so 
$w_\ell (T_{n,r})=nd^{\ell -1}$ and $\lambda^\ell (T_{n,r})
 =(1-\frac1r)w_\ell (T_{n,r})$. Since every
$r$-partite graph is $F$-free, no smaller constant
can replace $1-\frac1r$. 
Two features distinguish Theorem \ref{thm:large-lambda} from (\ref{eq:nikiforov-walk}). First, a threshold is
unavoidable. If $F$ is color-critical with
$\chi (F)=r+1$ and $F\neq K_{r+1}$, then
$v(F)\ge r+2$, so $K_{r+1}$ is $F$-free,
and $\lambda^{\ell}(K_{r+1})= r^{\ell}$
exceeds $\bigl(1-\frac1r\bigr)w_\ell (K_{r+1})
 =(r^{2}-1)r^{\, \ell -2}$. 
Second, we must exclude $\chi (F)=3$ and $\ell =2$.
In this case, the bound $\lambda^2 (G) \le \frac{1}{2}w_2(G) = m$ fails for many graphs $F$, for which the extremal graphs are the nearly split graphs \cite{LZS2024, LLZ-Critical}.
\end{remark}

The threshold in Theorem \ref{thm:large-lambda} is imposed on the spectral radius, and $\lambda_0$ does not depend on $\ell$. It can be imposed on the walk count, which confirms a conjecture of Li, Liu and Zhang \cite{LLZ-Critical}. 

\begin{corollary} \label{cor:threshold-walks} 
For every color-critical graph with $\chi (F)=r+1 \ge 4$, and every integer $\ell \ge 1$,  
there is a constant $W_0=W_0(F,\ell )$ such
that every $F$-free graph $G$ with
$w_\ell (G)\ge W_0$ satisfies 
\begin{equation*}
  \lambda^\ell (G)\le
  \Bigl(1-\frac{1}{r}\Bigr)w_\ell (G),
\end{equation*}
with equality if and only if $G$ is a 
regular complete $r$-partite graph, possibly with
isolated vertices when $\ell \ge 2$. 
For $\chi (F)=3$ and $\ell \neq 2$, we have $\lambda^{\ell} (G)\le \frac{1}{2}w_{\ell} (G)$, and equality holds exactly for 
the regular complete bipartite graph when $\ell$ is odd; 
 for all complete bipartite graphs when $\ell$ is even. 
\end{corollary}

Indeed, suppose that $G$ is an $F$-free graph with $\lambda^\ell(G) > \left( 1- \frac{1}{r}\right)w_\ell(G)$ for some $\ell $ with $(r, \ell) \neq (2,2)$. 
Then Theorem \ref{thm:large-lambda} implies  $\lambda (G)<\lambda_0 (F)$. 
Hence, it follows that  $w_\ell (G) < \lambda^{\ell}(G)/ (1- \frac{1}{r})  <\lambda_0^{\ell}(F) /(1-\frac1r) =:W_0(F,\ell)$. 
Consequently, the bound (\ref{eq:main-bound}) holds for every $w_{\ell}(G)\ge W_0(F,\ell)$.   
Moreover, if $w_\ell (G)\ge W_0$ and  
equality holds in \eqref{eq:main-bound}, then
$\lambda^{\ell}(G)=(1-\frac1r)\, w_\ell (G)
 \ge (1-\frac1r)W_0=\lambda_0^{\ell}$, so
$\lambda (G)\ge \lambda_0$. Therefore, the equality characterization of Theorem \ref{thm:large-lambda} applies to $G$ verbatim.

\subsection{Tree homomorphisms and degree powers}

Walks are the homomorphism counts of paths, 
and the method of this paper is not tied to paths.  
If $ \hom (P_\ell ,G)\le \hom (T,G)$ held for every tree $T$ on $\ell$ vertices and every graph $G$, 
then a bound in terms of $\hom (T,G)$ 
would follow at once
 from Theorem \ref{thm:large-lambda}. 
However, this is not the case: the path need not
minimize $\hom (\,\cdot\,,G)$ over the trees on $\ell$ vertices, and the inequality already fails for a tree on seven vertices \cite{CL2014}. It does
hold when $G$ is a path or a star \cite{CL2015}, but characterizing the graphs $G$ for which it holds remains open. 
A bound for tree homomorphism counts is
therefore a genuinely different statement, 
and we prove one for the following class of trees. 

\smallskip 
For a tree $T$ on $\ell$ vertices, 
we write $c(T)$ for the size of the larger side
of its bipartition, and call $T$ \emph{unbalanced} if $2c(T)>\ell$. Every tree of odd order is unbalanced.

\begin{theorem}\label{thm:tree-intro}
Let $F$ be color-critical with $\chi (F)=r+1\ge 3$, and let $T$ be an unbalanced tree on $\ell$ vertices. There is a constant
$\lambda_0=\lambda_0(F,T)$ such that every $F$-free graph $G$ with
$\lambda (G)\ge \lambda_0$ satisfies
\[
  \lambda^{\ell}(G)\le \Bigl(1-\frac1r\Bigr)\hom (T,G).
\]
\end{theorem}

We give an explicit value of $\lambda_0$ in Section
\ref{subsec:trees}. 
The path $P_\ell$ with odd $\ell$ is
unbalanced, so Theorem \ref{thm:tree-intro} recovers the odd case of \eqref{eq:main-bound}. Taking $T=K_{1,\ell -1}$
 gives the following corollary.
 
\begin{corollary} \label{coro-spectral-power}
Let $F$ be color-critical with
$\chi (F)=r+1\ge 3$, and let $\ell \ge 1$ be an
integer with $\ell \ne 2$. There is a constant
$\lambda_0=\lambda_0(F,\ell )$ such that every
$F$-free graph $G$ with $\lambda (G)\ge \lambda_0$
satisfies
\[
  \lambda^\ell (G)\le \Bigl(1-\frac{1}{r}\Bigr)
  \sum_{v\in V(G)}d_v^{\,\ell -1}.
\]
 For $\ell =2$, the same bound holds under 
 the stronger hypothesis $\chi (F)=r+1\ge 4$.
\end{corollary}

Here the case $\ell =1$ recovers Nikiforov's bound \cite{Nik2009-critical} and the case $\ell =2$ reduces to 
the edge-spectral bound of Li, Liu and Zhang \cite{LLZ-Critical}. 
The cases $\ell \ge 3$ are new for color-critical graphs.  
By the identities for $w_3$ and $w_4$ recorded above, the cases $\ell =3, 4$ of Theorem \ref{thm:large-lambda} 
read as follows. 

\begin{corollary} \label{coro=3and4}
For every color-critical graph $F$ with
$\chi (F)=r+1\ge 3$, there is a constant
$\lambda_0=\lambda_0(F)$ such that if  $G$ is an $F$-free graph with $\lambda (G)\ge \lambda_0$, then 
\[
  \lambda^3 (G) \le\left(1-\frac1r\right)
    \sum_{v\in V(G)}d_v^{\,2}  \quad 
    \text{and} \quad 
      \lambda^4 (G) \le 2 \left(1-\frac1r\right) 
      \sum_{uv\in E(G)}d_ud_v. \]
\end{corollary}

Equivalently, if $\lambda (G)\ge \lambda_0$ and
$\lambda^3 (G)> \left(1-\frac1r \right)\sum_{v}d_v^2$, then $G$ contains a copy of 
$F$. This is a spectral criterion for color-critical graphs that involves only the degree sequence. 
 As an upper bound on
$\lambda(G)$, the quantity $((1-\frac1r)\sum_v d_v^{\,2})^{1/3}$ may be
compared with $( (1-\frac1r )2m )^{1/2}$, which follows from \eqref{eq-LLZ-critical}. Indeed,  for a $d$-regular graph,  the new bound is smaller
whenever $d<\bigl(1-\frac1r\bigr)n$. For
$\ell \ge 3$, the quantity $w_\ell (G)$ depends on the whole degree sequence rather than
only $n$ and $m$. Degree powers of $F$-free graphs are a classical extremal
quantity; see Bollob\'as and Nikiforov \cite{BN2004,BN2012}.

\subsection{Probabilistic method: an entropy-Perron bridge}

\label{subsec:bridge}

As mentioned above, the proof of Theorem \ref{thm:large-lambda} rests on two
statements: an inequality in which no forbidden subgraph appears, presented in
this subsection, and an inequality in which no eigenvalue appears, given in the next one. 
Neither of them involves the spectral stability argument.  
For the first, let $\bm{x}$ be the unit Perron
vector of $G$, and define the
probability vector $\pi_v:=x_v^{2}$ together with
 the Markov chain
$(X_t)_{t\ge 1}$ with initial distribution $\pi$ and
transition matrix $P_{uv}:=a_{uv}x_v/(\lambda x_u)$. 
Applying the entropy method to
the first $\ell$ steps of this chain gives the following.

\begin{theorem} 
\label{thm:bridge-intro}
Let $G$ be a connected graph with at least one edge, let $\bm{x}=(x_v)_{v\in V(G)}$ be its
positive unit Perron vector, and put $\pi_v:=x_v^{2}$ for every $v\in V(G)$. Then for every integer
$\ell \ge 1$,
\[
  \lambda^{\ell}(G)\ \le\ \Phi_G(\pi )\cdot w_\ell (G),
  \quad\text{where}\quad
  \Phi_G(\pi ):=2\!\!\sum_{uv\in E(G)}\!\!\pi_u\pi_v .
\]
\end{theorem}

Theorem \ref{thm:bridge-intro} is based on a single
computation: each step of the chain above has conditional entropy exactly $\log \lambda (G)$, so the chain rule determines the entropy of $(X_1,\ldots ,X_\ell )$, and the uniform bound shows that 
 this entropy is at most $\log w_\ell (G)$. We single out two features of Theorem \ref{thm:bridge-intro}. 
First, no subgraph is forbidden, 
so every hypothesis on $F$ is deferred to a bound on
$\Phi_G(\pi )$ in which no eigenvalue appears. Second,   we prove an exact identity (Theorem \ref{thm:entropy-bridge}) whose two error terms are relative entropies, and tracking them yields the characterization of equality in Theorem \ref{thm:large-lambda}. 
After the proof of Theorem \ref{thm:entropy-bridge}, we  make two comparisons: the first with the entropic Tur\'{a}n theorem \cite{CY2026} (Remark \ref{rem:compare-CY}), and the second with the weighted-graph estimate \cite{LSWW2026}, which yields an inequality of the same shape but with a weight vector for which we have no flatness estimate (Remark \ref{rem:not-replaceable}).

\smallskip 
Theorem \ref{thm:bridge-intro} recovers Nikiforov's bound \eqref{eq:nikiforov-walk}: if $G$ is
$K_{r+1}$-free, then the Motzkin--Straus theorem \cite{MS1965} gives
$\Phi_G(\pi )\le 1-\frac1r$, as needed. 
With the aid of Theorem \ref{thm:bridge-intro}, 
the rest of the paper therefore concentrates 
on a weighted extremal graph problem in which no eigenvalue appears: \textit{for
an $F$-free graph $G$ with $F$ color-critical, when is $\Phi_G(\pi )\le 1-\frac1r$?} 
The forthcoming Theorem \ref{thm-second-key} answers this question for every ``sufficiently flat'' probability vector. 
 More generally, every spectral inequality in Table \ref{tab:all-unified} can be recovered within our framework. 
 Figure \ref{fig:mechanism} shows the mechanism.

\begin{figure}[htbp]
\centering
\begin{tikzpicture}[
  font=\footnotesize,
  itembox/.style={rounded corners=2.5pt, draw, line width=0.5pt,
                  inner xsep=4pt, inner ysep=3.5pt, align=center, minimum width=44mm},
  kn/.style={itembox, draw=figblue!45, fill=figblue!8},
  nw/.style={itembox, draw=figamber!60, fill=figamber!10},
  st/.style={rounded corners=2.5pt, draw=figteal!50, fill=figteal!8, line width=0.5pt,
             inner xsep=4pt, inner ysep=4pt, align=center, minimum width=48mm},
  tl/.style={rounded corners=2pt, draw=black!22, fill=black!3, line width=0.4pt,
             inner xsep=4pt, inner ysep=4pt, align=center, font=\scriptsize,
             minimum width=48mm, text=black!65},
  bridge/.style={rounded corners=3pt, draw=black!40, fill=white, line width=0.6pt, inner xsep=5pt, inner ysep=4pt, align=center, font=\scriptsize},
  cont/.style={rounded corners=8pt, line width=0.9pt, inner sep=4pt},
  mycap/.style={font=\scriptsize\itshape, text=black!60},
  arr/.style={-{Stealth[length=2.6mm,width=2.2mm]}, line width=0.9pt}
]
\def\LX{0} \def\RX{9.75}

\node[font=\footnotesize, text=figblue] (ltitle) at (\LX, 2.30) {\textbf{Spectral Tur\'an Problems}};
\node[mycap] (lsub) at (\LX, 1.95) {$F$ color-critical, $\chi(F)=r\!+\!1$, $G$ is $F$-free};
\node[kn] (l1) at (\LX, 1.25)
   {$\ell=1$: \ $\lambda\le\bigl(1-\tfrac1r\bigr)n$\\[-2pt]
    {\color{black!45}\scriptsize ~~vertex-spectral \ (Nikiforov)}};
\node[kn] (l2) at (\LX, 0.2)
   {$\ell=2$: \ $\lambda^{2}\le\bigl(1-\tfrac1r\bigr)2m$\\[-2pt]
    {\color{black!45}\scriptsize edge-spectral  (Li--Liu--Zhang)}};
\node[nw] (l3) at (\LX,-0.85)
   {$\ell\ge 1$: \ $\lambda^{\ell}\le\bigl(1-\tfrac1r\bigr)w_{\ell}(G)$\\[-2pt]
    {\color{figamber}\scriptsize \quad walk-spectral \ (Theorem \ref*{thm:large-lambda})\quad}};
\node[nw] (l4) at (\LX,-1.9)
   {tree $T$: $\lambda^{\ell}\le\bigl(1-\tfrac1r\bigr)\hom(T,G)\!$\\[-2pt]
    {\color{figamber}\scriptsize tree-spectral \ (Theorem \ref*{thm:tree-intro})}};
    
\begin{scope}[on background layer]
\node[cont, draw=figblue!40, fill=figblue!3, fit=(ltitle)(lsub)(l1)(l4)] (LEFT) {};
\end{scope}

\node[font=\footnotesize, text=figteal] (rtitle) at (\RX, 2.30) {\textbf{Weighted Tur\'an Problems}};
\node[mycap] (rsub) at (\RX, 1.95) {No  eigenvalues};
\node[st] (r1) at (\RX, 0.9)
  {$\Phi_G(\bm p)\ :=\ 2\!\!\textstyle\sum\limits_{uv\in E}\!p_up_v$\\
   $\le\ 1-\tfrac1r$ \ {\scriptsize if $\ \lVert\bm p\rVert_{\infty}\le\eta(F)$}\\[1pt]
   {\color{figteal}\scriptsize (Theorem \ref*{thm-second-key})}};
\node[mycap] (ring) at (\RX,-0.2) {By probabilistic graph method:};
\node[tl] (r2) at (\RX,-1.43)
  {Blow-up $+$ graph removal lemma\\ 
  Chebyshev + Markov inequalities\\ 
  Erd\H{o}s--Stone--Simonovits $+$ stability\\
  A random weighted sampling\\ An exact bound near $r$-partition};
\begin{scope}[on background layer]
\node[cont, draw=figteal!40, fill=figteal!3, fit=(rtitle)(rsub)(r1)(ring)(r2)] (RIGHT) {};
\end{scope}

\coordinate (MID) at ($(LEFT.east)!0.5!(RIGHT.west)$);
\node[bridge] (BR) at ($(MID)+(0,0.68)$)
  {\textbf{Entropy Method}\\[2pt]
   $H(X_{t+1} | X_{t})=\log\lambda$\\[2pt]
   $ \frac{\lambda^{\ell}(G)}{w_{\ell}(G)}  \le  2\!\! \sum\limits_{uv\in E} \! \pi_u \pi_v$\\[2pt]
   with $\pi_v=x_v^{2}$\\[2pt]
   {\color{black!60}(Theorem \ref*{thm:bridge-intro})}};
\begin{scope}[on background layer]
\draw[arr, figblue!65!black] ([yshift=7.5mm]LEFT.east) -- ([yshift=7.5mm]RIGHT.west);
\end{scope}

\coordinate (BOT) at ([yshift=-6.5mm]LEFT.south);
\draw[arr, figamber!85!black, rounded corners=7pt]
   (RIGHT.south) |- (BOT) -- (LEFT.south);
\coordinate (BOT2) at (BOT -| RIGHT.south);
\node[font=\scriptsize, text=black, fill=white, inner xsep=4pt, inner ysep=1pt]
   at ($(BOT)!0.5!(BOT2)$)
   {Flatness lemma \cite{LLZ-ESS}: $\max_{v\in V}x_v^{2}=O_{\delta}(1/\lambda)\le\eta(F)$};
\end{tikzpicture}
\caption{The mechanism uses the entropy method
(Theorem \ref{thm:bridge-intro}) to transform every spectral inequality on the left box into one statement about the weighted quantity $\Phi_G(\pi )$ on the right.}
\label{fig:mechanism}
\end{figure}
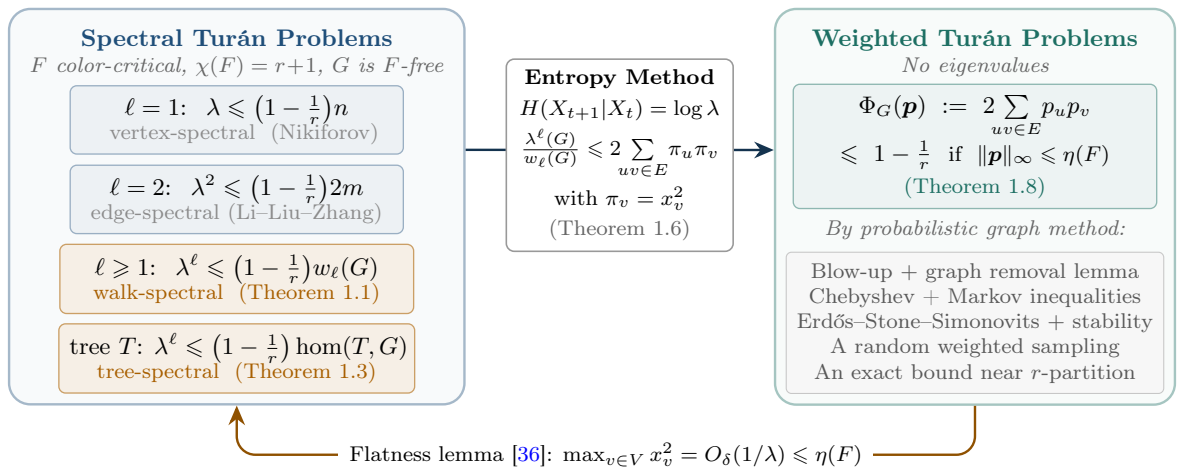

\subsection{Two weighted Tur\'{a}n theorems}

We now start with the last main part of this paper and
 answer the question raised at the end of Section \ref{subsec:bridge}. 
The quantity $\Phi_G(\pi )$ of Theorem \ref{thm:bridge-intro} makes sense for an
arbitrary probability vector $\bm{p}$ on $V(G)$, 
that is, the weight $p_v \ge 0$ for
each $v\in V(G)$ and $\sum_{v\in V(G)} p_v =1$. Now we call  
\[
  \Phi_G(\bm{p}):=2\sum_{uv\in E(G)}p_up_v 
\]
the weighted edge density of $G$ with respect to $\bm{p}$, and we write
$\lVert \bm{p}\rVert_\infty =\max_{v\in V(G)}p_v$.

The Motzkin--Straus theorem \cite{MS1965} states that
$\Phi_G(\bm{p}) \le 1-1/\omega (G)$, 
where $\omega (G)$ is the clique number of $G$. 
This inequality yields the classical bounds of Tur\'an
\cite{Turan1941}, Wilf \cite{Wilf1986} and Nikiforov \cite{Nik2002}. It is useful only for $K_{r+1}$-free graphs, however, and it yields no information for $F$-free graphs when $F$ is not a clique, since an
$F$-free graph may contain cliques larger than $K_{r+1}$.  
Weighted versions of Tur\'an's theorem have been studied in several directions, such as 
the local variants investigated in \cite{LiuNing2026, LiuNing2025, KKP2025, LSWW2026}; in all these results, the clique number of $G$ still appears and necessarily so.

In this paper, we prove that the clique number may be replaced by $\chi(F)-1$, up to an error term in general, and with no error term for color-critical $F$, whenever $\lVert \bm{p}\rVert_\infty$ is sufficiently small.

\begin{theorem}
\label{thm-weighted-ESS}
Let $F$ be a graph with $\chi (F)=r+1\ge 3$. For every $\varepsilon >0$, there is $\eta (F,\varepsilon )>0$ such that if $G$ is an $F$-free graph and
$\bm{p}$ is a probability vector on $V(G)$ satisfying
$\lVert \bm{p}\rVert_\infty\le \eta (F,\varepsilon )$, then
\[  \Phi_G(\bm{p})\ \le \ 1-\frac1r+\varepsilon .\] 
\end{theorem}

Theorem \ref{thm-weighted-ESS} is a weighted generalization of the Erd\H{o}s--Stone--Simonovits theorem. We prove it in Section \ref{sec:weighted}  in the sharper form of a weighted supersaturation theorem. 
Our next result removes the error term;  by Remark \ref{rem:necessity},  this is possible only when $F$ is color-critical. The result may be viewed as a new Motzkin--Straus-type theorem, which we believe to be of independent interest.

\begin{theorem}
\label{thm-second-key}
For every color-critical graph $F$ with $\chi (F)=r+1\ge 3$, there exists
$\eta (F)>0$ such that if $G$ is an $F$-free graph and $\bm{p}$ is a
probability vector with $\lVert \bm{p}\rVert_\infty \le \eta (F)$, then
\[ \Phi_G(\bm{p}) \ \le\  1-\frac{1}{r}. \] 
\end{theorem}

\begin{remark} \label{rem:necessity}
Both hypotheses of Theorem \ref{thm-second-key} are necessary.  
First, the bound on $\pmax$ in Theorem \ref{thm-second-key} cannot be dropped unless $F=K_{r+1}$. 
Let $F$ be color-critical with $\chi (F)=r+1$ and with $f\ge r+2$ vertices, take $G=K_{f-1}$ and $\bm{p}$ uniform. Then $G$ is $F$-free, $\pmax = \frac{1}{f-1}$ and
$ \Phi_G(\bm{p})=1-\frac{1}{f-1}>1-\frac1r $.
Hence, we get $\eta (F)\le \frac{1}{f-1}$ for every
color-critical $F\neq K_{r+1}$.  
Second, color-criticality cannot be dropped either. It is exactly the hypothesis under which the error term of Theorem \ref{thm-weighted-ESS} can be removed. Suppose that $F$ is not color-critical, let $r\mid n$, and let $G$ be
the Tur\'an graph $T_{n,r}$ with one extra edge inside a part. 
Since $F$ does
not embed in the $r$-partite graph $T_{n,r}$, every copy of $F$ in $G$ would
have to use the extra edge. Deleting from such a copy the edge $e'$ of $F$
mapped to it would exhibit $F-e'$ inside $T_{n,r}$ and force
$\chi (F-e')\le r$, which is impossible. Hence $G$ is $F$-free, while the
uniform probability vector $\bm{p}$ on $V(G)$ has
$\lVert \bm{p}\rVert_\infty =1/n$, which is as small as we please once $n$ is large, and
$\Phi_G(\bm{p})=1-\frac1r+\frac{2}{n^{2}}>1-\frac1r$. Thus, for a graph $F$
with $\chi (F)=r+1\ge 3$, the conclusion of Theorem \ref{thm-second-key} holds
for some $\eta >0$ if and only if $F$ is color-critical.
\end{remark}

Theorem \ref{thm-second-key} contains the case
$\ell =1$ of Theorem \ref{thm:large-lambda} in a
sharper form with a smaller threshold. 
Let $\bm{x}$ be a unit Perron vector of $G$. 
 We write $\lVert \bm{x}\rVert_1:=\sum_{v\in V(G)}x_v$ and put
$\bm{p}:=\bm{x}/\lVert \bm{x}\rVert_1$. 
 Then $\Phi_G(\bm{p})=\bm{x}^{\mathsf T}A(G)\bm{x}
 / \lVert \bm{x}\rVert_1^{2}
 =\lambda (G)/\lVert \bm{x}\rVert_1^{2}$. Moreover, if $i$ is a vertex with $x_i=\max_v x_v$, then 
$\lVert \bm{x}\rVert_1\ge x_i+\sum_{j\sim i}x_j
 =(1+\lambda (G))x_i$. Hence  
$\lVert \bm{p}\rVert_\infty \le 1/(1+\lambda (G))$.
Theorem \ref{thm-second-key} therefore applies as soon as $\lambda (G)\ge \eta (F)^{-1}-1$, and 
it yields the following strengthening of Wilf's bound (\ref{eq-Wilf}).

\begin{corollary}\label{cor:perron-wilf}
Let $F$ be a color-critical graph with $\chi (F)=r+1\ge 3$, and let
$\eta =\eta (F)$ be the constant of Theorem \ref{thm-second-key}. If $G$ is an
$F$-free graph with at least one edge and $\lambda (G)\ge \eta^{-1}-1$, and if
$\bm{x}$ is a unit Perron vector of $G$, then
\begin{equation*}
  \lambda (G)\le \Bigl(1-\frac1r\Bigr)\lVert \bm{x}\rVert_1^{2}. 
\end{equation*}
\end{corollary}

\begin{remark}
The Cauchy--Schwarz inequality gives
$\lVert \bm{x}\rVert_1^{2}\le n\sum_v x_v^{2}=n$, with
equality if and only if $\bm{x}$ is constant, that is, if
and only if $G$ is regular. Thus Corollary
\ref{cor:perron-wilf} refines Wilf's bound
$\lambda (G)\le (1-\frac1r)n$ for color-critical forbidden
graphs, and the gap $n-\lVert \bm{x}\rVert_1^{2}$ measures
how far $G$ is from being regular. Since
$\lVert \bm{x}\rVert_1^{2}=n-\sum_{u<v}(x_u-x_v)^{2}$,
Corollary \ref{cor:perron-wilf} reads
$\lambda (G)\le \bigl(1-\frac1r\bigr)
 \bigl(n-\sum_{u<v}(x_u-x_v)^{2}\bigr)$.
\end{remark}

\subsection{Proof overview, applications and organization}

\noindent
{\bf The spectral Tur\'an theorem.} 
This paper provides a new framework for spectral extremal graph problems; for instance, every spectral inequality in Table \ref{tab:all-unified} can be deduced from a weighted one through the entropy-Perron bridge.  
The proof of Theorem \ref{thm:large-lambda} needs three  ingredients, summarized in Figure \ref{fig:mechanism}. 
Theorem \ref{thm:bridge-intro} bounds $\lambda^{\ell}(G)/w_\ell (G)$ by
$\Phi_G(\pi )$, and Theorem \ref{thm-second-key} bounds $\Phi_G(\bm{p})$ by
$1-\frac1r$ for every sufficiently flat $\bm{p}$. 
The two bounds combine only when 
$\max_{v\in V(G)} \pi_v\le \eta (F)$, and the third ingredient supplies exactly this: a flatness estimate of Li, Liu and Zhang \cite[Lemma 3.1]{LLZ-ESS} says that every
graph $G$ with $\lambda^{\ell}(G) \ge (\frac12+\delta )w_\ell (G)$ for some fixed $\delta \in (0,0.1)$ must  satisfy
\[
  \max_{v\in V(G)}x_v^{2}=O_\delta \bigl(1/\lambda (G)\bigr).
\]
This is where  the hypothesis $\chi (F)=r+1\ge 4$ enters, since it makes $1-\frac1r$ exceed
$\frac12$ by a fixed amount.  
When $\chi (F)=3$, the desired gap disappears, and we supply two substitutes: for odd $\ell$, a flatness estimate that needs no gap (Lemma \ref{lem:flat-odd}), and for even $\ell \ge 4$, a lemma that reverses the logic (Lemma \ref{lem:large-perron}). The latter proves \eqref{eq:main-bound} directly by assuming a lower bound on $\max_v x_v^2$ instead of finding an upper bound. 
We combine the three ingredients by contradiction. 
Suppose that an $F$-free graph $G$ violates
\eqref{eq:main-bound}. On the one hand, Theorem \ref{thm:bridge-intro} then
forces $\Phi_G(\pi )>1-\frac1r$, so the contrapositive of Theorem
\ref{thm-second-key} says that $\pi$ has an entry larger than $\eta (F)$. On the
other hand, the flatness estimate says that no entry of $\pi$ exceeds
$O(1/\lambda (G))$. For sufficiently large $\lambda (G)$, 
the two conclusions are incompatible. 
 Hence,  every counterexample has bounded spectral radius.

\smallskip
\noindent
{\bf The weighted Tur\'an theorem.}
Proving Theorem \ref{thm-second-key} occupies Sections \ref{sec:weighted} and
\ref{sec-prove-exact}.

\emph{The asymptotic bound.} 
We first round $\bm{p}$ to a rational probability vector $\bm{q}$, and then construct a blow-up $B$ of $G$ by 
replacing each vertex $v$ of $G$ by an independent cluster of size proportional to
$q_v$ and joining the clusters of adjacent vertices 
by complete bipartite graphs. The resulting
blow-up $B$ on $N$ vertices has edge density $\Phi_G(\bm{p})$, up to a
negligible rounding error. Since $G$ is $F$-free, every copy of $F$ in $B$ has
two vertices in a common cluster, and small weights make the number of such
copies $o(N^{f})$. The removal lemma turns $B$ into an $F$-free graph of
almost the same density, and the Erd\H{o}s--Stone--Simonovits theorem gives
$\Phi_G(\bm{p})\le 1-\frac1r+\varepsilon$, which proves Theorem
\ref{thm-weighted-ESS}. We prove this in the sharper supersaturation form
in Theorem \ref{thm:weighted-supersat}: once the weighted edge density exceeds
$1-\frac1r$ by a fixed amount, the weighted number of copies of $F$ is bounded
away from zero.

\emph{The stability.} If $\Phi_G(\bm{p})$ is within $\delta$ of $1-\frac1r$,
then the same blow-up is within $o(N^{2})$ edges of being extremal, so the
Erd\H{o}s--Simonovits stability theorem partitions the vertex set of the blow-up $B$ into $r$ nearly equal
parts missing few cross-pairs. 
This partition is defined on $B$, and it has to be
transported back to $G$. We do so by rounding at random: independently for each vertex $v$, we 
assign $v$ to the part $U_i$ with probability equal to 
the fraction of its cluster that lies in $B_i$. 
The resulting part masses $p(U_i):=\sum_{u\in U_i} p_u$ and the total 
missing cross-weight $M$ are then sums of independent contributions whose
variances are controlled by $\lVert \bm{p}\rVert_\infty$, so Chebyshev's and
Markov's inequalities furnish a single outcome that is good simultaneously for
all $r+1$ quantities. This is Theorem \ref{thm:weighted-stability}: a partition
$V(G)=U_1\cup \cdots \cup U_r$ with $\lvert p(U_i)-\frac1r\rvert \le \varepsilon$
for every $i$ and $M\le \varepsilon$.

\emph{The exact bound.} Near such a partition the error term disappears, and 
this is where color-criticality is used. The sampling lemma (Lemma \ref{lem:weighted-sampling}) says that from any $r$
sets of weight bounded below, one lying in each part, one can select $t$ vertices from each so that vertices selected from different sets are adjacent. 
Applying this to 
the neighborhoods of a single vertex, and to the common neighborhoods of the two
ends of an edge, would produce a copy of $K_r^+[t]$, which is impossible.
Consequently every vertex misses a fixed fraction of the weight of some part,
and all but a small fraction of the total weight is carried by $r$ independent
sets. The Cauchy--Schwarz inequality applied to their masses  gives
$\Phi_G(\bm{p})\le 1-\frac1r$ with no error term
(Lemma \ref{lem:near-turan-weighted}). Theorem \ref{thm-second-key} follows by
feeding the partition produced by the stability theorem into this local
estimate.

\smallskip
\noindent
{\bf The entropy method.}
The bridge of Section \ref{subsec:bridge} is an entropy argument, and entropy \cite{Shannon1948} 
has become a powerful tool in extremal combinatorics. 
Early examples include the work of Chung, Graham, Frankl and Shearer \cite{CGFS1986}
on triangle-intersecting families, 
where Shearer's inequality appeared;   Radhakrishnan's proof of Bregman's theorem \cite{Rad1997};  
the theorem of Friedgut and Kahn \cite{FK1998}
on counting copies of one hypergraph in another;  and Kahn's entropy approach \cite{Kahn2002} for counting independent sets and antichains. 
The entropy method has grown quickly in the last few years;
see Szegedy \cite{Szegedy2014} on the Sidorenko conjecture, 
Gilmer \cite{Gil2022} on Frankl's union-closed sets conjecture, 
Chao and Yu  \cite{CY2024} on the Kruskal--Katona-type problem,  
Gowers, Green, Manners and Tao
\cite{GGMT2025} on the Marton conjecture, 
Green, Manners and Tao
\cite{GMT2025} on sumset problems, where an entropic doubling constant replaces the combinatorial one and turns
out to behave better under homomorphisms.  
The entropy method also gives clean reformulations of arguments first found by other means: Tao \cite{Tao2020} recast the improved sunflower bound of Alweiss, Lovett, Wu and Zhang \cite{ALWZ2021} in these terms. 

\smallskip 
Closest to the present paper is the work of Chao and Yu \cite{CY2026}, who brought entropy arguments into spectral Tur\'an theory: they identified $\lambda (G)$ with the largest value of $2^{H(Y\mid X)}$ over
symmetric pairs supported on the edges of $G$, and used this to bound $\lambda^{\ell}(G)$ by the number of homomorphisms of an arbitrary tree on $\ell$ vertices into a $K_{r+1}$-free graph. 
 In addition, Liu, Sun, Wang and Wu \cite{LSWW2026} built a
Markov chain from a Perron vector and combined it with the weighted local inequalities of Liu and Ning
\cite{LiuNing2026,LiuNing2025} to obtain an edge-local version  of \eqref{eq:nikiforov-walk}, confirming a conjecture in \cite{KKP2025}. 
All these results localize the clique number of $G$ and remain restricted to the clique-free setting. 
Our result removes the clique from the hypothesis altogether: the clique number of $G$ is replaced by $\chi (F)-1$ for a
color-critical forbidden graph $F$, even though an $F$-free graph may contain cliques larger than $K_{r+1}$.
The price is the threshold $\lambda_0(F)$, which does not appear in the local inequalities.  Combining the two directions is an interesting problem for future work. 

\medskip
\noindent
{\bf Applications.}
The proof of Theorem \ref{thm:large-lambda} has three modular ingredients: the entropy-Perron bridge, the weighted Tur\'{a}n theorem it is fed with, and a flatness estimate for the Perron vector. 
Replacing the ingredients yields further applications, which we present in Section \ref{sec:applications}. First, walks may be replaced by the homomorphism counts of an arbitrary unbalanced tree $T$ on $\ell$ vertices; here $\chi (F)\ge 3$ already suffices, and the threshold is explicit 
(Theorem \ref{thm:spectral-tree}).
Second, the weighted supersaturation theorem, fed into the bridge in place of the exact one, shows that once $\lambda^{\ell}(G)$ exceeds the right-hand side of \eqref{eq:nikiforov-walk} by $\varepsilon w_\ell (G)$, the graph $G$ contains 
$\Omega (\lambda(G)^{f})$ copies of $F$, 
where $f=v(F)$; 
for this application $F$ need not be color-critical (Theorem \ref{thm:spectral-supersat}). 
Third, the weighted stability theorem alone yields the vertex-spectral stability theorem (Proposition
\ref{prop:vertex-stability}) and, in Appendix
\ref{sec:appendix-stability}, its edge-spectral counterpart; neither deduction uses the argument on eigenvalues.
Fourth, reading the bridge backwards turns Theorem \ref{thm-second-key} into an entropic Tur\'{a}n theorem for color-critical graphs, together with its equality case
(Theorem \ref{thm:entropic-turan-critical}). 
Fifth, the same route attaches to each bipartite graph $H$ a walk-spectral inequality, which we term the spectral Sidorenko inequality, that holds for every graph $G$ precisely when $H$ satisfies the traditional  Sidorenko inequality (Theorem \ref{thm:walk-sidorenko}).

\medskip 
\noindent
{\bf Organization.}
In Section \ref{sec:preliminaries}, we record the graph  removal lemma and the stability theorem, and collect the facts about entropy and the probabilistic bounds. 
In Section \ref{sec:entropy}, we prove the entropy-Perron bridge, and show the differences from those in \cite{CY2026, LSWW2026}.   
Section \ref{sec:main-proof} contains the two Perron-vector estimates, a rigidity lemma for the equality case, and the proof of Theorem \ref{thm:large-lambda}, assuming Theorem \ref{thm-second-key}.  
The next two sections concern the weighted Tur\'{a}n problems:  in Section \ref{sec:weighted},  we
establish the weighted supersaturation and weighted  stability, and in Section \ref{sec-prove-exact}, we combine them with a sampling argument to remove the error term and to prove the exact bound in Theorem \ref{thm-second-key}. 
Section \ref{sec:applications} contains the five applications mentioned above. 
Section \ref{sec:remarks} places our results in the setting of a general limit-superior problem and discusses the case of balanced trees.
Finally, Appendix \ref{sec:appendix-stability} deduces the edge-spectral 
stability theorem via the weighted stability theorem.

\medskip 
\noindent
{\bf Notation.} 
 We write $A=(a_{uv})$ for the adjacency
matrix of $G$, $\one$ for the all-ones vector, and
$\tr (A)$ for the trace of $A$. The degree of a vertex $v$ is $d_v$, and $G[S]$ is the subgraph induced on a set $S\subseteq V(G)$. We write $\chi (G)$ for the chromatic number and $\omega (G)$ for the clique number. 
The path and the cycle on $\ell$ vertices are $P_\ell$ and $C_\ell$, and the star with $d$ leaves is $K_{1,d}$. 
Let $I_t$ be the graph on $t$ vertices with no edges, and let $T_{n,r}$ be the $r$-partite Tur\'an graph on $n$ vertices. The join of two graphs is written $G\vee H$. For graphs $H$ and $G$, we write $\hom (H,G)$ for the number of homomorphisms from $H$ to $G$, and $N_H(G)$ for the number of  copies of $H$ in $G$. We write $\ex(n,F)$ for the maximum number of edges in an $n$-vertex $F$-free graph. 
A vector $\bm{p}=(p_v)_{v\in V(G)}$ is a
\emph{probability vector} if its entries are
nonnegative and add up to one. For such a vector and a set $S\subseteq V(G)$, we write
$p(S):=\sum_{v\in S}p_v$ and
$\pmax:=\max_{v\in V(G)}p_v$, and we denote 
\[
  \PhiG(\bm{p}):=2\!\!\sum_{uv\in E(G)}\!\!p_up_v .
\]
Let $\bm{x}=(x_v)_{v\in V(G)}$ be a
nonnegative unit eigenvector of $A(G)$ corresponding to $\lambda (G)$, and let $\pi$ be the probability vector $\pi_v=x_v^2$ attached to it. We write $\supp (\cdot)$ for the support of a random variable, $H(\cdot )$ for the Shannon entropy, and
$\log =\log_2$. We write $|G|:=v(G)$, $\lVert \cdot \rVert$ for the Euclidean norm, and
$\lambda^{\ell}(G):= (\lambda (G) )^{\ell}$
for every real exponent $\ell$.
 In asymptotic statements, the error term $o(1)$ refers to the limit as the relevant size parameter tends to infinity, with $F$ and $\ell$ fixed.

\section{Preliminaries} 

\label{sec:preliminaries}

\subsection{Tools from extremal graph theory} 

We first record the quantitative form of the Erd\H{o}s--Stone--Simonovits theorem.

\begin{theorem}[Erd\H{o}s--Stone--Simonovits
\cite{ES1946,ES1966}]
\label{thm:ess}
Let $F$ be a graph with $\chi (F)=r+1\ge 2$. For every $\varepsilon >0$, 
there is $N_0=N_0(F,\varepsilon )$ such that 
every $F$-free graph $G$ 
on $N \ge N_0$ vertices satisfies
$$e(G)\le \Bigl(1-\frac1r+\varepsilon \Bigr)\frac{N^2}{2}.$$
\end{theorem}

We use the graph removal lemma, which says roughly that 
every $N$-vertex graph with $o(N^f)$ copies of $F$ can be made $F$-free by removing $o(N^2)$ edges. This is a direct consequence of the celebrated Szemer\'{e}di regularity lemma; 
 see \cite{Fox2011} for a proof and for quantitative bounds.

\begin{theorem}[Graph removal lemma]\label{thm:removal}
Let $F$ be a fixed graph with $f=v(F)$. For every
$\varepsilon >0$, there is $\nu =\nu (F,\varepsilon )>0$
such that every graph $G$ on $N$ vertices containing at most $\nu N^{f}$ copies of $F$ can be made $F$-free by deleting at most $\varepsilon N^{2}$ edges.
\end{theorem}

We also use the Erd\H{o}s--Simonovits stability theorem \cite{Sim1968}
 in the following quantitative form.

\begin{theorem}[Erd\H{o}s--Simonovits stability]
\label{thm:stability}
Let $F$ be a fixed graph with $\chi (F)=r+1\ge 2$. For
every $\varepsilon >0$, there are $\theta =\theta
(F,\varepsilon )>0$ and $N_1=N_1(F,\varepsilon )$ such
that every $F$-free graph $G$ on $N\ge N_1$ vertices with 
$$e(G)\ge \Big(1-\frac1r-\theta \Big)\frac{N^{2}}{2}$$ 
admits a partition $V(G)=B_1\cup \cdots \cup B_r$ with
$\bigl\lvert |B_i|-\frac Nr\bigr\rvert \le \varepsilon N$
for every $i\in [r]$, and with at most $\varepsilon N^{2}$
pairs that lie in different parts and are non-adjacent
in $G$.
\end{theorem}

Simonovits \cite{Sim1968} determined the
exact Tur\'{a}n number for color-critical 
graphs.

\begin{theorem}[Simonovits \cite{Sim1968}]
\label{thm:simonovits}
Let $F$ be a color-critical graph with
$\chi (F)=r+1\ge 3$. There is a constant
$n_0=n_0(F)$ such that if $G$ is an $F$-free graph on $n\ge n_0$ vertices, then 
\[
  e(G) \ \le \ e(T_{n,r}),
\]
with equality if and only if $G=T_{n,r}$.
\end{theorem}

We next record the fixed template supplied by a
critical edge. For integers $r\ge 2$ and $t\ge 2$,
let $K_r[t]$ denote the complete $r$-partite graph
with $t$ vertices in every part, and let
$K_r^+[t]$ be the graph obtained from $K_r[t]$ by
adding one edge inside the first part.

\begin{lemma}\label{lem:critical-template}
Let $F$ be a color-critical graph with
$\chi (F)=r+1$, and put $t:=v(F)$. Then
$F\subseteq K_r^+[t]$. Consequently, every $F$-free
graph is $K_r^+[t]$-free.
\end{lemma}

\begin{proof}
Choose an edge $uv\in E(F)$ with
$\chi (F-uv)=r$ and fix a proper $r$-coloring of
$F-uv$. The vertices $u$ and $v$ get the same
color, since otherwise the coloring would be
proper for $F$ as well. Map the color classes to
the $r$ parts of $K_r[t]$, sending the class of
$u$ to the first part and placing $u$ and $v$ at
the two ends of the additional edge. This is
possible because every class has at most
$t=v(F)$ vertices. Thus, 
the resulting map embeds $F$ into $K_r^+[t]$, as needed. 
\end{proof}

We also need the number of homomorphisms
from a tree into a star.

\begin{lemma}\label{lem:hom-star}
Let $T$ be a tree on $\ell \ge 1$ vertices, and
let $X$ and $Y$ be the two sides of its
bipartition, with $|X|=c$ and $|Y|=\ell -c$. Then
for every integer $d\ge 1$,
\[
  \hom \bigl(T,K_{1,d}\bigr)=d^{\,c}+d^{\,\ell -c}.
\]
\end{lemma}

\begin{proof}
Let $(\{z\},L)$ be the bipartition of $K_{1,d}$, so that $z$ is its
center and $|L|=d$. The preimages of $z$ and of $L$ under a
homomorphism $f\colon T\to K_{1,d}$ are independent sets partitioning
$V(T)$, so $f^{-1}(z)\in \{X,Y\}$ because $T$ is connected. Conversely, the vertices outside $f^{-1}(z)$ may be sent to leaves arbitrarily, since every edge of $T$ has exactly one end in $f^{-1}(z)$. 
This yields $d^{\,\ell -c}$ homomorphisms with $f^{-1}(z)=X$,  and $d^{\,c}$ homomorphisms with
$f^{-1}(z)=Y$, and the two families are disjoint.
\end{proof}

\subsection{Shannon entropy and Markov chains}

\label{subsec:entropy-prelim}

Shannon \cite{Shannon1948} introduced entropy to measure how much information a random variable carries.  
Let $X$ be a discrete random variable 
with values in a finite
set. We write $p(x):=\Pr (X=x)$ and $\log (\cdot )= \log_2(\cdot)$ for short. The
Shannon entropy of $X$ is
\[
  H(X):=-\sum_{x\in \supp (X)} p(x) \cdot \log p(x),
\]
where $\supp (X)$ is the support of $X$, i.e., the set of all $x$ such that $p(x)>0$. 
The entropy depends only on the distribution of
$X$, so we also write $H( p )$ for the right-hand side
above. In particular, $H(\pi)=-\sum_v\pi_v\log \pi_v$
for a probability vector $\pi$ on $V(G)$.

For the joint random variable $Z=(X,Y)$,  we write 
$p(x,y)= \Pr(X=x,Y=y)$ and  
\begin{align*}
  H(X,Y):= \sum_{(x,y)} - p(x,y) \cdot \log p(x,y) ,
     \end{align*}
     where the sum runs over the pairs $(x,y)$ with $ p (x,y)>0$. 
If $X$ and $Y$ are independent, then $H(X,Y)=H(X) + H(Y)$. 
For any sequence of random variables $X_1,X_2,\ldots ,X_{\ell}$, we write $H(X_1,X_2,\ldots ,X_{\ell})$ for the entropy of the random tuple $(X_1,X_2,\ldots ,X_{\ell})$.

For random variables $X$ and $Y$, we write
$H(Y \,|\, X=x)$ for the entropy of the conditional
distribution of $Y$ given $X=x$, and we define the
{\it conditional entropy} of $Y$ given $X$ by 
\begin{align*}
  H(Y\mid X):= \sum_{x} p(x) \cdot H(Y\mid X=x)  
  = - \sum_{x,\,y} p(x,y) 
     \log \frac{p(x,y)}{p(x)}, 
\end{align*}
where both sums run over the values with $p(x)>0$. 
By definition, $H(Y | X) = H(X,Y) - H(X).$ 

\smallskip 
A \emph{Markov chain} on a set $V$ with initial
distribution $\pi$ and transition matrix
$P=(P_{uv})_{u,v\in V}$ is a sequence $(X_t)_{t\ge 1}$ of
$V$-valued random variables such that $X_1$ has
distribution $\pi$ and
\[
  \Pr (X_{t+1}=v\mid X_1=u_1,\ldots ,X_{t-1}=u_{t-1},
       X_t=u)=P_{uv}
\]
whenever the conditioning event has positive
probability. Here $P$ is nonnegative and each row of
$P$ sums to one. A probability vector $\pi$ on $V$ is
\emph{stationary} for $P$ if
$\sum_{u\in V}\pi_uP_{uv}=\pi_v$ for every $v\in V$,
and the pair $(P,\pi )$ satisfies \emph{detailed
balance}, or is \emph{reversible}, if
$\pi_uP_{uv}=\pi_vP_{vu}$ for all $u,v\in V$.
Detailed balance implies stationarity: summing over $u$
and using that row $v$ of $P$ sums to one gives
$\sum_u\pi_uP_{uv}=\pi_v\sum_uP_{vu}=\pi_v$.
We call the chain \emph{stationary} when $\pi$ is
stationary for $P$ and $X_1$ has distribution $\pi$. In
that case
\[
  \Pr (X_t=u_0,X_{t+1}=u_1,\ldots ,X_{t+k}=u_k)
  =\pi_{u_0}P_{u_0u_1}\cdots P_{u_{k-1}u_k},
\]
which does not depend on $t$. In particular, every
$X_t$ has distribution $\pi$, that is, $\mathrm{Pr}(X_t=u) = \pi_u$,  and $(X_t,X_{t+1})$ has
the same distribution for every $t$, namely 
$\mathrm{Pr}\big( (X_t,X_{t+1})=(u,v) \big) = \pi_u P_{uv}$; see \cite{CT2006}
for background. We use the following standard facts;
we refer the reader to \cite{AS2016,ZhaoYufei}. 

\begin{enumerate}
\item[E1] (Uniform bound.) For any random variable $X$, we have 
  $$ H(X) \ \le \  \log |\supp (X)|, $$ 
  with equality if and
  only if $X$ is uniformly distributed on $\supp (X)$.
  
\item[E2] (Chain rule.) For every $\ell \ge 2$,
$
    H(X_1,\ldots ,X_\ell)=H(X_1)
      + \sum_{t=2}^{\ell}H(X_t\mid X_1,\ldots ,X_{t-1})$.
  
  \item[E3] 
  (Subadditivity.) For every $\ell \ge 2$, we have 
  $H(X_1,\ldots ,X_{\ell}) \le \sum_{i=1}^{\ell} H(X_i)$.
  
  \item[E4] (Conditioning does not increase entropy.) 
  $H(X\,|\,Y)\le H(X)$ and $H(X\,|\,Y,Z) \le H(X\,|\,Y)$. 
  
\item[E5] (Markov property.) If $(X_t)_{t\ge 1}$ is
  a Markov chain, then
  \[
    H(X_{t+1}\mid X_1,\ldots ,X_t)
      =H(X_{t+1}\mid X_t)
    \quad \text{for every }t\ge 1.
  \]
 \end{enumerate}

 For two probability vectors $q$ and $s$
on the same finite set, with $s(a)>0$ whenever
$q(a)>0$, the {\it relative entropy} of $q$ with respect to
$s$ is defined by 
\[
  D(q\,\Vert \,s):=\sum_a q(a)\log \frac{q(a)}{s(a)}.
\]
We use the nonnegativity of $D(q\,\Vert\,s)$ several times. 

 \begin{lemma}[Gibbs' inequality]  \label{lem-Gibb}
 We have $D(q\,\Vert \,s)\ge 0$,
  with equality if and only if $q=s$.
 \end{lemma}
 
For completeness,  we recall the proof of Lemma \ref{lem-Gibb}.

\begin{proof}
By the concavity of the logarithm and Jensen's inequality, we have 
\[
  -D(q\,\Vert \,s)
  =\sum_{a\in \supp (q)}q(a)
     \log \frac{s(a)}{q(a)}
  \le \log \!\sum_{a\in \supp (q)}\!\! s(a)
  \le 0 ,
\]
where equality in the first step forces $s(a)/q(a)$ to be
constant on $\supp (q)$, and equality in the second
forces $s(a)=0$ for $a\notin \supp (q)$; together these
give $s=q$. 
\end{proof}

The next lemma  says that the entropy of an
initial segment of a stationary chain grows linearly,
with slope equal to the one-step conditional entropy. 

\begin{lemma}[Entropy of a stationary chain]
\label{lem:chain-entropy}
Let $(X_t)_{t\ge 1}$ be a stationary Markov chain on a
finite state space, with initial distribution $\pi$.
Then for every integer $\ell \ge 1$,
\[
  H(X_1,\ldots ,X_\ell)
 \ = \ H(\pi)+(\ell -1)H(X_2\mid X_1).
\]
\end{lemma}

\begin{proof}
Since $X_1$ has distribution $\pi$, we have $H(X_1)=H(\pi)$,
which settles the case $\ell =1$. Let $\ell \ge 2$.
Stationarity implies that $(X_t,X_{t+1})$ has the same
distribution for every $t\ge 1$, and hence
$$H(X_{t+1}\mid X_t)=H(X_2\mid X_1).$$
 Combining this
with (E5), we see that every term of (E2) with $t\ge 2$ equals 
\[
  H(X_t\mid X_1,\ldots ,X_{t-1})
  =H(X_t\mid X_{t-1})
  =H(X_2\mid X_1),
\]
and there are $\ell -1$ such terms. Thus, simplifying the chain rule (E2) gives the desired identity. 
\end{proof}

We also use the following two standard probabilistic bounds; see
\cite{AS2016, ZhaoYufei}. 

\begin{lemma}[Markov's inequality]
\label{lem:markov}
Let $Y$ be a nonnegative random variable with finite
mean. Then for every $a>0$,
\[
  \Pr (Y\ge a)\le \frac{\EE [Y] }{a}.
\]
\end{lemma}

\begin{lemma}[Chebyshev's inequality]
\label{lem:chebyshev}
Let $Z$ be a random variable with finite variance.
Then for every $a>0$,
\[
  \Pr \bigl( \big\lvert Z-\EE [Z] \big\rvert \ge a\bigr)
  \le \frac{\Var (Z)}{a^2}.
\]
\end{lemma}

\section{The entropy-Perron bridge}
\label{sec:entropy}

The entropy method suits our problem for a simple reason. To bound a power of
$\lambda (G)$ by a count, sample a random object of the kind being counted: by
the uniform bound, its entropy is at most the logarithm of the count. If the
object is built one vertex at a time, the chain rule splits that entropy into
one local term per step. An $\ell$-walk is built by $\ell -1$ steps, each moving to a neighbor, 
and each step, as we now show, contributes exactly $\log \lambda$. 
This is why the entropy method is particularly effective  for counting homomorphisms; see
\cite{Szegedy2014, KLL2016, CKLL2018} for its use on Sidorenko's conjecture. 

\smallskip 
Before carrying out the above idea, we record the entropic Tur\'{a}n theorem of Chao and Yu \cite{CY2026}, which is the starting point of the present approach. 
In \cite[Theorem 1.2]{CY2026},  
Chao and Yu  proved the entropic Tur\'an theorem in the following formulation: if $G$ is a $K_{r+1}$-free graph and $(X,Y)$ is a random 
symmetric\footnote{The distribution of $(X,Y)$ and the one of $(Y,X)$ are the same.} pair of vertices of $V(G)$ with $XY\in E(G)$, 
then 
\begin{equation}
\label{eq-Chao-Yu-entropic}
  H(X,Y) \ \le \  2H(X)+\log \Big(1-\frac1r \Big).
  \end{equation}
 Moreover, they showed that this entropic version is equivalent to the density version \cite[Corollary 5.7]{CY2026}. 
In addition, they proved two entropic identities for such pair: first,  the largest value of $2^{H(X,Y)-2H(X)}$ is the blow-up density of $G$,
which is twice its Lagrangian; second, the largest value of $2^{H(Y\mid X)}$ is the spectral radius
$\lambda (G)$. 
Both of them are instances of a general result established for the $p$-spectral radius of a
$k$-graph in \cite[Proposition 5.4]{CY2026}. 
The second one turns \eqref{eq-Chao-Yu-entropic} into a
spectral statement: sampling $(X,Y)$ along the edges
of an $\ell$-vertex tree $T$ gives
$\lambda^{\ell}(G)\le (1-\frac1r)\hom (T,G)$ for every
$K_{r+1}$-free graph $G$. This recovers the bounds of
Wilf \cite{Wilf1986} and Nikiforov
\cite{Nik2002, Nik2006}. Combining the identity
for general $p$ with \eqref{eq-Chao-Yu-entropic} then 
 recovers the bound of Kang and Nikiforov
\cite{KN2014}.

\smallskip  
We now return to the plan described above, and use the entropy method to bound the $\ell$-th power of $\lambda (G)$ in terms of the number of $\ell$-walks. 
Recall that $\PhiG(\pi) = \sum_{u,v\in V(G)} a_{uv}\pi_u \pi_v$.

\begin{theorem}[Entropy-Perron bridge]
\label{thm:entropy-bridge}
Let $G$ be a connected graph with at least one
edge, let $\lambda =\lambda (G)$, let $\bm{x}$ be
its positive unit Perron vector and put
$\pi_v=x_v^2$. Let $(X_t)_{t\ge 1}$ be the Markov
chain on $V(G)$ with initial distribution $\pi$ and
transition matrix $P_{uv}=a_{uv}x_v/(\lambda x_u)$,
let $\mu_\ell$ be the distribution of
$(X_1, \ldots ,X_\ell )$, and let $\nu_\ell$ be the
uniform distribution on the set of walks on $\ell$
vertices of $G$. Define probability distributions on
the ordered pairs $(u,v)$ with $u,v\in V(G)$ by
\[ Q_{uv}=\frac{a_{uv}x_ux_v}{\lambda},
  \qquad
  R_{uv}=\frac{a_{uv}\pi_u\pi_v}{\Phi_G(\pi )}.
\]
Then for every $\ell \ge 1$, 
\begin{equation}\label{eq:bridge-exact}
  \log w_\ell (G)=\ell \log \lambda -\log \Phi_G (\pi )
  +D(Q\, \Vert \, R)+D(\mu_\ell \, \Vert \, \nu_\ell ).
\end{equation}
In particular, dropping the two relative entropies gives Theorem
\ref{thm:bridge-intro}:
$$\lambda^\ell (G) \ \le \  
\Phi_G(\pi )\cdot w_\ell (G), $$
with equality if and only if $Q=R$ and $\mu_\ell =\nu_\ell$. 
\end{theorem}

\begin{proof}
Since $G$ is connected and has an edge, we
have $\lambda >0$, $\Phi_G(\pi) >0$ and $x_v>0$ for every $v$. 
Each row sum of $P$ is equal to $1$, since 
$ \sum_{v} P_{uv} = 
\frac{1}{\lambda x_u} \sum_{v} a_{uv}x_v= 1.$ 
 Moreover,
\[
  \pi_uP_{uv}=\frac{a_{uv}x_ux_v}{\lambda}=Q_{uv},
\]
which is symmetric in $u$ and $v$. Hence $(P,\pi )$
satisfies detailed balance, and $\pi$ is stationary by
the discussion in Subsection
\ref{subsec:entropy-prelim}. Consequently,
 $Q$ is the distribution of $(X_t,X_{t+1})$ for every $t\ge 1$; it is indeed a probability distribution because
$\sum_{u,v}Q_{uv}=\frac1\lambda \bm{x}^{\mathsf T}A(G)
\bm{x}=1$, and both of its marginals equal $\pi$. The
same computation with $\Phi_G(\pi )=\sum_{u,v}a_{uv}
\pi_u\pi_v$ shows that $R$ is a probability
distribution. Both $Q$ and $R$ are supported on the
ordered edge set
$$E^{*}(G):=\{(u,v):uv\in E(G)\},$$
 so $D(Q\, \Vert\, R)$ is
finite. Finally, a sequence $(v_1,\ldots ,v_\ell )$ has
$\mu_\ell (v_1,\ldots ,v_\ell )>0$ if and only if
consecutive entries are adjacent, that is, if and only
if the sequence is a walk on $\ell$ vertices of $G$.
Hence $\mu_\ell$ and $\nu_\ell$ have the same support,
and $D(\mu_\ell \,\Vert \, \nu_\ell )$ is finite as well. 

We now compute the two ingredients of
\eqref{eq:bridge-exact}. Since $X_1$ has distribution
$\pi$, we have $H(X_1)=H(\pi )$. 
Moreover, for every $t\ge 1$, it follows that  
\begin{align*}
  H(X_{t+1}\mid X_t)
  &=-\sum_{(u,v)\in  E^* (G)}Q_{uv}\log P_{uv}\\
  &=-\sum_{(u,v)\in E^* (G)}Q_{uv}
       \bigl(\log x_v-\log\lambda-\log x_u\bigr)\\
  &=\log\lambda
    +\sum_{u\in V(G)}\pi_u\log x_u
    -\sum_{v\in V(G)} \pi_v\log x_v\\
  &=\log \lambda.
\end{align*} 
The chain rule for entropy, in the form of Lemma \ref{lem:chain-entropy}, now gives 
\begin{equation}  \label{eq-chain-and-spectral}
  H(\mu_{\ell}) = H(X_1,\ldots,X_\ell)
  =H(\pi)+(\ell-1)\log\lambda.
\end{equation}

\smallskip
Let $S$ be the set of walks on $\ell$ vertices of $G$. 
Since $\nu_{\ell}$ is uniformly distributed on $S$ and $\mu_{\ell}$ is also supported on $S$, we have 
$D(\mu_{\ell} \,\Vert\, \nu_{\ell} )=\log |S|-H(\mu_{\ell} )$. Since $|S|=w_\ell (G)$  and using  \eqref{eq-chain-and-spectral},
\begin{equation}\label{eq:step3}
  \log w_\ell (G)
  =H(\mu_\ell )+D(\mu_\ell \, \Vert \, \nu_\ell )
  =H(\pi )+(\ell -1)\log \lambda
   +D(\mu_\ell \, \Vert\,  \nu_\ell ).
\end{equation}

\smallskip
We next compute the relative entropy of $Q$ with respect to $R$. Using $Q_{uv}=a_{uv}x_ux_v/\lambda$,
$R_{uv}=a_{uv}\pi_u\pi_v/\PhiG(\pi )$ and $\pi_v=x_v^2$,
we obtain 
\begin{align*}
  D(Q \, \Vert \, R)
  &=\sum_{(u,v)\in E^* (G)}Q_{uv}
      \log\frac{Q_{uv}}{R_{uv}}\\
  &=\sum_{(u,v)\in E^* (G)}Q_{uv}
      \log\frac{\PhiG(\pi)}{\lambda x_ux_v}\\
  &=\log\frac{\PhiG(\pi)}{\lambda}
    -\sum_{u\in V(G)}\pi_u\log x_u
    -\sum_{v\in V(G)}\pi_v\log x_v\\
  &=\log\frac{\PhiG(\pi)}{\lambda}+H(\pi),
\end{align*}
where the last equality follows from
\[
  H(\pi)=-\sum_{v\in V(G)}x_v^2\log(x_v^2)
  =-2\sum_{v\in V(G)}\pi_v\log x_v.
\]
 Hence, we get 
\begin{equation}\label{eq:step4}
  H(\pi )=D(Q \, \Vert \, R)+\log \lambda -\log \Phi_G(\pi) .
\end{equation}
Substituting \eqref{eq:step4} into \eqref{eq:step3}
gives
\[
  \log w_\ell (G)=\ell \log \lambda -\log \Phi_G(\pi)
  +D(Q\, \Vert \, R)+D(\mu_\ell \, \Vert \, \nu_\ell ),
\]
which is \eqref{eq:bridge-exact}. Both relative
entropies are nonnegative by Gibbs' inequality, so
$\log w_\ell (G)\ge \ell \log \lambda -\log \Phi_G(\pi)$,
that is, $\lambda^\ell (G)\le \Phi_G(\pi) w_\ell (G)$.
Equality holds in Gibbs' inequality precisely when the two distributions coincide, so equality holds here if
and only if $Q=R$ and $\mu_\ell =\nu_\ell$.
\end{proof}

\begin{remark}[Comparison with the entropic Tur\'an theorem]
\label{rem:compare-CY}
Theorem \ref{thm:entropy-bridge} follows a route different from
\eqref{eq-Chao-Yu-entropic}. The pair $(X,Y)$ used here is not arbitrary: it is
the edge distribution of the Markov chain built from
$P_{uv}=a_{uv}x_v/(\lambda x_u)$, and it is exactly the pair attaining
$2^{H(Y\mid X)}=\lambda (G)$, that is, the extremal pair produced by Chao and Yu 
\cite[Proposition 5.4]{CY2026} when $k=p=2$.
For this pair $(X,Y)$, we do not bound
$H(X,Y)-2H(X)$ as in (\ref{eq-Chao-Yu-entropic}). 
Instead, we compare it with
the Lagrangian at the single probability vector $\pi_v=x_v^2$,
and Gibbs' inequality (Lemma \ref{lem-Gibb}) gives
$\log \Phi_G(\pi) - \log \lambda + H(\pi ) \ge 0. $ 
Together with $\log \lambda = H(Y|X)$ and
$H(\pi) = H(X)$, this yields
$$  {H(X,Y)-2H(X)}\le \log \PhiG(\pi).$$
We must stop here: the entropic Tur\'an theorem needs $G$ to be $K_{r+1}$-free,
whereas an $F$-free graph may contain much larger cliques. The bound on
$\PhiG(\pi )$ instead comes from Theorem \ref{thm-second-key}, whose proof is
structural and probabilistic rather than entropic, and which applies since Lemma \ref{lem:perron-flatness} below makes $\pi$ flat. Entropy is thus a
bridge, not the whole argument.  
The comparison made here does have a purely entropic application:
carried out for an arbitrary symmetric pair rather than for the Perron pair, it
extends Chao--Yu's bound \eqref{eq-Chao-Yu-entropic} to color-critical forbidden graphs; see
Section \ref{subsec:entropic-turan}. 
\end{remark}

\begin{remark}[Comparison with the weighted graph route]  \label{rem:not-replaceable}
The entropy-Perron bridge has a natural competitor. 
Recently, Liu, Sun, Wang and Wu
\cite{LSWW2026} proved a Markov chain estimate, which attaches to $G$ a weighted graph $W_k=(G,\psi_k)$ with the edge-weight $\psi_k(uv) :=\sqrt{(w_{k+1}(u)+w_{k+1}(v))/2}$, where $w_j(v)$ denotes the number of 
walks on $j$ vertices starting at $v$, and shows that 
$\lambda (W_k)\ge \lambda^{1+k/2}(G)$.  
Taking $k=\ell -2$, let $\bm{y}$ be a nonnegative unit Perron vector of $W_{\ell -2}$. It follows that 
\begin{equation}\label{eq:W-route}
  \lambda^{\ell}(G)\ \le\ \lambda^{2}(W_{\ell -2}) 
  \ \le\ \Phi_G(\bm{p}_{_{W}})\,w_\ell (G),\quad \text{where}~~ (\bm{p}_{_W})_v:=y_v^{2} 
\end{equation}
and the second inequality follows by the Cauchy--Schwarz inequality and $2\sum_{uv}\psi_{\ell -2}(uv)^{2}=w_\ell (G)$. 
Observe that (\ref{eq:W-route}) has the shape of Theorem \ref{thm:entropy-bridge} with $\pi$
replaced by $\bm{p}_{_W}$. 
One would like to apply Theorem \ref{thm-second-key} to
$\Phi_G(\bm{p}_{_W})$, but the attempt fails at the flatness condition $\lVert \bm{p}_{_W}\rVert_\infty \le \eta (F)$, which Lemma \ref{lem:perron-flatness} does not supply.  
We are unable to show the flatness of $\bm{p}_{_W}$, so we retain Theorem \ref{thm:entropy-bridge}.
 \end{remark}

\section{Flatness lemma and the proof of
Theorem \ref{thm:large-lambda}} 

\label{sec:main-proof}

In this section, we prove Theorem \ref{thm:large-lambda}.
The proof splits into two cases according to 
$\chi (F)$: Theorem \ref{thm:large-lambda-ge4} treats
$\chi (F)\ge 4$, and Theorem \ref{thm:large-lambda-chi=3}
treats $\chi (F)=3$, and together they give Theorem
\ref{thm:large-lambda}. Both proofs combine the same three
ingredients: the entropy bridge of Section \ref{sec:entropy}, the exact weighted Tur\'{a}n theorem (Theorem
\ref{thm-second-key}) proved in Sections
\ref{sec:weighted} and \ref{sec-prove-exact} independently
of the present section, and a flatness estimate for the
Perron vector. 
Only the third ingredient changes between
the two cases, and this is where $\chi (F)=3$ is more
involved: the coefficient $1-\frac1r$ then equals
$\frac12$, so the flatness estimate of the forthcoming Lemma \ref{lem:perron-flatness} supplies no fixed gap, and we provide two substitutes, one for odd $\ell$ and one for even $\ell \ge 4$. 
We first record the flatness estimate 
\cite[Lemma 3.1]{LLZ-ESS}. 

\begin{lemma}[Li--Liu--Zhang \cite{LLZ-ESS}]
\label{lem:perron-flatness}
Let $G$ be a graph, and let $\bm{x}=(x_v)_{v\in V(G)}$ be
its unit Perron vector. If $\ell \ge 1$ is an integer and
$\lambda^\ell (G)\ge \left(\frac12+\delta \right)
w_\ell (G)$ for some $\delta \in (0,0.1)$, then
\[ 
  \max_{v\in V(G)} x_v^2 \ < \ \delta^{-8}\lambda (G)^{-1}.
\] 
\end{lemma}

The two equality analyses run along the same lines, and we isolate them once.

\begin{lemma} 
\label{lem:equality-rigidity}
Let $F$ be a color-critical graph with
$\chi (F)=r+1\ge 3$, let $\eta (F)$ be the constant of
Theorem \ref{thm-second-key}, and let $n_0(F)$ be the
threshold of Theorem \ref{thm:simonovits}. Let
$\ell \ge 1$ be an integer, and let $G$ be a connected
$F$-free graph with at least one edge, with positive unit
Perron vector $\bm{x}$ and $\pi_v=x_v^2$. If
$ \lambda (G)\ge n_0(F)$, 
 $ \max_{v}\pi_v\le \eta (F)$ and 
$  \lambda^{\ell}(G)=\bigl(1-\frac1r\bigr)w_\ell (G)$, 
then $\Phi_G(\pi )=1-\frac1r$ and $x_u=x_w$ whenever $u$
and $w$ are joined by a walk of even length. Moreover,
exactly one of the following holds:
\begin{itemize}
\item[\rm (a)] $G$ is bipartite; then $r=2$ and $G$ is a
complete bipartite graph.

\item[\rm (b)] $G$ is not bipartite; then $r\ge 3$ and
$G=T_{n,r}$ is a regular complete $r$-partite graph.
\end{itemize}
\end{lemma}

\begin{proof}
Since $\max_v\pi_v\le \eta (F)$, Theorem
\ref{thm-second-key} gives $\Phi_G(\pi )\le 1-\frac1r$. On
the other hand, substituting
$\lambda^\ell (G)=(1-\frac1r)w_\ell (G)$ into the identity 
\eqref{eq:bridge-exact} gives
\[
  \log \Phi_G(\pi )=\log \Bigl(1-\frac1r\Bigr)
  +D(Q\Vert R)+D(\mu_\ell \Vert \nu_\ell )
  \ \ge\ \log \Bigl(1-\frac1r\Bigr).
\]
Hence $\Phi_G(\pi )=1-\frac1r$ and
$D(Q\Vert R)=D(\mu_\ell \Vert \nu_\ell )=0$. The vanishing
of $D(Q\Vert R)$ says that $Q_{uv}=R_{uv}$ for every edge
$uv\in E(G)$, that is,
$  x_ux_v = {\Phi_G(\pi )}/{\lambda}$
for every $uv\in E(G)$.
Thus $x_u=x_w$ whenever $u$ and $w$ have a common
neighbor, and chaining along even walks gives the second
assertion.

Suppose first that $G$ is bipartite, with parts $A$ and
$B$. Every edge of $G$ joins $A$ to $B$, so
\[
  \Phi_G(\pi )=2\!\!\sum_{uv\in E(G)}\!\!\pi_u\pi_v
  \le 2\pi (A)\pi (B)\le \frac12 ,
\]
with equality if and only if
$\pi (A)=\pi (B)=\frac12$ and every pair in $A\times B$ is
an edge. As $\Phi_G(\pi )=1-\frac1r$, this forces $r=2$,
and then equality holds throughout, so $G$ is complete
bipartite. This is (a).

Suppose now that $G$ is not bipartite. Being connected,
$G$ contains a closed walk of odd length, say at a vertex
$z$. Given $u,w\in V(G)$, concatenating a $u$--$z$ walk,
possibly one copy of that closed walk, and a $z$--$w$ walk
produces $u$--$w$ walks of both parities. In particular,
any two vertices are joined by a walk of even length, so
$\bm{x}$ is constant and $G$ is $d$-regular with
$\lambda =d$ and
$\Phi_G(\pi )=\frac{2e(G)}{n^2}=\frac dn$. From
$\Phi_G(\pi )=1-\frac1r$, we get
$e(G)=(1-\frac1r)\frac{n^2}{2}$. Moreover, we have 
$n>d=\lambda (G)\ge n_0(F)$, so Theorem
\ref{thm:simonovits} gives
$ e(G)\le \ex (n,F)=e\bigl(T_{n,r}\bigr)
  \le \bigl(1-\frac1r \bigr)\frac{n^2}{2}$,
with equality in the last step only when $r\mid n$. Hence
$e(G)=\ex (n,F)$, and the uniqueness in Theorem
\ref{thm:simonovits} gives $G=T_{n,r}$ with $r\mid n$, a
regular complete $r$-partite graph. Since $T_{n,2}$ is
bipartite, we must have $r\ge 3$. This proves (b).
\end{proof}

\subsection{The case $\chi (F)\ge 4$}

\begin{theorem}
\label{thm:large-lambda-ge4}
Let $F$ be a color-critical graph with
$\chi (F)=r+1\ge 4$. There is a constant
$\lambda_0=\lambda_0(F)$ such that if $G$ is an $F$-free
graph with $\lambda (G)\ge \lambda_0$, then for every
integer $\ell \ge 1$,
\begin{equation}  \label{eq:main-bound-ge4}
  \lambda^\ell(G) \le
  \Big( 1- \frac{1}{r}\Big)w_\ell(G),
\end{equation}
with equality if and only if $G$ is a regular complete
$r$-partite graph, possibly together with isolated
vertices when $\ell \ge 2$. 
\end{theorem} 

\begin{proof}[{\bf Proof of Theorem
\ref{thm:large-lambda-ge4}}]
Let $\eta =\eta (F)>0$ be the constant from Theorem
\ref{thm-second-key}, let $n_0=n_0(F)$ be the threshold in
Theorem \ref{thm:simonovits}, and fix the constant 
$\delta :=\frac1{20}$. Since
$r\ge 3$, we have $1-\frac1r\ge \frac23>\frac12+\delta$,
so Lemma \ref{lem:perron-flatness} is available
throughout. Put 
$\lambda_0(F):=\max \{\delta^{-8}\eta^{-1} ,\ n_0(F)\}$.

\smallskip
Suppose that $G$ is an $F$-free graph that violates
\eqref{eq:main-bound-ge4} for some $\ell \ge 1$.  Then
\[  \lambda^{\ell}(G)
  >\Bigl(1-\frac1r\Bigr)w_\ell (G). \]
Choose a
connected component $C$ of $G$ with
$\lambda (C)=\lambda (G)$. Since
$w_\ell (C)\le w_\ell (G)$, we get 
\begin{equation}\label{eq:component-C}
  \lambda^{\ell}(C)=\lambda^{\ell}(G)
  >\Bigl(1-\frac1r\Bigr)w_\ell (G)
   \ge \Bigl(1-\frac1r\Bigr)w_\ell (C).
\end{equation}
In particular $\lambda (C)>0$, so $C$ has at least one
edge. Let $\bm{x}$ be the positive unit Perron vector of
$C$ and put $\pi_v:=x_v^2$ for every $v\in V(C)$. By
Theorem \ref{thm:entropy-bridge}, we have
$$\lambda^\ell (C)\le \Phi_C(\pi ) \cdot w_\ell (C).$$ 
Together with \eqref{eq:component-C}, this gives $\Phi_C(\pi )>1-\frac1r$. The
component $C$ is $F$-free, so the contrapositive of
Theorem \ref{thm-second-key} implies
\begin{equation}\label{eq:pi-heavy}
  \max_{v\in V(C)}\pi_v>\eta .
\end{equation}
On the other hand, \eqref{eq:component-C} shows that
$\lambda^{\ell}(C)>(\frac12+\delta )w_\ell (C)$, 
so Lemma \ref{lem:perron-flatness} gives
$$\max_{v\in V(C)}\pi_v< \delta^{-8} \lambda (C)^{-1}.$$ 
Combining this with
\eqref{eq:pi-heavy}, we find
$\lambda (G)=\lambda (C)< \delta^{-8} \eta^{-1} \le \lambda_0$.
Consequently, every $F$-free graph $G$ with
$\lambda (G)\ge \lambda_0$ must satisfy the desired bound 
\eqref{eq:main-bound-ge4} for every integer $\ell \ge 1$.

\smallskip
\noindent
{\bf Equality case.}~
Suppose that $G$ is an $F$-free graph with
$\lambda (G)\ge \lambda_0$ and
$\lambda^\ell (G)=(1-\frac1r)w_\ell (G)$. Let $C$ be a
component of $G$ with $\lambda (C)=\lambda (G)$. Then $C$
is $F$-free and $\lambda (C)\ge \lambda_0$, so the
inequality proved above applies to $C$ and gives
$\lambda^\ell (G)=\lambda^\ell (C)\le
(1-\frac1r)w_\ell (C)\le (1-\frac1r)w_\ell (G)$. Hence
$w_\ell (C)=w_\ell (G)$, and no other component carries a
walk on $\ell$ vertices. For $\ell \ge 2$ the other
components are isolated vertices, and for $\ell =1$ there
are none. Replacing $G$ by $C$ changes neither side of
\eqref{eq:main-bound-ge4}, so we may assume that $G$ is
connected; then $\bm{x}>0$ and we put $\pi_v=x_v^2$.

Since $\lambda^\ell (G)=(1-\frac1r)w_\ell (G)\ge
(\frac12+\delta )w_\ell (G)$, 
Lemma \ref{lem:perron-flatness} gives
$\max_{v\in V(G)}\pi_v<\delta^{-8}\lambda (G)^{-1}
\le \delta^{-8}\lambda_0^{-1}\le \eta$, while
$\lambda (G)\ge \lambda_0\ge n_0(F)$.
Hence, Lemma \ref{lem:equality-rigidity} applies to $G$.
Its alternative (a) requires $r=2$, so alternative (b)
holds and $G=T_{n,r}$ is a regular complete $r$-partite
graph.

Conversely, let $G_0=K_r[s]$ be the regular complete $r$-partite
graph with all $r$ parts of size $s$. Then $G_0$ is
$r$-partite, hence $G_0$ is $F$-free, and $\lambda (G_0)=(r-1)s$
and $w_\ell (G_0)=rs((r-1)s)^{\ell -1}$, so equality holds
in \eqref{eq:main-bound-ge4}. If
$s\ge \lambda_0/(r-1)$, then
$\lambda (G_0)\ge \lambda_0$, so $G_0$ satisfies the
hypothesis of the theorem. Adding isolated vertices
changes neither side of \eqref{eq:main-bound-ge4} when
$\ell \ge 2$.
\end{proof}

\subsection{The case $\chi (F)=3$}

The case $\chi (F)=3$ of Theorem \ref{thm:large-lambda}
reduces to the following result. 

\begin{theorem}
\label{thm:large-lambda-chi=3}
Let $F$ be a color-critical graph with $\chi (F)=3$, and
let $\ell \neq 2$ be a positive integer. There
is a constant $\lambda_0=\lambda_0(F)$ such that every
$F$-free graph $G$ with $\lambda (G)\ge \lambda_0$
satisfies
\begin{equation}\label{eq:main-bound-chi=3}
  \lambda^{\ell}(G)\ \le\ \frac12\,w_\ell (G).
\end{equation}
For odd $\ell$, equality holds if and only if $G$ is a
regular complete bipartite graph; for even
$\ell \ge 4$, equality holds if and only if $G$ is a
complete bipartite graph. In both cases, $G$ may in
addition contain isolated vertices when $\ell \ge 2$.
\end{theorem}

In the case $\chi (F)=3$, the flatness estimate of Lemma \ref{lem:perron-flatness} cannot be applied since there is no $\delta >0$ such that $\lambda^{\ell} (G) \ge (\frac{1}{2} + \delta)w_{\ell}(G)$. 
For odd $\ell$, we prove the following flatness estimate  
that requires no gap hypothesis; for even $\ell \ge 4$, we prove Theorem \ref{thm:large-lambda-chi=3} in a different way. 

\begin{lemma} 
\label{lem:flat-odd}
Let $G$ be a connected graph with at least one edge, let $\lambda =\lambda (G)$, and let $\bm{x}$ be its unit
Perron vector. If $\ell \ge 1$ is odd and
$\lambda^{\ell} (G)\ge c_0\,w_\ell (G)$ for some
$c_0\in (0,1]$, then
\[
  \max_{v\in V(G)}x_v^2
  \ \le \ \frac{1}{c_0\,\lambda (G)} .
\]
\end{lemma}

\begin{proof}
Let $u_1=\bm{x},u_2,\ldots ,u_n$ be an orthonormal basis
of eigenvectors of $A=A(G)$, with $Au_i=\lambda_iu_i$ and
$\lambda_1=\lambda$, and write $\one =\sum_i c_iu_i$, so
that $c_1=\langle \one ,\bm{x}\rangle
=\lVert \bm{x}\rVert_1$. Then 
\[
  w_\ell (G)=\one^{\mathsf T}A^{\ell -1}\one
  =\sum_{i=1}^{n}c_i^{2}\lambda_i^{\ell -1}.
\] 
Since $\ell -1$ is even, every term of the above sum is
nonnegative, and therefore
$w_\ell (G)\ge \lVert \bm{x}\rVert_1^{2}
\lambda^{\ell -1}$. Let $z$ be a vertex with
$x_z=\max_{v\in V(G)}x_v$. The vertices of
$\{z\}\cup N_G(z)$ are pairwise distinct, so
\[
  \lVert \bm{x}\rVert_1\ \ge\ x_z+\sum_{j\sim z}x_j
  \ =\ (1+\lambda )x_z\ \ge\ \lambda x_z .
\]
Thus, we get $w_\ell (G)\ge \lambda^{\ell +1}x_z^{2}$,
while the hypothesis reads
$w_\ell (G)\le c_0^{-1}\lambda^{\ell}$. Hence
$x_z^{2}\le c_0^{-1}\lambda^{-1}$.
\end{proof}

\begin{remark}
For even $\ell$, the argument fails, since $K_{1,d}$
satisfies $\lambda (K_{1,d})=\sqrt d$ and
$w_{2k}(K_{1,d})=2d^{k}$, so
$\lambda^{2k}=\frac12 w_{2k}$, while
$\max_v x_v^{2}=\frac12$. No bound of the
form $\max_{v}x_v^{2}=O(\lambda^{-1})$ can hold.
Thus, for even $\ell$, a gap hypothesis such as
$\lambda^{\ell}\ge (\frac12+\delta )w_\ell (G)$ is
unavoidable to guarantee $\max_{v} x_v^2 =O(\lambda^{-1})$. \end{remark}

When $G=K_{1,d}$, although the case $\ell =2k$ lies outside the range of Lemma \ref{lem:flat-odd},  
the star satisfies
$\lambda^{2k}(K_{1,d})=\frac12 w_{2k}(K_{1,d})$, so it
attains equality in \eqref{eq:main-bound-chi=3}. 
In other words, the star $K_{1,d}$ obstructs the flatness estimate but not the bound \eqref{eq:main-bound-chi=3}. 
Rather than looking for a flatness estimate, we avoid the issue and prove \eqref{eq:main-bound-chi=3} directly for every graph whose Perron vector has a large coordinate.

\begin{lemma} 
\label{lem:large-perron}
Let $0<\eta \le \frac12$ and let $G$ be a connected graph
with at least one edge and with positive unit Perron
vector $\bm{x}$. If $\max_{v\in V(G)}x_v^2\ge \eta$ and
$\lambda (G)\ge 16\eta^{-6}$, then
$\lambda^{\ell}(G) \le \frac{1}{2} w_{\ell} (G)$ for every even integer
$\ell \ge 4$. For any one such $\ell$, equality holds if
and only if $G$ is complete bipartite.
\end{lemma}

For even $\ell$, this lemma takes over the role of the previous flatness estimates, with the logic reversed: 
we do not bound $\max_v x_v^2$ from above, but use a lower bound on it instead. 

\begin{proof}
Write $\lambda =\lambda (G)$, $A=A(G)$ and
$\bm{d}=A\one$. Fix a vertex $u$ with
$x_u=\max_{v\in V(G)}x_v$ and put 
\[
  \bm{y}:=\bm{d}-\frac{\lambda}{x_u}\bm{x},
  \qquad
  Y:=\sum_{v\in N(u)}y_v,
  \qquad
  W:=\!\!\sum_{vw\in E(G[N(u)]) }\!\!x_vx_w .
\]
Let $W'$ denote the same sum over the edges of
$G[V(G)\setminus N(u)]$. 

In what follows, we use two facts.
First, we see that 
$  y_v=\sum_{w\sim v}\bigl(1- {x_w}/{x_u}\bigr) \ge 0$ 
for every $v\in V(G)$.
Second, the Cauchy--Schwarz inequality gives
$\lambda^2x_u^2=(\sum_{w\sim u}x_w)^2
\le d_u\sum_{w\sim u}x_w^2\le d_u$. So
$y_u=d_u-\lambda \ge \lambda^2x_u^2-\lambda
\ge \frac12\lambda^2x_u^2$ since
$\lambda x_u^2\ge 16\eta^{-5}\ge 2$.

We write
$w_\ell (G) =\one^{\mathsf T}A^{\ell -1}\one
=\bm{d}^{\mathsf T}A^{\ell -3}\bm{d}$. 
Since $A^{\ell -3}\bm{x}=\lambda^{\ell -3}\bm{x}$ and 
$\lVert \bm{x}\rVert =1$, we have 
$ \bm{y}^{\mathsf T}\bm{x} 
  =(A\bm{1})^{\mathsf T}\bm{x} -\frac{\lambda}{x_u}
  =\lambda \lVert \bm{x}\rVert_1-\frac{\lambda}{x_u}$. 
Expanding $\bm{d}=\bm{y}+\lambda x_u^{-1}\bm{x}$ in $w_{\ell} (G)$, we
obtain
\begin{align}
  w_\ell (G) - 2 \lambda^{\ell} (G) 
  &=\bm{y}^{\mathsf T}A^{\ell -3}\bm{y}
    +\frac{2\lambda}{x_u}\,
     \bm{y}^{\mathsf T}A^{\ell -3}\bm{x}
    +\frac{\lambda^{2}}{x_u^{2}}\,
     \bm{x}^{\mathsf T}A^{\ell -3}\bm{x}
    -2\lambda^{\ell}  \notag \\[1pt]
  &=\bm{y}^{\mathsf T}A^{\ell -3}\bm{y}
    +\frac{2\lambda^{\ell -2}}{x_u}\Big( \lambda \lVert \bm{x}\rVert_1-\frac{\lambda}{x_u} \Big)
    +\frac{\lambda^{\ell -1}}{x_u^{2}}
    -2\lambda^{\ell}  \notag \\[1pt]
  &=\bm{y}^{\mathsf T}A^{\ell -3}\bm{y}
    +\frac{\lambda^{\ell -1}}{x_u^{2}}
     \bigl(2x_u\lVert \bm{x}\rVert_1-1
     -2\lambda x_u^{2}\bigr), \label{eq-difference}
\end{align}
and we bound the two terms on the right separately.
Appending the edge $uv$ to a closed walk at $u$ is
injective, so 
$(A^{\ell -3})_{uv}\ge (A^{\ell -4})_{uu}$ for $v\in N(u)$. 
Since $\ell$ is even, we may put $k:= \frac{1}{2}(\ell -4)$. The Cauchy--Schwarz inequality gives $(A^{\ell -4})_{uu}
=\lVert A^{k}\bm{e}_u\rVert^{2}
\ge \langle A^{k}\bm{e}_u,\bm{x}\rangle^{2}
=\lambda^{\ell -4}x_u^{2}$. Since
$A^{\ell -3}$ and $\bm{y}$ are nonnegative, keeping only
the distinct ordered pairs $(u,v)$ and $(v,u)$ with
$v\in N(u)$ gives
\[
  \bm{y}^{\mathsf T}A^{\ell -3}\bm{y} \ \ge \ 
  \sum_{v\in N(u)} 2\, y_u (A^{\ell -3})_{uv} y_v 
  \ \ge\ 2y_u\lambda^{\ell -4}x_u^{2}Y
  \ \ge\ \lambda^{\ell -2}x_u^{4}Y .
\]
For the other term, the Perron equation at $u$ gives
$\sum_{v\in N(u)}x_v = \lambda x_u$, hence
\begin{equation} \label{eq-xu-1-norm}
x_u\lVert \bm{x}\rVert_1 \ = \ \sum_{v\in N(u)} x_u x_v 
+ \sum_{v\notin N(u)} x_ux_v 
\ \ge \  \lambda x_u^{2}
+\sum_{v\notin N(u)}x_v^{2}.
\end{equation}
Multiplying the Perron equation at $v$ by $x_v$ and
summing over $v\in N(u)$ gives
$$\lambda \sum_{v\in N(u)} x_v^2= \sum_{v\in N(u)} \sum_{w\in N(v)} x_vx_w =  2W + C, 
\quad \text{where}\quad  C:=\sum_{\substack{v\in N(u),\,w\notin N(u) \\ vw\in E(G)}} x_vx_w, $$ 
since each edge inside $N(u)$ is counted once from each of its two ends and each crossing edge once. The same computation
over $V(G)\setminus N(u)$ gives
$\lambda (\sum_{v\notin N(u)} x_v^2)=2W'+C$ with the same cross sum $C$. Subtracting
the two identities and using
$\sum_{v\notin N(u)} x_v^2 =1 - \sum_{v\in N(u)} x_v^2 $, we obtain
\begin{equation} \label{eq-identity}
\lambda \bigg(2\sum_{v\in N(u)} x_v^2 -1 \bigg) 
\ = \ 2W-2W'.
\end{equation}
Using (\ref{eq-xu-1-norm}) and (\ref{eq-identity}), we get 
\[
  2x_u\lVert \bm{x}\rVert_1-1-2\lambda x_u^{2}
  \ \ge\ 1-2 \sum_{v\in N(u)} x_v^2 \ =\ \frac{2W'-2W}{\lambda}
  \ \ge\ -\frac{2W}{\lambda},
\]
where the first step is an equality precisely when
$x_v=x_u$ for every $v\notin N(u)$, and the last step precisely
when $W'=0$. Combining the above inequalities, we get from (\ref{eq-difference}) that 
\begin{equation}\label{eq:even-reduction}
  w_\ell (G)-2\lambda^{\ell}(G)
  \ \ge\  \lambda^{\ell -2}x_u^{4}Y - \frac{\lambda^{\ell -1}}{x_u^{2}} \cdot \frac{2W}{\lambda} \ = \ 
  \frac{\lambda^{\ell -2}}{x_u^{2}}
         \bigl(x_u^{6}Y-2W\bigr).
\end{equation}
We claim that $x_u^{6}Y\ge 2W$, with strict inequality
when $W>0$. 
The case $W=0$ is trivial, so assume that
$W>0$. Note that $2W\le \lambda$
and $x_u\le 1$. Suppose first that
$x_v>\frac12x_u^{3}$ for some $v\in N(u)$. Then
$y_v\ge \lambda^2x_v^{2}-\lambda
>\frac14\lambda^{2}x_u^{6}-\lambda$, while
$x_u^{12}\ge \eta^{6}$ and $\lambda \ge 16\eta^{-6}$
give $\frac14\lambda x_u^{12}\ge 4$. Hence
\[
  x_u^{6}Y\ \ge\ x_u^{6}y_v
  \ >\ \lambda \Bigl(\tfrac14\lambda x_u^{12}
  -x_u^{6}\Bigr)
  \ \ge\ 3\lambda \ >\ 2W .
\]
Suppose now that $x_v\le \frac12x_u^{3}$ for every
$v\in N(u)$. Every edge $vw$ of $G[N(u)]$ satisfies
$x_vx_w\le \frac14x_u^{6}$, so $W\le e(G[N(u)]) \, \frac14x_u^{6}$, which gives 
$e(G[N(u)])\ge 4W/x_u^{6}$. 
Every edge $vw$ of $G[N(u)]$ contributes the two
terms $y_v$ and $y_w$ to $Y$, and the sum  
$y_v+y_w \ge (1-x_w/x_u) + (1-x_v/x_u)= 2- {(x_v+x_w)}/{x_u} \ge 2-x_u^{2} \ge 1 $. 
Hence $Y\ge e(G[N(u)])$, and therefore $x_u^{6}Y\ge x_u^{6}\, e(G[N(u)])\ge 4W>2W$, as claimed. 
By \eqref{eq:even-reduction}, the claim gives
$w_\ell (G)\ge 2\lambda^{\ell}(G)$, which is the first
assertion of the lemma. 

\smallskip
\noindent
\textit{Equality case.}\ Suppose that
$w_\ell (G)=2\lambda^{\ell}(G)$ for some even $\ell \ge 4$.
Then \eqref{eq:even-reduction} gives $x_u^{6}Y \le 2W$,
so $W=0$ by the above claim, and then $Y=0$. Both sides of
\eqref{eq:even-reduction} vanish, so every inequality
used in its derivation is an equality; in particular
$W'=0$ and $x_v=x_u$ for every $v\notin N(u)$. As $\bm{x}$
is positive, $W=W'=0$ says that $N(u)$ and $V(G)\setminus N(u)$ are independent sets. Let $z\notin N(u)$. Then
$N(z)\subseteq N(u)$, and the Perron equation at $z$ reads
$\sum_{v\in N(z)}x_v=\lambda x_z=\lambda x_u
=\sum_{v\in N(u)}x_v$, so $N(z)=N(u)$. Hence $G$ is complete
bipartite with parts $N(u)$ and $V(G)\! \setminus \! N(u)$.
Conversely, a walk on $\ell$ vertices of $K_{a,b}$
alternates between the parts, so
$w_\ell (K_{a,b})=2(ab)^{\ell /2}\! =2\lambda^{\ell}$ for
 even $\ell$.
\end{proof}

We are now in a position to prove Theorem
\ref{thm:large-lambda-chi=3}.

\begin{proof}[{\bf Proof of Theorem
\ref{thm:large-lambda-chi=3}}]
 Put
$\eta :=\min \{\eta (F),\frac12\}$, where $\eta (F)$ is
the constant of Theorem \ref{thm-second-key}, let $n_0(F)$
be the threshold in Theorem \ref{thm:simonovits}, and set
$\lambda_0(F):=\max \{16\eta^{-6},\ n_0(F)\}$. Since
$\eta \le \frac12$, we have $\eta^{-5}\ge 32$ and hence
$16\eta^{-6}\ge 2/\eta$, a bound used for odd $\ell$.

\smallskip
Suppose that an $F$-free graph $G$ with
$\lambda (G)\ge \lambda_0$ violates  
\eqref{eq:main-bound-chi=3} for some $\ell \neq 2$.
Choose a component $C$ with $\lambda (C)=\lambda (G)$.
Since $w_\ell (C)\le w_\ell (G)$, we have 
\begin{equation}\label{eq:component-C-chi3}
  \lambda^{\ell}(C)=\lambda^{\ell}(G)
  >\tfrac12 w_\ell (G) \ge \tfrac12 w_\ell (C).
\end{equation}
Therefore, we have $\lambda (C)>0$ and $C$ has at least one edge. Let
$\bm{x}$ be the positive unit Perron vector of $C$ and
$\pi_v=x_v^2$. The entropy-Perron bridge of Theorem \ref{thm:entropy-bridge} gives  
$\lambda^{\ell} (C) \le \Phi_C(\pi) w_{\ell}(C)$, which 
together
with \eqref{eq:component-C-chi3} implies 
$\Phi_C(\pi )>\frac12$, so the contrapositive of Theorem
\ref{thm-second-key} yields
\begin{equation}\label{eq:pi-heavy-chi3}
  \max_{v\in V(C)}\pi_v>\eta .
\end{equation}
We contradict this in each of the two ranges of $\ell$.

If $\ell$ is odd, then \eqref{eq:component-C-chi3} is the
hypothesis of Lemma \ref{lem:flat-odd} with
$c_0=\frac12$. So the flatness estimate of Lemma \ref{lem:flat-odd} gives
$\max_{v\in V(C)}\pi_v\le 2/\lambda (C)\le 2/\lambda_0
\le \eta$. This leads to a contradiction with \eqref{eq:pi-heavy-chi3}.

If $\ell \ge 4$ is even, then no flatness estimate is available, 
and \eqref{eq:pi-heavy-chi3} together with 
$\lambda (C)\ge \lambda_0\ge 16\eta^{-6}$ gives the
hypotheses of Lemma \ref{lem:large-perron}. 
So Lemma \ref{lem:large-perron} gives
$\lambda^{\ell}(C) \le \frac{1}{2} w_\ell (C)$, which  contradicts \eqref{eq:component-C-chi3}. 

Hence, every $F$-free graph $G$ with
$\lambda (G)\ge \lambda_0$ must satisfy 
\eqref{eq:main-bound-chi=3} for every $\ell \neq 2$.

\smallskip
\noindent
{\bf Equality case.}~
Suppose that $G$ is an $F$-free graph with
$\lambda (G)\ge \lambda_0$ and
$\lambda^\ell (G)=\frac12w_\ell (G)$ for some
$\ell \neq 2$. Exactly as in the proof of Theorem
\ref{thm:large-lambda-ge4}, comparing $G$ with a component
$C$ attaining $\lambda (C)=\lambda (G)$ shows that
$w_\ell (C)=w_\ell (G)$, that the remaining components are
isolated vertices when $\ell \ge 2$ and absent when
$\ell =1$, and that we may assume $G$ to be connected.

Assume first that $\ell \ge 4$ is even and that
$\max_{v\in V(G)}\pi_v\ge \eta$. Since
$\lambda (G)\ge 16\eta^{-6}$, Lemma
\ref{lem:large-perron} applies, and the equality
$w_\ell (G)=2\lambda^{\ell}(G)$ forces $G$ to be complete
bipartite.

In the remaining cases, we have
$\max_{v\in V(G)}\pi_v<\eta$: for even $\ell$ this is the
complementary assumption, and for odd $\ell$ it follows
from Lemma \ref{lem:flat-odd} with $c_0=\frac12$, which
gives $\max_v\pi_v\le 2/\lambda (G)\le 2/\lambda_0\le
\eta$. Together with $\lambda (G)\ge \lambda_0\ge n_0(F)$,
this puts $G$ within the hypotheses of Lemma
\ref{lem:equality-rigidity}. Its alternative (b) requires
$r\ge 3$, so alternative (a) holds and $G=K_{a,b}$ is
complete bipartite.

It remains to identify $a$ and $b$ when $\ell$ is odd. A
homomorphism from $P_\ell$ to $K_{a,b}$ sends one side of
the bipartition of $P_\ell$ into one part and the other
side into the other, so
$w_\ell (K_{a,b})=(ab)^{(\ell -1)/2}(a+b)$, while
$\lambda^{\ell}(K_{a,b})=(ab)^{(\ell -1)/2}\sqrt{ab}$.
Thus $\lambda^\ell (G)=\frac12w_\ell (G)$ reads
$\sqrt{ab}=\frac{a+b}{2}$, which gives $a=b$.

Conversely, $K_{a,b}$ is bipartite, and
$\lambda (K_{a,b})=\sqrt{ab}$. For even $\ell$, a walk on $\ell$ vertices of $K_{a,b}$ alternates between the two parts, so $w_\ell (K_{a,b})=2(ab)^{\ell /2}
=2\lambda^{\ell}(K_{a,b})$, and every complete bipartite graph with $ab\ge \lambda_0^2$ attains equality. For odd $\ell$, the  above argument shows that $K_{a,a}$ attains
equality as soon as $a\ge \lambda_0$. Adding isolated vertices changes
neither side of \eqref{eq:main-bound-chi=3} when
$\ell \ge 2$.
\end{proof}

\begin{remark} 
\label{rem:parity}
The pair $(r,\ell )=(2,2)$ is the only one excluded from
Theorem \ref{thm:large-lambda}, and it is excluded for a
genuine reason rather than as an artifact of the proof: by
Example \ref{ex:book}(i) below, the bound
\eqref{eq:main-bound-chi=3} already fails at $\ell =2$ for
$F=C_5$ along a family whose spectral radius tends to
infinity. When $r=2$, the coefficient $1-\frac1r$ equals
$\frac12$, so Lemma \ref{lem:perron-flatness} supplies no
fixed $\delta >0$, and the two substitutes used above
cover exactly the remaining values: Lemma
\ref{lem:flat-odd} needs $\ell$ to be odd, and Lemma
\ref{lem:large-perron} needs $\ell \ge 4$ to be even. A
threshold $\lambda_0$ is also unavoidable, by Example
\ref{ex:book}(ii).
\end{remark}

\begin{example} \label{ex:book}
Let $B_t:=K_2\vee I_t$ be the book with $t$ pages, and let
$F=C_5$, which is color-critical with $\chi (F)=3$. The
graph $B_t$ is $C_5$-free, and its equitable quotient
matrix
$\bigl(\begin{smallmatrix}1&t\\2&0\end{smallmatrix}\bigr)$
gives $\lambda^{2}(B_t)=\lambda (B_t)+2t$, so that
$\lambda (B_t)=\frac{1}{2}(1+\sqrt{1+8t}\,)\to \infty$ as
$t\to \infty$.
\begin{itemize}
\item[(i)] Since $e(B_t)=2t+1$ and $\lambda (B_t)>1$, we
  have $\lambda^{2}(B_t)>e(B_t)=\tfrac12w_2(B_t)$ for
  every $t\ge 1$.
\item[(ii)] A walk on four vertices is determined by its
  middle edge together with its neighbors in two
  directions, so
  $\tfrac12w_4(B_t)=\sum_{uv\in E(B_t)}d_ud_v
  =5t^2+6t+1$, while
  $\lambda^{4}(B_t)=(4t+1)\lambda (B_t)+4t^2+2t$. Hence,
  we see that $\lambda^{4}(B_t)\le \tfrac12w_4(B_t)$ fails
  for $t\le 28$ and holds for $t\ge 29$.
\end{itemize}
\end{example}

\section{Weighted supersaturation and weighted stability} 

\label{sec:weighted}

In this section, we prove 
the two weighted Tur\'{a}n theorems
that power every spectral result of this paper. 
In the weighted setting, 
there is no vertex count to induct on: 
the hypothesis ``$n$ is sufficiently large'' of the classical Tur\'{a}n results is replaced by the flatness condition on
$\lVert \bm{p}\rVert_\infty$, which is exactly what the
Perron vector of a graph with large spectral radius
supplies.

We first prove a weighted supersaturation result  (Theorem \ref{thm:weighted-supersat}), 
which implies Theorem \ref{thm-weighted-ESS}, 
and then we prove the corresponding stability 
result (Theorem \ref{thm:weighted-stability}). 
In Section \ref{sec-prove-exact}, 
we combine these two results with a random sampling argument to prove the exact bound of Theorem \ref{thm-second-key}.

\subsection{A weighted supersaturation theorem} 

Recall that $N_F(G)$ denotes the number of copies of $F$ in $G$. For a
probability vector $\bm{p}$ on $V(G)$, we define the \emph{weighted
number of copies} of $F$ in $G$ by
\[
  N_F(G,\bm{p})\ :=\ \sum_{F^*}\ \prod_{v\in V(F^*)}p_v ,
\]
where the sum runs over all copies $F^*$ of $F$ in $G$, that is, over
all subgraphs of $G$ isomorphic to $F$. 
 If $\bm{p}$ is uniformly distributed on $V(G)$, then $N_F(G,\bm{p})=N_F(G)/n^{v(F)}$, and the
next weighted theorem recovers the supersaturation theorem of Erd\H{o}s and Simonovits
\cite{ES1983}. 

\begin{theorem}[Weighted supersaturation]
\label{thm:weighted-supersat}
Let $F$ be a graph with $\chi (F)=r+1\ge 3$ and $f=v(F)$. For every
$\varepsilon >0$, there exist $\eta =\eta (F,\varepsilon )>0$ and
$c=c(F,\varepsilon )>0$ such that if $G$ is a graph and $\bm{p}$ is a
probability vector on $V(G)$ satisfying
$\lVert \bm{p}\rVert_\infty\le \eta$ and
$\Phi_G(\bm{p})\ge 1- \frac{1}{r} +\varepsilon$, then
\[ N_F(G,\bm{p})\ \ge\ c.  \]
\end{theorem}

Since $N_F(G,\bm{p})=0$ whenever $G$ is $F$-free, Theorem
\ref{thm:weighted-supersat} implies Theorem \ref{thm-weighted-ESS}.

\begin{proof}  
Given $\varepsilon>0$, we first fix the constants $\nu$ and $N_0$ in the following ways. Apply Theorem
\ref{thm:removal} with parameter $\varepsilon/16$ to obtain
$\nu=\nu(F,\varepsilon/16)>0$ such that every graph on $N$ vertices
containing at most $\nu N^f$ copies of $F$ can be made $F$-free by
deleting at most $(\varepsilon/16)N^2$ edges, and apply Theorem
\ref{thm:ess} with parameter $\varepsilon/8$ to obtain
$N_0=N_0(F,\varepsilon/8)$. We now set 
\[
  \eta:=\frac{\nu}{4\binom f2}
  \qquad \text{and} \qquad
  c:=\frac{\nu}{2^{f+1}}.
\]

Suppose on the contrary that there exist a graph $G$ and a probability
vector $\bm p$ on $V(G)$ with $\lVert\bm p\rVert_\infty\le\eta$ and 
$\Phi_G(\bm p)\ge 1-\frac1r+\varepsilon$, but $N_F(G,\bm p)<c$. For every
$v\in V(G)$, we can choose a {\it rational} number $q_v'\in[\tfrac12 p_v,\,p_v]$ with
$\sum_{v\in V(G)}(p_v-q_v')\le\frac{\varepsilon}{32}$. We normalize these entries by putting 
$q_v:=q_v'/S$, where $S:=\sum_{u\in V(G)}q_u'$. Then $\frac12 \le S \le 1$, so
$q_v= q_v'/S\le 2p_v$. Moreover $q_v\ge q_v'$ and
$\sum_v(q_v-q_v')=1-S=\sum_v(p_v-q_v')$. Hence $\bm q=(q_v)_{v\in V(G)}$
is a rational probability vector with
\[
  \sum_{v\in V(G)}\lvert p_v-q_v\rvert\le\frac{\varepsilon}{16}
  \quad\text{and}\quad
  q_v\le 2p_v \quad\text{for every } v\in V(G).
\]
Writing $\Phi_G(\bm p)=\sum_{u,v}a_{uv}p_up_v$ as a sum over ordered pairs and using $p_u p_v-q_u q_v = 
p_u(p_v-q_v) + q_v(p_u-q_u)$, the first property gives $\lvert \Phi_G(\bm p)-\Phi_G(\bm q)\rvert\le 2\sum_{v\in V(G)}\lvert p_v-q_v\rvert\le \varepsilon/8$. Hence, $\bm{q}$ is a rational probability vector  satisfying
\begin{equation}\label{eq:q-density-supersat}
  \Phi_G(\bm q)\ \ge\ 1-\frac1r+\frac{7\varepsilon}{8} \quad \text{and} \quad 
  \lVert\bm q\rVert_\infty\le2\eta. 
\end{equation}
 The second property gives
$\prod_{v\in V(F^*)}q_v\le 2^{f}\prod_{v\in V(F^*)}p_v$ for every copy
$F^*$ of $F$ in $G$, and therefore
\begin{equation}\label{eq:NF-transfer}
  N_F(G,\bm q)\ \le\ 2^{f}N_F(G,\bm p)\ <\ 2^{f}c\ =\ \frac{\nu}{2}.
\end{equation}

Since all entries $q_v$ are rational numbers,  
we can choose $N_{\bm q}$ as a common denominator of the $q_v$ with $q_v>0$. For
a positive integer $t$, put $N:=tN_{\bm q}$, and construct the blow-up
$B$ of $G$ by replacing each vertex $v$ with an independent cluster
$C_v$ of size $k_v:=Nq_v$ and placing a complete bipartite graph
between $C_u$ and $C_v$ for every $uv\in E(G)$. Then $v(B)=N$ and, by
\eqref{eq:q-density-supersat},
\[
  \frac{2e(B)}{N^2}\ = \  2\sum_{uv\in E(G)} q_u q_v 
 \ = \ \Phi_G(\bm q)\ \ge\ 1-\frac1r+\frac{7\varepsilon}{8}.
\]

We claim that the blow-up $B$ contains at most $\nu N^f$ copies of $F$. 
We count the copies of $F$ in $B$ in two ways according to whether two of their vertices lie in a common cluster.
First of all, for a fixed pair of vertices of $F$, the number of injective maps
$V(F)\to V(B)$ sending that pair into a common cluster is at most
$$\sum_{v\in V(G)}k_v^2\cdot N^{f-2}\ = \ N^f\sum_{v\in V(G)} q_v^2 
\ \le \ N^f\lVert\bm q\rVert_\infty.$$
So a union bound over the ${f \choose 2}$ pairs of $F$ shows that the number of copies of $F$ in $B$ having two vertices in a common cluster of $B$ is at most 
\[
  \binom f2 N^f\lVert\bm q\rVert_\infty\ \le\ 2\binom f2\eta\, N^f
  \ =\ \frac{\nu}{2}N^f. 
\] 
Now, consider a copy $F^*$ of $F$ in $B$ whose vertices lie in pairwise distinct
clusters. Projecting each vertex of $F^*$ to the vertex of $G$ whose
cluster contains it is injective on $V(F^*)$ and maps $E(F^*)$ to
edges of $G$, so the projection of $F^*$ is a copy $F_0$ of $F$ in
$G$. Conversely, for a fixed copy $F_0$ of $F$ in $G$, the copies of
$F$ in $B$ projecting to $F_0$ are obtained precisely by choosing one
vertex from the cluster $C_u$ for each $u\in V(F_0)$ and lifting the
edges of $F_0$, which is possible because the clusters of adjacent
vertices of $G$ are joined completely. This gives
$\prod_{u\in V(F_0)}k_u$ copies. Hence the number of copies of $F$ in
$B$ with pairwise distinct clusters equals
\[
  \sum_{F_0}\ \prod_{u\in V(F_0)}k_u
  \ =\ N^f\sum_{F_0}\ \prod_{u\in V(F_0)}q_u
  \ =\ N^f\,N_F(G,\bm q)\ <\ \frac{\nu}{2}N^f,
\]
where the sums run over all copies $F_0$ of $F$ in $G$, and the last
step uses \eqref{eq:NF-transfer}. Adding the above two counts proves that $B$ contains at most $\nu N^f$ copies of $F$, as claimed.

By Theorem \ref{thm:removal} and the choice of $\nu$, we can delete
at most $(\varepsilon/16)N^2$ edges from $B$ to produce an $F$-free subgraph
$H$ on the same $N$ vertices. 
Therefore, we obtain 
\[
  \frac{2e(H)}{N^2}\ \ge\ \frac{2e(B)}{N^2}-\frac{\varepsilon}{8}
  \ \ge\ 1-\frac1r+\frac{3\varepsilon}{4}.
\]
On the other hand, choosing $t$ so large that $N\ge N_0$ and applying
Theorem \ref{thm:ess} to $H$ gives
$$ \frac{2e(H)}{N^2} \le 1-\frac1r+\frac{\varepsilon}{8},$$
 a contradiction. Hence
$N_F(G,\bm p)\ge c$ for every $G$ and $\bm p$ as in the statement.
\end{proof}

\subsection{A weighted stability theorem}

Suppose that $V(G)=U_1\cup\cdots\cup U_r$ is a
partition of the vertex set of $G$. Let $\bm{p}=(p_v)_{v\in V(G)}$ be a  probability vector on $V(G)$. 
We call a pair of vertices lying in two different parts a \emph{cross-pair}, and we call a cross-pair that is not an edge of $G$ a \emph{missing cross-pair}. 
Define the total weight of the missing cross-pairs by
\begin{equation}\label{eq:M-definition}
  M=M(G,\bm{p},U_1,\ldots,U_r)
  :=\sum_{ i<j }
    \sum_{\substack{u\in U_i, v\in U_j\\uv\notin E(G)}}p_up_v.
\end{equation}

Next, we prove a stability version of Theorem \ref{thm-weighted-ESS}. The proof follows the same lines:  replacing each
vertex $v$ by an independent set of size proportional to
$q_v$ turns the weighted edge density into the 
edge density of a blow-up $B$, and the small
weights of $\bm{p}$ make $B$ contain only  $o(N^f)$ copies of $F$.  
The graph removal lemma and the
Erd\H{o}s--Stone--Simonovits theorem then apply and show that $B$ 
is nearly extremal up to $o(N^2)$ edges. Consequently, Theorem
\ref{thm:stability} gives a partition of $V(B)$
into $r$ nearly equal parts with few missing
cross-pairs. Assigning each vertex of $G$ to a
random part, with probability the fraction of its
cluster in that part, carries this partition back to
$G$. The small weights make the part masses
and the weight of missing cross-pairs concentrate.

\begin{theorem}[Weighted stability]
\label{thm:weighted-stability}
Let $F$ be a graph with $\chi (F)=r+1\ge 3$. For every
$\varepsilon >0$, there are $\eta =\eta (F,\varepsilon )>0$ and $\delta =\delta(F,\varepsilon)>0$ such that
if $\bm{p}$ is a probability vector with
$\Vert \bm{p} \Vert_\infty\le \eta$, 
and $G$ is an $F$-free graph with $\Phi_G(\bm{p})\ge 1- \frac{1}{r} - \delta$, then there is a
  partition 
  $$V(G)=U_1\cup \cdots \cup U_r$$
   satisfying 
  \[
    \Bigl|p(U_i)-\frac1r\Bigr|\le \varepsilon 
    \quad (\text{for every}~i\in [r]),
    \qquad
    M(G,\bm{p},U_1,\ldots ,U_r)\le \varepsilon .
  \] 
\end{theorem}

\begin{proof}
Let $f:=v(F)$, and assume without loss of generality
that $\varepsilon\le\frac12$. We first fix the
constants. 
Apply Theorem \ref{thm:stability} with parameter
$\varepsilon/8$ to obtain $\theta=\theta(F,\varepsilon/8)>0$
and $N_1=N_1(F,\varepsilon/8)$ such that every $F$-free
graph $D$ on $N\ge N_1$ vertices with
$e(D)\ge\bigl(1-\frac1r-\theta\bigr)\frac{N^2}{2}$
admits a partition $V(D)=B_1\cup\cdots\cup B_r$ with
$\bigl\lvert|B_i|-\frac Nr\bigr\rvert
\le\frac{\varepsilon}{8}N$ for every $i\in[r]$ and
with at most $\frac{\varepsilon}{8}N^2$ pairs lying
in different parts that are non-adjacent in $D$.
Apply Theorem \ref{thm:removal} with parameter
$\theta/8$ to obtain $\nu=\nu(F,\theta/8)>0$ such that
every graph on $N$ vertices containing at most
$\nu N^f$ copies of $F$ can be made $F$-free
by deleting at most $\frac{\theta}{8}N^2$ edges.
Finally, we set  
\[
  \delta:=\frac{\theta}{4},
  \qquad
  \eta:=\min\Bigl\{\frac{\nu}{f^2},\
  \frac{\varepsilon^2}{32r},\
  \frac{\varepsilon}{8}\Bigr\}.
\]

\noindent
\textit{Step 1. Rationalization and blow-up.}
Put $\xi:=\min\{\eta,\frac{\theta}{16},\frac{\varepsilon}{8}\}$. Since $V(G)$ is finite and the rational points are dense, 
we may choose a {\it rational probability vector} $\bm q=(q_v)_{v\in V(G)}$ with
$\sum_{v}\lvert p_v-q_v\rvert\le\xi$ and with $q_v>0$ for every $v\in V(G)$. 
For any set $\mathcal P$ of unordered pairs of distinct vertices, we get 
\[
  2\Bigl\lvert\sum_{uv\in\mathcal P}
  (p_up_v-q_uq_v)\Bigr\rvert
  \le\sum_{u}\sum_{v}
  \bigl(p_u\lvert p_v-q_v \rvert
  +q_v\lvert p_u-q_u \rvert\bigr)\le2\xi .
\]
Taking $\mathcal P=E(G)$, 
and taking $\mathcal P$ to be the set
of cross non-edges of any partition, we obtain 
\begin{equation}\label{eq:rat}
  \lvert\Phi_G(\bm p)-\Phi_G(\bm q)\rvert\le2\xi,
  \quad
  M(G,\bm p, \cdot)\le M(G,\bm q, \cdot)+2\xi,
  \quad
  \lvert p(S)-q(S)\rvert\le\xi
\end{equation}
for every $S\subseteq V(G)$. Moreover
$\lVert\bm q\rVert_\infty\le\eta+\xi\le2\eta
\le\nu/\binom f2$. In particular
\begin{equation}\label{eq:q-density}
  \Phi_G(\bm q)\ \ge\ \Phi_G(\bm p)-2\xi
  \ \ge\ 1-\frac1r-\frac{\theta}{4}
   -\frac{\theta}{8}
  \ \ge\ 1-\frac1r-\frac{3\theta}{8}.
\end{equation}
Let $N_{\bm q}$ be a common denominator of the $q_v$ with $q_v >0$. For a positive integer $t$ put
$N:=tN_{\bm q}$, so that $Nq_v$ is a nonnegative integer
for every $v\in V(G)$. Enlarging $t$ does not change
$\bm q$, so $N$ may be taken as large as we wish, and
we fix $t$ so large that $N\ge N_1$. Construct the
blow-up $B$ of $G$ by replacing each $v\in V(G)$ with an
independent cluster $C_v$ of size $k_v:=Nq_v$ and placing a complete bipartite graph between $C_u$ and
$C_v$ for every $uv\in E(G)$. 
Then $v(B)=N$.  By  \eqref{eq:q-density},
\begin{equation}\label{eq:blowup-density}
  \frac{2e(B)}{N^2}
  =2\sum_{uv\in E(G)}q_uq_v
  =\Phi_{G}(\bm q)
  \ \ge\ 1-\frac1r-\frac{3\theta}{8}.
\end{equation}

\noindent
\textit{Step 2. Using the graph removal lemma and stability.} 
We claim that $B$ contains at most $\nu N^f$  
copies of $F$. Project each vertex of such a copy to
the vertex of $G$ whose cluster contains it. If the
projections were distinct they would form a copy of
$F$ in $G$, contradicting $F$-freeness. Hence two
vertices of the copy lie in a common cluster. For a
fixed pair of vertices of $F$, the number of injective
maps $V(F)\to V(B)$ sending that pair into a common
cluster is at most
\[
  \sum_{v\in V(G)}k_v^2N^{f-2}
  =N^f\sum_{v\in V(G)}q_v^2
  \le N^f\lVert\bm q\rVert_\infty ,
\]
so a union bound over the $\binom f2$ pairs gives
\[
  N_F(B)\le\binom f2N^f\lVert\bm q\rVert_\infty
  \le\nu N^f,
\]
where $N_F(B)$ is the number of copies of $F$
in $B$. By the choice of $\nu$, deleting at most
$\frac{\theta}{8}N^2$ edges from $B$ yields an
$F$-free graph $H$ on $V(B)$.  Then, by
\eqref{eq:blowup-density},
\[
  \frac{2e(H)}{N^2}\ \ge\ \frac{2e(B)}{N^2}
  -\frac{\theta}{4}
  \ \ge\ 1-\frac1r-\theta .
\]
Since $N\ge N_1$, Theorem \ref{thm:stability} applies
to $H$ and yields a partition
$V(B)=B_1\cup\cdots\cup B_r$ such that 
\begin{equation}\label{eq:part-sizes}
  \Bigl\lvert|B_i|-\frac Nr\Bigr\rvert
  \le\frac{\varepsilon}{8}N
\end{equation}
for every $i\in [r]$, and such that at most $\frac{\varepsilon}{8}N^2$ pairs lie
in different parts and are non-adjacent in $H$. As
$E(H)\subseteq E(B)$, every pair non-adjacent in $B$
is non-adjacent in $H$, so
the number of pairs in different parts that are
  non-edges of $B$ is at most $\frac{\varepsilon}{8}N^2$.

\smallskip 
\noindent
\textit{Step 3. Transferring the partition of the blow-up back to $G$.}
The remaining step is to transform the above partition of the blow-up $B$ into a partition of the original graph $G$.  
For every vertex $v\in V(G)$ and every $i\in[r]$, 
we put  
$$q_{vi}:= \frac{| C_v\cap B_i|}{| C_v|}. $$
Then $\sum_{i}q_{vi}=1$. 
Assign each vertex 
$v\in V(G)$
independently to a random set $U_i$ with probability $q_{vi}$. 
We write $\xi_{vi}:=\mathbf 1_{\{v\in U_i\}}$ 
and $Z_i:=q(U_i)=\sum_{v\in V(G)}q_v\,\xi_{vi}$. Since
$\lvert C_v\rvert=Nq_v$, \eqref{eq:part-sizes} gives
\begin{equation}\label{eq:Zi-expectation}
  \EE[Z_i]=\sum_{v\in V(G)}q_v \, q_{vi}
  =\frac1N\sum_{v\in V(G)}\lvert C_v\cap B_i\rvert
  =\frac{|B_i|}{N}
  =\frac1r\pm\frac{\varepsilon}{8}.
\end{equation}
The variables $\xi_{vi}$, with $v\in V(G)$, are independent
with $\Var(\xi_{vi})=q_{vi}(1-q_{vi})$, so
\begin{equation}\label{eq:Zi-variance}
  \Var(Z_i)=\sum_{v\in V(G)}q_v^2 \, 
  q_{vi}(1-q_{vi})
  \le\sum_{v\in V(G)}q_v^2
  \le\lVert\bm q\rVert_\infty\le2\eta ,
\end{equation}
where the last inequality holds because
$\sum_vq_v^2\le \lVert\bm q\rVert_\infty 
\sum_v q_v = \lVert\bm q\rVert_\infty $.

Let $M:=M(G,\bm q, U_1,\ldots,U_r)$. For a
non-edge $uv$ of $G$ the assignment separates $u$ and
$v$ with probability $1-\sum_{i}q_{ui}q_{vi}$, and the
number of pairs of $C_u\times C_v$ lying in different
parts is exactly
$N^2q_uq_v\bigl(1-\sum_{i=1}^r q_{ui}q_{vi}\bigr)$. 
 All these pairs are non-edges of $B$. Hence,
\begin{equation}\label{eq:M-expectation}
  \EE[M]=\sum_{\substack{uv\notin E(G)\\u<v}}
  q_uq_v\Bigl(1-\sum_{i=1}^rq_{ui}q_{vi}\Bigr)
  \le\frac{\#\{\text{cross non-edges of }B\}}{N^2}
  \le\frac{\varepsilon}{8},
\end{equation}
where the inequality holds because the right-hand side also counts pairs of vertices lying in the same cluster $C_v$, which are non-edges of $B$ and have no counterpart on the left. 

By \eqref{eq:Zi-variance} and Chebyshev's inequality (Lemma \ref{lem:chebyshev}), we get that for each $i\in[r]$, 
\[
  \Pr \bigg( \bigl\lvert Z_i-\EE[Z_i]\bigr\rvert
  >\frac{\varepsilon}{2}  \bigg)
  \le\frac{4\Var(Z_i)}{\varepsilon^2}
  \le\frac{8\eta}{\varepsilon^2} \le  \frac{1}{4r}.
\]
Moreover, by \eqref{eq:M-expectation} and Markov's inequality (Lemma \ref{lem:markov}), we get 
$$\Pr\Bigl(M>\frac{\varepsilon}{2}\Bigr)
\le\frac{\varepsilon/8}{\varepsilon/2}=\frac14.$$
Since $r\cdot\frac{1}{4r}+\frac14=\frac12<1$, a union
bound shows that with positive probability none of
these $r+1$ events occurs. Consequently, 
there exists an $r$-partition $V(G)=U_1\cup \cdots \cup U_r$ satisfying 
\[   \Big| Z_i - \mathbb{E}[Z_i]\Big| \le \frac{\varepsilon}{2} \quad 
\text{for every $i\in [r]$}, \quad \text{and} \quad M \le \frac{\varepsilon}{2}.  \]
Recall that $Z_i=q(U_i)$.  By \eqref{eq:Zi-expectation}, the triangle inequality yields 
\[
  \Bigl\lvert q(U_i)-\frac1r\Bigr\rvert
  \le\frac{\varepsilon}{2}+\frac{\varepsilon}{8}
  ~~ \text{for every}~i\in[r],
  \quad \text{and} \quad 
  M(G,\bm q, U_1,\ldots,U_r)\le\frac{\varepsilon}{2}.
\]
Finally, by \eqref{eq:rat} and $\xi\le\varepsilon/8$, using the triangle inequality again, we have 
\[
  \Bigl\lvert p(U_i)-\frac1r\Bigr\rvert
  \le\frac{5\varepsilon}{8}+\xi\le\varepsilon ~~
  \text{for every}~i\in[r], 
  \quad \text{and} \quad 
  M(G,\bm p, U_1,\ldots,U_r)
  \le\frac{\varepsilon}{2}+2\xi\le\varepsilon,
\]
which is the desired assertion.
\end{proof}

\section{Proof of the exact weighted Tur\'{a}n theorem}

\label{sec-prove-exact}

The proof of Theorem \ref{thm-second-key} proceeds in the following three steps. First, we give a weighted sampling lemma (Lemma \ref{lem:weighted-sampling}). 
Second, we prove the exact bound when the weighted graph is close to a complete $r$-partite graph (Lemma \ref{lem:near-turan-weighted}). Finally, we combine the two with the weighted stability theorem of Section \ref{sec:weighted} to reduce the general case to that local statement.

\subsection{A weighted sampling lemma}

\begin{lemma} \label{lem:weighted-sampling}
Let $r,t\ge 2$, let $V(G)=U_1\cup\cdots\cup U_r$, and let $\bm{p}$ be a probability vector on $V(G)$. For each $i\in[r]$, let $A_i\subseteq U_i$ satisfy $p(A_i)\ge \gamma>0$, and let $s_i$ be an integer with $0\le s_i\le t$. Write $M=M(G,\bm{p},U_1,\ldots ,U_r)$ for the total
weight of the missing cross-pairs in
\eqref{eq:M-definition}. If
\begin{equation}\label{eq:sampling-condition}
  r\binom{t}{2} \cdot \frac{\pmax}{\gamma}
  +\binom{r}{2}t^2 \cdot \frac{M}{\gamma^2}<1,
\end{equation}
then one can choose $s_i$ distinct vertices from each $A_i$ so that any two chosen vertices from different sets $A_i$ and $A_j$ are adjacent.
\end{lemma}

\begin{proof}
For each $i\in [r]$, pick $s_i$ vertices from $A_i$ independently and
with repetitions allowed, each pick landing on $v\in A_i$ with
probability $p_v/p(A_i)$. We call the $s_1+\cdots +s_r$ picks the
\emph{sample positions}, and call an outcome \emph{bad} if two
positions in the same part land on the same vertex, or if two positions
in different parts land on a pair of non-adjacent vertices.

First, fix two positions inside the same set $A_i$.
They land on the same vertex with probability
\[
  \sum_{v\in A_i}
    \left(\frac{p_v}{p(A_i)}\right)^{2}
  =\frac{1}{p(A_i)^2}\sum_{v\in A_i}p_v^2
  \le \frac{\pmax}{p(A_i)^2}\sum_{v\in A_i}p_v
  =\frac{\pmax}{p(A_i)}
  \le \frac{\pmax}{\gamma},
\]
where we used $p_v \le \pmax$ and $p(A_i)\ge \gamma$. 
Moreover, the
number of pairs of positions inside a single part is
$\sum_{i=1}^{r}\binom{s_i}{2}\le r\binom{t}{2}$,
because $s_i\le t$ for every $i$.

Second, fix a position in the set $A_i$ and a position in the set 
$A_j$ with $i<j$. They land on a pair of vertices
that is not an edge with probability
\[
  \sum_{\substack{u\in A_i, v\in A_j\\ uv\notin E(G)}}
  \frac{p_u}{p(A_i)}\cdot \frac{p_v}{p(A_j)}
  \le \frac{M}{p(A_i)\,p(A_j)}
  \le \frac{M}{\gamma^2}, 
\]
where we used $A_i\subseteq U_i$ and
$A_j\subseteq U_j$, so the sum  is
at most $ \frac{M}{p(A_i)\,p(A_j)}$, and then used 
$p(A_i),p(A_j)\ge \gamma$ again. The number of pairs
of positions in different parts is
$\sum_{i<j}s_is_j\le \binom{r}{2}t^2$.

By the union bound, the probability 
that the outcome is bad is at most
\[
   r\binom{t}{2} \cdot \frac{\pmax}{\gamma}
  +\binom{r}{2}t^2 \cdot \frac{M}{\gamma^2},
\]
which is less than $1$ by the hypothesis. So the probability that the
outcome is not bad is positive, and therefore at
least one outcome is not bad. In that outcome, the
$s_i$ vertices picked from $A_i$ are pairwise distinct for
every $i$, and every two picked vertices from
different parts are adjacent.
\end{proof}

\subsection{An exact bound near an $r$-partite
template} 

This subsection is the only place in which we obtain the
bound $1-\frac1r$ with no error term. 
The setting is local: 
the partition $U_1,\ldots ,U_r$ is already 
close to an $r$-partite template, and we show that 
no weighted graph near this template exceeds the bound.  
In what follows, we fix integers $r\ge 2$ and
$t\ge 2$. When we apply Lemma \ref{lem:near-turan-weighted}, $t$ will be the value $v(F)$ supplied by Lemma \ref{lem:critical-template}. Write
\[
  C_1:=r\binom{t}{2},\qquad
  C_2:=\binom{r}{2}t^2
\] 
for the two constants appearing in
\eqref{eq:sampling-condition}. Set
\begin{equation}\label{eq:constants-near}
  \varepsilon_*:=\frac{1}{100r},
  \qquad
  \rho_*:=\frac{1}{1000\,r^2C_2},
  \qquad
  \eta_*:=\frac{1}{100\,rC_1}.
\end{equation}
Note that $C_1\ge 2$ and $C_2\ge 1$. All three
constants depend only on $r$ and $t$.

\begin{lemma}[Exact bound near an $r$-partite graph] 
\label{lem:near-turan-weighted}
Let $r\ge 2$, let $G$ be a $K_r^+[t]$-free graph, let
$\bm{p}$ be a probability vector on $V(G)$, and
suppose that $V(G)=U_1\cup\cdots\cup U_r$ is a
partition satisfying 
\begin{equation*}
  \left|p(U_i)-\frac1r\right|\le \varepsilon_*
  \quad\text{for every }i\in[r],
\end{equation*}
\begin{equation*}
  M(G,\bm{p},U_1,\ldots,U_r)\ \le\  \rho_*,
\end{equation*}
and
\begin{equation*}
  \pmax \ \le \ \eta_*.
\end{equation*}
Then 
$$\Phi_G(\bm{p}) \ \le \ 1-\frac{1}{r}.$$
\end{lemma}

The proof of Lemma \ref{lem:near-turan-weighted} takes
its strategy from the edge-spectral argument of Li, Liu
and Zhang \cite[Theorem 3.1]{LLZ-Critical}, 
but the weighted setting calls for a different implementation: there is no vertex count available, 
and every counting step is replaced by the random sampling of Lemma \ref{lem:weighted-sampling}. 
 The bound on $\pmax$ keeps independent samples from landing on the same vertex. The bound on $M$ makes
samples from different parts adjacent. Together they let Lemma \ref{lem:weighted-sampling} build a copy of $K_r^+[t]$ from any local excess. Since $G$ is $K_r^+[t]$-free, no such excess exists. The bound on the part masses makes this quantitative: every vertex misses a fixed fraction of the weight of some part. 
The outcome is an almost exact $r$-partition: $r$
independent sets carry all but a small fraction of
the weight, and the Cauchy--Schwarz inequality then completes the proof.

\begin{proof}
For $i\in[r]$, write $a_i:=p(U_i)$. 
The hypothesis on $p(U_i)$ gives 
\begin{equation}\label{eq:a-lower}
  a_i\ge \frac{99}{100r}
  \qquad\text{for every }i\in[r].
\end{equation}
We use this twice. First, for any two different
indices $i,j$, we have 
\begin{equation}\label{eq:two-part-deficit}
  \frac{9}{10}a_i+\frac13a_j
  \ge \frac{37}{30}\cdot\frac{99}{100r} 
  >\frac1r+\frac{1}{10r}.
\end{equation}
Second, put  $\gamma:=\frac{1}{11r}$ and 
$M:=M(G,\bm{p},U_1,\ldots,U_r)$.
By \eqref{eq:a-lower}, each of the numbers
$\frac{1}{10}a_i$, $\frac13a_i$ and
$a_i-2\eta_*$ is larger than $\gamma$. Also, if
$p(A_i)\ge \gamma$ for every $i$, then
\eqref{eq:constants-near} gives
\begin{equation}\label{eq:sampling-holds}
  C_1\frac{\pmax}{\gamma}
  +C_2\frac{M}{\gamma^2}
  \le \frac{11rC_1}{100\,rC_1}
     +\frac{121r^2C_2}{1000\,r^2C_2}
   <1,
\end{equation}
so \eqref{eq:sampling-condition} holds and we
may use Lemma \ref{lem:weighted-sampling}.

For a vertex $v$ and an index $i$, we write
\[
  d_i(v):=p\bigl(N_G(v)\cap U_i\bigr)
\]
for the weight of the neighbors of $v$ in
$U_i$. We first prove three claims about the
structure of $G$.

\medskip
\noindent\textbf{Claim 1.}
Every vertex $v\in V(G)$ has an index
$i\in[r]$ with $d_i(v)\le \frac{1}{10}a_i$.

\begin{proof}[Proof of claim]
Suppose not. Then
$d_i(v)>\frac{1}{10}a_i\ge \frac{99}{1000r}
>\gamma$ for every $i$. Apply Lemma
\ref{lem:weighted-sampling} to the sets
$A_i:=N_G(v)\cap U_i$, with $s_1=t-1$ and
$s_i=t$ for $i\ne 1$; its hypothesis
\eqref{eq:sampling-condition} holds by
\eqref{eq:sampling-holds}. So we can pick $t-1$
vertices from $A_1$ and $t$ vertices from each
other $A_i$ so that any two picked vertices from
different sets are joined by an edge. Adding $v$ to
this selection, we find a copy of $K_r^+[t]$. This is impossible, since $G$ is
$K_r^+[t]$-free, and Claim 1 follows.
\end{proof}

For each $i\in[r]$, let
$ V_i:=\left\{v\in V(G):d_j(v)\ge \tfrac23 a_j
    \text{ for every }j\ne i\right\}$. 

\medskip
\noindent\textbf{Claim 2.}
The sets $V_1,\ldots,V_r$ are pairwise disjoint.

\begin{proof}[Proof of claim]
Suppose $v$ lies in $V_i$ and in $V_j$ for two
different indices $i,j$. Being in $V_j$ gives
$d_i(v)\ge \frac23a_i$, and being in $V_i$ gives
$d_j(v)\ge \frac23a_j$. For any other index $k\notin \{i,j\}$,
each of the two conditions gives
$d_k(v)\ge \frac23a_k$. So
$d_k(v)\ge \frac23a_k>\frac{1}{10}a_k$ for every
$k$, which Claim 1 rules out.
\end{proof}

\medskip
\noindent\textbf{Claim 3.}
Each set $V_i$ is an independent set of $G$.

\begin{proof}[Proof of claim]
Suppose $u,v\in V_i$ and $uv\in E(G)$. For every
$j\ne i$, the weight of the common neighbors of
$u$ and $v$ in $U_j$ is at least
\[
  d_j(u)+d_j(v)-a_j
  \ge \frac13a_j
  \ge \frac{33}{100r}
  >\gamma .
\]
Using the hypothesis on $\pmax$ and the bound in (\ref{eq:a-lower}), we get  
\[
  p\bigl(U_i\setminus\{u,v\}\bigr)
  \ge a_i-2\pmax
  \ge \frac{99}{100r}-\frac{1}{100r} 
  >\gamma .
\]
Use Lemma \ref{lem:weighted-sampling} with
$s_i=t-2$ and $A_i:=U_i\setminus\{u,v\}$, and,
for each $j\ne i$, with $s_j=t$ and
$A_j:=N_G(u)\cap N_G(v)\cap U_j$. Again
\eqref{eq:sampling-holds} gives
\eqref{eq:sampling-condition}. The picked
vertices, together with the edge $uv$, form a copy
of $K_r^+[t]$. This is again impossible, and Claim 3
follows. 
\end{proof}

In what follows, we set   
\[
  T:=V_1\cup\cdots\cup V_r,
  \qquad
  S:=V(G)\setminus T,
  \qquad
  X:=p(T).
\]
We first show that $p(S)$ is small. Let
$v\in S\cap U_i$ for some $i\in [r]$. Then $v\in S$ implies $v\notin V_i$, so
$d_j(v)<\frac23a_j$ for some $j\ne i$. Hence the
vertices of $U_j$ that are not neighbors of $v$
have weight more than
$\frac13a_j\ge \frac{33}{100r}$. 
For each $v\in S$, fix one such index $j=j(v)$.
Multiplying the last inequality by $p_v$ and summing  over $v\in S$, and noting that every missing cross-pair is counted at most twice, we obtain  
\[
  \frac{33}{100r}\,p(S) 
  <\sum_{v\in S}p_v\cdot 
    p\bigl(U_{j(v)}\setminus N_G(v)\bigr)
  \le 2M .
\]
By the hypothesis $M\le \rho_*$, the choice of
$\rho_*$ in \eqref{eq:constants-near} and
$C_2\ge 1$, this gives 
\begin{equation}\label{eq:s-small}
  p(S) <\frac{200r}{33}\,\rho_*
  \le \frac{200}{33\cdot 1000\,rC_2}
  \le \frac{1}{100r}.
\end{equation}

Now let $v\in S$. By Claim 1, there is an index
$i$ with $d_i(v)\le \frac{1}{10}a_i$, and since
$v\notin V_i$ there is a different index $j$
with $d_j(v)<\frac23a_j$. So the total weight of
the neighbors of $v$,
\[
p(N_G(v))=\sum_{k=1}^r d_k(v),
\]
satisfies, by \eqref{eq:two-part-deficit},
\[
p(N_G(v)) 
  \le 1-\frac{9}{10}a_i-\frac13a_j
  \le \left(1-\frac{1}{r}\right)-\frac{1}{10r}.
\]
Note that $X=p(T)=1-p(S)$. 
At the same time, \eqref{eq:s-small} gives
\[
  \left(1-\frac{1}{r}\right)X
  =\left(1-\frac{1}{r}\right)(1-p(S))
  \ge \left(1-\frac{1}{r}\right)- p(S)
  \ge \left(1-\frac{1}{r}\right)-\frac{1}{10r}.
\]
Comparing the last two lines, we get
\begin{equation}\label{eq:S-degree-bound}
   p(N_G(v)) \le \left(1-\frac{1}{r}\right) X
  \quad\text{for every }v\in S.
\end{equation}

 By Claim 3, each $V_i$ is an independent set. Write
$$\Phi_{T}(p) :=2\sum_{uv\in E(G[T])}p_up_v.$$ 
For each $i$, write $b_i:=p(V_i)$, so that
$\sum_i b_i=X$. Then
\begin{align}
\Phi_{T}(p)
  \le 2\sum_{1\le i<j\le r}b_ib_j
  =X^2-\sum_{i=1}^r b_i^2
  \le X^2-\frac{X^2}{r}
   =\left(1-\frac{1}{r}\right) X^2,
  \label{eq:T-contribution}
\end{align}
where the last step uses the Cauchy--Schwarz inequality.

Let
\[ Q_{ST}:=\sum_{\substack{uv\in E(G)\\
    u\in S,\ v\in T}}p_up_v,
  \qquad
  Q_{SS}:=\sum_{uv\in E(G[S])}p_up_v.
\]
Multiplying \eqref{eq:S-degree-bound} by $p_v$
and summing over $v\in S$ gives
\[
  Q_{ST}+2Q_{SS}
  =\sum_{v\in S}p_v\cdot p(N_G(v)) 
  \le p(S) \left(1-\frac{1}{r}\right) X ,
\]
and therefore
\begin{equation}\label{eq:S-total-contribution}
  2Q_{ST}+2Q_{SS}
  \le 2Q_{ST}+4Q_{SS}
  \le 2 p(S) \left(1-\frac{1}{r}\right) X.
\end{equation}
Adding \eqref{eq:T-contribution} and
\eqref{eq:S-total-contribution}, and using $p(S)=1-X$, we get
\begin{align*}
  \Phi_G(\bm{p})
  \le \left(1-\frac{1}{r}\right)
    \bigl(X^2+2X(1-X)\bigr) 
  \le \left(1-\frac{1}{r}\right).
\end{align*}
This proves the lemma.
\end{proof}

\subsection{Completing the proof of Theorem \ref{thm-second-key}}

We now combine Lemma \ref{lem:near-turan-weighted} with the weighted stability theorem of Section \ref{sec:weighted}.

\begin{proof}[{\bf Proof of Theorem
  \ref{thm-second-key}}]
Put $t:=v(F)$. By Lemma
\ref{lem:critical-template}, every $F$-free graph is
$K_r^+[t]$-free. Let $\varepsilon_*,\rho_*,\eta_*$
be the constants \eqref{eq:constants-near} formed
with this value of $t$, and recall
$\rho_*\le \varepsilon_*$. Let
$\eta_1=\eta (F,\rho_*)$ and
$\delta_1=\delta (F,\rho_*)$ be the constants given
by Theorem \ref{thm:weighted-stability} applied
with $\varepsilon :=\rho_*$. We claim that
\[
  \eta (F):=\min \{\eta_1,\ \eta_*\}
\]
has the required property.

Let $G$ be an $F$-free graph and let $\bm{p}$ be a
probability vector on $V(G)$ with
$\lVert \bm{p}\rVert_\infty \le \eta (F)$. 
Suppose for contradiction that
$\Phi_G(\bm{p})>1-\frac1r$. Then we have  
$\Phi_G(\bm{p})\ge 1-\frac1r-\delta_1$ and
$\lVert \bm{p}\rVert_\infty \le \eta_1$, so Theorem
\ref{thm:weighted-stability} supplies an $r$-partition
$V(G)=U_1\cup \cdots \cup U_r$ with
\[
  \Bigl|p(U_i)-\frac1r\Bigr|\le \rho_*
  \le \varepsilon_* \quad (\text{for every}~i\in [r]),
  \qquad
  M(G,\bm{p},U_1,\ldots ,U_r)\le \rho_* .
\]
Since $G$ is $K_r^+[t]$-free and
$\lVert \bm{p}\rVert_\infty \le \eta_*$, Lemma
\ref{lem:near-turan-weighted} applies to this
partition and gives $\Phi_G(\bm{p})\le 1-\frac1r$,
which is a contradiction. Thus, we must have $\Phi_G(\bm{p}) \le 1- \frac{1}{r}$, completing the proof. 
\end{proof}

\section{Applications of the bridge and the weighted theorems}

\label{sec:applications}

The framework of Sections \ref{sec:entropy} -- 
\ref{sec-prove-exact} has three modular parts: the entropy-Perron bridge, the weighted Tur\'{a}n theorem combined with it, 
and the flatness estimate that lets the two meet.
This section replaces each part in turn and derives the five applications announced in the introduction:  
(i) walks are replaced by homomorphism counts of trees (Theorem \ref{thm:spectral-tree});  
(ii) a fixed excess over $1-\frac1r$ in the walk-spectral setting is shown to force many copies of $F$ (Theorem \ref{thm:spectral-supersat});  
(iii) the vertex-spectral stability of Nikiforov is
recovered from the weighted theorems (Proposition
\ref{prop:vertex-stability});  
(iv) the entropic Tur\'an theorem is extended from cliques to color-critical forbidden graphs (Theorem \ref{thm:entropic-turan-critical}); 
(v) the resulting inequalities are compared with the walk-spectral Sidorenko inequalities (Theorem \ref{thm:walk-sidorenko}).

\subsection{Trees in place of walks}

\label{subsec:trees}

As promised, we now prove Theorem \ref{thm:tree-intro} with the explicit threshold.

\begin{theorem}\label{thm:spectral-tree}
Let $F$ be a color-critical graph with $\chi (F)=r+1\ge 3$, let $\eta =\eta (F)$ be the constant of Theorem \ref{thm-second-key}, and
let $T$ be an unbalanced tree on $\ell$ vertices. Put $c=c(T)$ and 
\[
  \lambda_0 :=\Bigl(\bigl(1-\tfrac1r\bigr)^{-1}
    \eta^{-c}\Bigr)^{1/(2c-\ell)}.
\]
Then every $F$-free graph $G$ with
$\lambda (G)\ge \lambda_0$ satisfies
\[
  \lambda^\ell (G)\le
  \Bigl(1-\frac1r\Bigr)\hom (T,G).
\]
\end{theorem}

The proof of Theorem \ref{thm:spectral-tree} uses the same three ingredients as the proofs in Section \ref{sec:main-proof}.
One of them is unchanged: the exact weighted bound of Theorem \ref{thm-second-key} involves no walks, and we use it verbatim. 
The other two are replaced by their tree analogues,
an entropy-Perron bridge for arbitrary trees (Theorem \ref{thm-tree-bridge}) and a flatness estimate for unbalanced trees  (Lemma \ref{lem:flat-tree}).

\begin{theorem} 
\label{thm-tree-bridge}
Let $G$ be a connected graph with at least one edge,
let $\bm{x}$ be its positive unit Perron vector, and put
$\pi_v=x_v^2$. Then, for every tree $T$ on $\ell$
vertices,
\[
  \lambda^\ell (G) \le \PhiG(\pi)\,\hom (T,G).
\]
\end{theorem}

\begin{proof}
Keep the notation of the proof of Theorem \ref{thm:entropy-bridge}: $\lambda =\lambda (G)$, $P_{uv}=a_{uv}x_v/(\lambda x_u)$ is a transition matrix with stationary distribution $\pi$, and $Q_{uv}=\pi_uP_{uv}=a_{uv}x_ux_v/\lambda$ is symmetric.
Root $T$ at a vertex $v_0$, and let $\phi: V(T)\to V(G)$ be the random map obtained by drawing $\phi_{v_0}$ from $\pi$ and then, going away from the root, drawing $\phi_j$ from the row $P_{\phi_i, *}$ whenever $i$ is the parent of $j$, written $i=\mathrm{par}(j)$. Thus, for every fixed map $f\colon V(T)\to V(G)$,
\[  \Pr (\phi =f) =\pi_{f(v_0)}\!\!\!
\prod_{{ij\in E(T),\, i=\mathrm{par}(j)}}\!\!\! P_{f(i)f(j)}
=\frac{\prod_{ij\in E(T)}Q_{f(i)f(j)}}{\prod_{i\in V(T)}\pi_{f(i)}^{\,d_T(i)-1}}, \]
by $Q_{uv}=\pi_uP_{uv}$, since each vertex $i$ is the parent of
$d_T(i)-1$ children when $i\ne v_0$, and of $d_T(v_0)$ children
when $i=v_0$. The right-hand side is symmetric in the two ends of every edge, so the distribution of $\phi$ does not depend on the choice of $v_0$. Every vertex of $T$ has marginal $\pi$, and every edge of $T$ has joint distribution $Q$.
Since $\pi$ is positive and $P_{uv}>0$ exactly when $uv\in E(G)$, a map $f$ has positive probability if
and only if it sends every edge of $T$ to an edge of
$G$. Hence, $\supp (\phi)$ is the set of homomorphisms
from $T$ to $G$, and (E1) gives
$\log \hom (T,G)\ge H(\phi)$.

List $V(T)=\{i_1,\ldots ,i_\ell\}$ with $i_1=v_0$ so
that every vertex is preceded by its parent. For
$t\ge 2$, the vertices $i_1,\ldots ,i_{t-1}$ lie
outside the subtree rooted at $i_t$, and their values
are generated without using the draw at $i_t$. Hence,
given the value at the parent of $i_t$, the value
$\phi_{i_t}$ is independent of
$\phi_{i_1},\ldots ,\phi_{i_{t-1}}$. 
As in Lemma \ref{lem:chain-entropy}, 
the chain rule (E2) gives
\[
  H(\phi)=H(\pi)+\sum_{j\ne v_0}
    H\bigl(\phi_j\mid \phi_{\mathrm{par}(j)}\bigr)
  =H(\pi)+(\ell -1)\log \lambda ,
\]
since each edge of $T$ carries the distribution $Q$, and the
computation in the proof of Theorem 
\ref{thm:entropy-bridge} shows that the corresponding
conditional entropy equals $\log \lambda$. 
Combining
this with $\log \hom (T,G)\ge H(\phi)$, the identity
\eqref{eq:step4} and $D(Q\, \Vert \, R)\ge 0$ gives
\[
  \log \hom (T,G)\ge \ell \log \lambda
    -\log \PhiG(\pi),
\]
and exponentiating finishes the proof.
\end{proof}

We also need a new flatness estimate for unbalanced trees.

\begin{lemma}
\label{lem:flat-tree}
Let $T$ be a tree on $\ell$ vertices, let
$c=c(T)$ be the size of the larger side of its
bipartition, and 
let $G$ be a connected graph with unit Perron vector $\bm{x}$. If
$\lambda^{\ell}(G)\ge c_0\hom (T,G)$ for some 
$c_0\in (0,1]$, then 
\begin{equation*}
  \max_{v\in V(G)}x_v^2  \le c_0^{-1/c}
    \lambda (G)^{-(2c-\ell )/c}. 
\end{equation*}
In particular, if $T$ is unbalanced, that is, 
$2c>\ell$, then $\max_{v\in V(G)} x_v^2$ tends to $0$ as $\lambda (G)\to \infty$.
\end{lemma}

\begin{proof}
Write $\lambda =\lambda (G)$, let $i$ be a vertex
with $x_i=\max_{v} \{x_v\}$, and put $d_i=|N_G(i)|$. Since
$\lVert \bm{x}\rVert =1$, the Cauchy--Schwarz
inequality gives
\[
  \lambda \, x_i =\sum_{j\sim i}x_j
  \le \sqrt{d_i}\Bigl(\sum_{j\sim i}x_j^2\Bigr)^{1/2}
  \le \sqrt{d_i},
\]
so $d_i\ge \lambda^2 x_i^2$. The star $K_{1,d_i}$ is a subgraph of $G$, so
Lemma \ref{lem:hom-star} gives
\[
  \hom (T,G)\ \ge\ \hom (T,K_{1,d_i})
  \ \ge\ d_i^{\,c}\ \ge\ \lambda^{2c} x_i^{2c}.
\]
Together with the hypothesis $\hom (T,G)\le 
c_0^{-1}\lambda^{\ell}$, 
this gives $\lambda^{2c}x_i^{2c} \le c_0^{-1} \lambda^{ \ell}$, as needed. The second
statement follows immediately, since the
exponent $-(2c-\ell )/c$ is negative.
\end{proof}

\begin{remark}
The hypothesis $2c>\ell$ serves to make the exponent of $\lambda$ negative so that $\pi=(x_v^2)_{v\in V(G)}$ is flat, and it also reappears in the proof of Theorem \ref{thm:spectral-tree}, where we take a $(2c - \ell)$-th root. For balanced trees, the
conclusion fails: for $T=P_{2k}$ and  $G=K_{1,d}$, 
we have $\hom (P_{2k},K_{1,d})=2d^{k}$ and
$\lambda^{2k}(K_{1,d})=d^{k}$, so
$\lambda^{2k} (K_{1,d}) \ge \frac12\hom (P_{2k},K_{1,d})$
while $\max_{v\in V(G)} x_v^2=\frac12$ does not decay.
\end{remark}

We are now in a position to prove Theorem \ref{thm:spectral-tree}. 

\begin{proof}[{\bf Proof of Theorem \ref{thm:spectral-tree}}]
Suppose that $\lambda^\ell (G)>(1-\frac1r)\hom (T,G)$. We choose a connected component $C$ of $G$ with $\lambda (C)=\lambda (G)$.
Since $\hom (T,C)\le \hom (T,G)$, we get
$\lambda^\ell (C) >(1-\frac1r)\hom (T,C)$, so $C$ has at least one edge. Let $\bm{x}$ be its unit Perron vector and
$\pi_v=x_v^2$. Theorem \ref{thm-tree-bridge} gives
$\Phi_C(\pi)>1-\frac1r$, so the contrapositive of
Theorem \ref{thm-second-key} yields
$$\max_{v\in V(C) }\pi_v>\eta.$$
 On the other hand, Lemma \ref{lem:flat-tree}
applied to $C$ with $c_0=1-\frac1r$ gives
$$ \max_{v\in V(C)} \pi_v 
 = \max_{v\in V(C)}  x_v^2
 \le c_0^{-1/c}\lambda(C)^{-(2c-\ell)/c} .$$
Combining this with $\max_v \pi_v>\eta$ yields
$\eta^{c}<c_0^{-1}\lambda^{-(2c-\ell)}(C)$, that is,
$\lambda^{2c-\ell} (C)<c_0^{-1}\eta^{-c}$. Since $T$ is unbalanced, i.e., $2c-\ell >0$, we may take the $(2c-\ell)$-th root
and obtain $\lambda (G) = \lambda (C)<\lambda_0$. Thus, every $F$-free graph $G$ with $\lambda^\ell (G)>(1-\frac1r)\hom (T,G)$ must satisfy $\lambda (G) < \lambda_0$.
\end{proof}

\subsection{A unified formula for spectral supersaturation}
\label{subsec:supersat}

Theorem \ref{thm:large-lambda} determines the exact walk-spectral bound for a color-critical forbidden graph, and it was proved by means of the weighted Tur\'{a}n theorem. 
In this subsection, we show that
any fixed relative excess  over $1-\frac1r$ forces at least 
$ \Omega(\lambda (G)^{v(F)})$ copies of the graph $F$. 
We point out that color-criticality is no longer needed: $F$ may be an
arbitrary graph with $\chi (F)=r+1\ge 3$.

\begin{theorem}[Walk-spectral supersaturation]
\label{thm:spectral-supersat}
Let $F$ be an arbitrary graph with $\chi (F)=r+1\ge 3$ and $f=v(F)$, and let $\varepsilon >0$. There exist $\lambda_0=\lambda_0(F,\varepsilon )$ and
$c=c(F,\varepsilon )>0$ such that if $G$ is a graph with
$\lambda (G)\ge \lambda_0$ and 
\[
  \lambda^{\ell}(G)\ \ge\ \Bigl(1-\frac1r+\varepsilon \Bigr)w_\ell (G)
\]
for some integer $\ell \ge 1$, then $G$ contains at least
$c\,\lambda (G)^{f}$ copies of $F$.
\end{theorem}

The case $\ell =1$ extends a result of Bollob\'{a}s and Nikiforov \cite{BN2007}, who proved the supersaturation for cliques.  
The case $\ell =2$ extends a recent result of Fang, Lin and Zhai \cite[Theorem 1.2]{FLZ2026}.

\begin{proof}
Let $\eta =\eta (F,\varepsilon )$ and $c_0=c_0(F,\varepsilon )$ be the constants
of Theorem \ref{thm:weighted-supersat}. 
Put
$\delta :=\min\{\varepsilon ,\tfrac1{20}\}$, 
and set
$\lambda_0:=\delta^{-8} /\eta$ and $c:=\delta^{8f}c_0$. Let $C$ be a connected component of $G$
with $\lambda (C)=\lambda (G)$. Since
$\lambda (C)=\lambda (G)\ge \lambda_0>0$, the component $C$ has an
edge. As $w_\ell (C)\le w_\ell (G)$, we get
$\lambda^{\ell}(C)\ge (1-\frac1r+\varepsilon )w_\ell (C)$. Let $\bm{x}=(x_v)_{v\in V(C)}$ be
the positive unit Perron vector of $C$, and let $\pi_v=x_v^2$ for every $v\in V(C)$. 
The entropy-Perron bridge in Theorem
\ref{thm:bridge-intro} gives
\[
  \Phi_C(\pi )\ \ge\ \frac{\lambda^{\ell}(C)}{w_\ell (C)}
  \ \ge\ 1-\frac1r+\varepsilon .
\] 
Since $1-\frac1r+\varepsilon \ge \frac12+\delta$ and
$\delta \in (0,0.1)$, the flatness estimate of Lemma \ref{lem:perron-flatness} yields
$$\max_{v\in V(C)}\pi_v< \delta^{-8}\lambda (C)^{-1} \le \delta^{-8}\lambda_0^{-1}=\eta.$$
 Applying Theorem
\ref{thm:weighted-supersat} to $C$ with the probability vector $\pi$, we get 
$N_F(C,\pi )\ge c_0$. Finally, every copy $F^*$ of $F$ in $C$ satisfies
$\prod_{v\in V(F^*)}\pi_v\le \bigl(\max_{v\in V(C)}\pi_v\bigr)^{f}
 \le \bigl(\delta^{-8}\lambda (C)^{-1}\bigr)^{f}$. Hence
\[
  N_F(G)\ \ge\ N_F(C)\ \ge\
  N_F(C,\pi )\left(\delta^8{\lambda (C)}\right)^{f}
  \ \ge\ \delta^{8f}c_0  \cdot \lambda (G)^{f},
\]
as required.
\end{proof}

\begin{remark}
This bound is tight up to a constant factor, that is, exponent $f$ is best possible. The graph $G=K_{r+1}[s]$ is
$rs$-regular on $(r+1)s$ vertices, so $\lambda (G)=rs$ and
$\lambda^{\ell}(G)/w_\ell (G)=\frac{r}{r+1}
=1-\frac1r+\frac{1}{r(r+1)}$ for every $\ell \ge 1$, while
$N_F(G)\le ((r+1)s)^{f}=(\frac{r+1}{r})^{f}\lambda (G)^{f}$.
Thus, for every $\varepsilon \le \frac{1}{r(r+1)}$, the hypothesis is satisfied
for every $\ell$. Thus, the conclusion is tight up to the value of $c$. 
\end{remark}

\subsection{Spectral stability via the weighted theorems}
\label{subsec:stability-app}

Supersaturation and stability are the two companions
of Tur\'an's theorem, and the weighted framework supplies both.
The weighted theorems of this paper 
allow any probability vector, 
and already the Perron weight  
$\pi$ suffices to recover Nikiforov's vertex-spectral stability theorem \cite{Niki2009jgt}.

\begin{proposition}\label{prop:vertex-stability}
Let $F$ be a graph with $\chi (F)=r+1\ge 3$.
For every $\varepsilon >0$, there are
$\delta =\delta (F,\varepsilon )>0$ and
$n_2=n_2(F,\varepsilon )$ such that 
every $F$-free graph $G$ on $n\ge n_2$ vertices with
$\lambda (G)\ge \bigl(1-\frac1r-\delta \bigr)n$ satisfies
$e(G)\ge \bigl(1-\frac1r-\varepsilon \bigr)\frac{n^{2}}{2}$  and differs from the Tur\'{a}n graph $T_{n,r}$ in at most $\varepsilon n^{2}$ edges.
\end{proposition}

\begin{proof}
We may assume $\varepsilon \le 0.1$ and write
$\kappa :=1-\frac1r\ge \frac12$ and $m:=e(G)$. Put
$\alpha :=\frac{\varepsilon}{2r+3}$, let $\theta$ and $N_1$
be the constants of Theorem \ref{thm:stability} for the
parameter $\alpha$, and set
$\beta :=\min \{\alpha ,\theta \}$ and $\delta :=\beta /8$.
Finally, let $\eta =\eta (F,\beta /4)$ be the constant of
Theorem \ref{thm-weighted-ESS}, and let $n_2$ exceed
$8/\eta$, $r/\alpha$ and $N_1$.

Let $\bm{x}$ be a nonnegative unit Perron vector of $G$ and
$\pi_v=x_v^{2}$. Since $\lambda x_u=\sum_{v\sim u}x_v\le
\sqrt{d_u}(\sum_{v\sim u}x_v^{2})^{1/2}\le \sqrt n$ and
$\lambda \ge (\kappa -\delta )n\ge \frac38 n$, we get
$\lVert \pi \rVert_\infty \le n/\lambda^{2}\le 8/n\le \eta$.
Hence Theorem \ref{thm-weighted-ESS} gives
$\Phi_G(\pi )\le \kappa +\frac{\beta}{4}$, while the
Cauchy--Schwarz inequality (or Theorem
\ref{thm:bridge-intro}) gives
\[
  \lambda^{2}=\bigg(\sum_{uv\in E(G)}2x_ux_v\bigg)^{2}
  \le 2m \cdot \sum_{uv\in E(G)}2x_u^{2}x_v^{2}
  =\Phi_G(\pi )\cdot 2m
  \ \le\ \Bigl(\kappa +\frac{\beta}{4}\Bigr)2m .
\]
We read this display in both directions. Together with
$\lambda \ge (\kappa -\delta )n$ it gives
\[
  m\ \ge\ \frac{(\kappa -\delta )^{2}}{\kappa +\beta /4}
  \,\frac{n^{2}}{2}\ \ge\ (\kappa -\beta ) \frac{n^{2}}{2}
  \ \ge\ \Bigl(1-\frac1r-\varepsilon \Bigr) \frac{n^{2}}{2},
\]
where the second inequality holds by $(\kappa -\delta )^{2}\ge \kappa^{2}-2\delta
=\kappa^{2}-\frac{\beta}{4}$ and $(\kappa +\frac{\beta}{4}) (\kappa -\beta ) \le \kappa^{2}-\frac38\beta$; this
is the first assertion. Together with $\lambda \ge 2m/n$, it
gives $4m^{2}/n^{2}\le \bigl(\kappa +\frac{\beta}{4}
\bigr)2m$, that is,
\[
  m\ \le\ \Bigl(\kappa +\frac{\beta}{4}\Bigr)
  \frac{n^{2}}{2}.
\]

As $\beta \le \theta$, Theorem \ref{thm:stability} applies
with parameter $\alpha$ and supplies a partition
$V(G)=B_1\cup \cdots \cup B_r$ with
$\bigl\lvert |B_i|-\frac nr\bigr\rvert \le \alpha n$ and
with at most $\alpha n^{2}$ non-adjacent cross-pairs. Let
$K$ be the complete $r$-partite graph on $V(G)$ with these
parts, so that $|E(K)\setminus E(G)|\le \alpha n^{2}$. From
$\sum_{i=1}^r |B_i|^{2}\le \frac{n^{2}}{r}+r\alpha^{2}n^{2}$, 
we get $e(K)\ge \kappa \frac{n^{2}}{2}
-\frac{r\alpha^{2}}{2}n^{2}$. 
Hence, the edges of $G$
inside the parts number
\[
  |E(G)\setminus E(K)|\le m-e(K)+|E(K)\setminus E(G)|
  \le \Bigl(\frac{\beta}{8}+\frac{r\alpha^{2}}{2}
  +\alpha \Bigr)n^{2}\le 2\alpha n^{2},
\]
because $\beta \le \alpha$ and $r\alpha \le 1$. Thus
$|E(G)\, \triangle \, E(K)|\le 3\alpha n^{2}$. Finally, moving at most $r(\alpha n+1)\le 2r\alpha n$ vertices between the parts turns $K$ into the 
Tur\'{a}n graph $T_{n,r}$, 
and each move changes at most $n$ pairs. Therefore, we have 
$|E(G)\, \triangle \, E(T_{n,r})|\le 3\alpha n^{2}+2r\alpha n^{2} =\varepsilon n^{2}$.
\end{proof}

As a further application, we give a new proof of the
edge-spectral stability theorem of Li, Liu and Zhang
\cite[Theorem 1.6]{LLZ-ESS} through the weighted stability. The proof uses no spectral
stability argument: its only spectral inputs are the
$\ell =2$ case of Theorem \ref{thm:bridge-intro} 
and the flatness estimate of Lemma
\ref{lem:perron-flatness}. 
 Indeed, if $G$ is an $F$-free graph with 
  $\lambda^{2}(G)\ge (1-\frac1r-\delta )\,2m$, then Theorem
\ref{thm:bridge-intro} gives $\Phi_G(\pi )\ge \lambda^2(G)/ w_2(G)\ge 1-\frac1r-\delta$ with the probability vector $\pi_v=x_v^{2}$. 
Lemma \ref{lem:perron-flatness} makes $\pi$
flat, so $\lVert \pi \rVert_{\infty} \le \eta$ and Theorem \ref{thm:weighted-stability} yields a partition
$V(G)=U_1\cup \cdots \cup U_r$ whose parts and missing cross-pairs are controlled in weight. 
What remains is the transition from vertex weights to edge counts. This transition is somewhat technical, and we defer the detailed deduction to Appendix \ref{sec:appendix-stability}.

\subsection{Entropic Tur\'{a}n theorem for color-critical graphs}

\label{subsec:entropic-turan}

The three applications above all 
run from entropy bridge to spectral extremal results. 
This subsection runs in a different direction: 
instead of using entropy to prove
the spectral inequalities, 
we use Theorem \ref{thm-second-key} to prove an entropic result.  
The comparison carried out in Remark \ref{rem:compare-CY} used no
property of the Perron pair beyond the fact that its two marginals coincide. 
Running it for an arbitrary symmetric pair extends the entropic Tur\'{a}n
theorem of Chao--Yu \eqref{eq-Chao-Yu-entropic} from $K_{r+1}$-free graphs to $F$-free graphs with $F$ color-critical, 
at the cost of a flatness hypothesis on the marginal. 
As in Section \ref{sec:entropy}, we write 
$E^{*}(G)=\{(u,v):uv\in E(G)\}$ for the set of ordered edges of $G$.

\begin{theorem}\label{thm:entropic-turan-critical}
For every color-critical graph $F$ 
with $\chi (F)=r+1\ge 3$, 
there exists $\eta (F) >0$ such that if $G$ is an
$F$-free graph, $(X,Y)$ is a random symmetric pair of vertices of $G$ with $XY\in E(G)$, and the distribution $\bm{p}$ of $X$ satisfies $\pmax \le \eta (F)$, then
\[
  H(X,Y)\ \le\ 2H(X)+\log \Bigl(1-\frac1r\Bigr),
\]
with equality if and only if $\Phi_G(\bm{p})=1-\frac1r$ and
\[ \Pr \bigl((X,Y)=(u,v)\bigr)
   =\frac{p_up_v}{1-\frac1r}
   \qquad \text{for every } (u,v)\in E^{*}(G). \] 
\end{theorem}

\begin{proof}
Let $\eta (F)>0$ be the constant from Theorem \ref{thm-second-key}. 
We may assume that $E(G)$ is non-empty. 
By symmetry, $Y$ has distribution $\bm{p}$ as well. 
Let $Q_{uv}:=\Pr \big( (X,Y)=(u,v) \big)$, which is a probability distribution supported on the ordered edge set $E^{*}(G)$. We define
\[
  R_{uv}:=\frac{a_{uv}\,p_up_v}{\Phi_G(\bm{p})}\qquad
  \text{for every}\quad (u,v)\in E^{*}(G) ,
\]
which is a probability distribution on $E^{*}(G)$ because
$\sum_{u,v}a_{uv}p_up_v=\Phi_G(\bm{p})$. 
If $Q_{uv}>0$, then $p_u,p_v>0$
and hence $R_{uv}>0$, so $D(Q\, \Vert \, R)$ is defined. Since both marginals of $Q$
equal $\bm{p}$,
\begin{align*}
  D(Q\, \Vert \, R)
  &=\sum_{(u,v)\in E^{*}(G)}Q_{uv}
      \log \frac{Q_{uv}\,\Phi_G(\bm{p})}{p_up_v}\\
  &=-H(X,Y)+\log \Phi_G(\bm{p})
    -\sum_{(u,v)\in E^{*}(G)}Q_{uv}\log p_u
    -\sum_{(u,v)\in E^{*}(G)}Q_{uv}\log p_v\\
  &=2H(X)-H(X,Y)+\log \Phi_G(\bm{p}).
\end{align*}
By Gibbs' inequality (Lemma \ref{lem-Gibb}), the left-hand side is nonnegative,
so $H(X,Y)-2H(X)\le \log \Phi_G(\bm{p})$. 
 Now Theorem \ref{thm-second-key} gives 
 $\Phi_G(\bm{p})\le 1-\frac1r$, as desired. 
Moreover, equality forces $Q=R$ and
$\Phi_G(\bm{p})=1-\frac1r$. 
Conversely, these two conditions give equality by a direct computation. 
\end{proof}

\begin{remark}\label{rem:entropic-necessity}
The bound is best possible. 
Let $t\ge 1/(r\,\eta (F))$ and let $G=K_r[t]$. 
Take $\bm{p}$ uniform on $V(G)$ and
$(X,Y)$ uniform on $E^{*}(G)$,  
we see that $\pmax =1/(rt)\le \eta (F)$,
$H(X,Y)=\log \bigl(r(r-1)t^{2}\bigr)$ and $H(X)=\log (rt)$, so equality holds. As in Remark \ref{rem:necessity}, 
the flatness hypothesis $\pmax \! \le  \eta$ cannot be omitted. 
 Let $F$ be color-critical with $\chi (F)=r+1$ 
 and with $f \ge r+2$ vertices, 
and let $G=K_{f -1}$, which is $F$-free. 
With both $\bm{p}$ and $(X,Y)$ being uniform, 
we get  $H(X,Y)-2H(X)=\log (1-\frac{1}{f -1} )>\log  (1-\frac1r )$. Hence, the conclusion fails, and 
$\eta (F)\le \frac{1}{f -1}$ in
Theorem \ref{thm:entropic-turan-critical}. \end{remark}

Combining Theorem \ref{thm:entropic-turan-critical} 
with Theorem \ref{thm:entropy-bridge} yields an entropic refinement of Corollary \ref{cor:perron-wilf}.

\begin{corollary}\label{cor:entropic-perron-wilf}
Let $F$ be a color-critical graph with $\chi (F)=r+1\ge 3$, and let $G$ be a
connected $F$-free graph with at least one edge, with unit Perron vector
$\bm{x}$ and $\pi_v=x_v^{2}$. If $\lVert \pi \rVert_\infty \le \eta (F)$, then
\[
  \lambda (G)\ \le\ \Bigl(1-\frac1r\Bigr)2^{H(\pi )}.
\]
\end{corollary}

\begin{proof}
Let $(X,Y)$ be the random pair of vertices with 
$\Pr ((X,Y)=(u,v))= {a_{uv}x_ux_v}/{\lambda (G)}$. 
It is symmetric, both of its marginals equal $\pi$. 
The proof of Theorem
\ref{thm:entropy-bridge} shows that $H(Y| X)=\log \lambda (G)$, so $H(X,Y)=H(\pi )+\log \lambda (G)$. Theorem \ref{thm:entropic-turan-critical} now 
 gives $H(\pi )+\log \lambda (G)\le 2H(\pi )+\log (1-\frac1r)$.
\end{proof}

\begin{remark}\label{rem:entropic-vs-wilf}
Applying Jensen's inequality gives $\sum_v x_v^{2}\log \frac{1}{x_v}\le \log \sum_v x_v$, that is,
$2^{H(\pi )}\le \lVert \bm{x}\rVert_1^{2}$, with equality if and only if
$\bm{x}$ is constant. Together with
$\lVert \bm{x}\rVert_1^{2}\le n $, this gives the chain
\[
  \lambda (G)\ \le\ \Bigl(1-\frac1r\Bigr)2^{H(\pi )}
  \ \le\ \Bigl(1-\frac1r\Bigr)\lVert \bm{x}\rVert_1^{2}
  \ \le\ \Bigl(1-\frac1r\Bigr)n . 
\]
The equality cases of Corollary \ref{cor:entropic-perron-wilf} are richer than those of 
Corollary \ref{cor:perron-wilf}: for $r=2$,  
every $K_{a,b}$ attains equality, whereas $K_{a,b}$ attains
equality in Corollary \ref{cor:perron-wilf} only when $a=b$.
\end{remark}

\subsection{The walk-spectral Sidorenko inequalities}
\label{subsec:sidorenko}

Theorem \ref{thm:spectral-tree} can be read as a lower bound on a homomorphism count in terms of the spectral radius under a Tur\'{a}n-type hypothesis: if $F$ is color-critical with $\chi (F)=r+1\ge 3$ and $T$ is an unbalanced tree on $\ell$ vertices, then every $F$-free graph $G$
of large spectral radius satisfies
\begin{equation} \label{eq-bound-tree-hom}
  \hom (T,G)\ \ge\ \frac{r}{r-1}\,\lambda^{\ell}(G).
\end{equation} 

There is a parallel line of work in which lower
bounds are proved with no forbidden subgraph, at the price of letting the size of $G$ enter as well. Let $H$ be a graph with $v$ vertices and $e$ edges, and recall that Sidorenko's conjecture \cite{Sid1993} predicts
$\hom (H,G)\ge (2m)^{e}n^{v-2e}$ for every
bipartite $H$ and every graph $G$ with $n$
vertices and $m$ edges; see \cite{Szegedy2014, KLL2016, CKLL2018} for background. 
Recently, Li, Lin, Liu and Zhang \cite{LLLZ-Sidorenko} proved that, in the range $v\le e$, this is equivalent to the spectral strengthenings:  
\begin{equation} 
\label{eq:vertex-spectral-sid}
\hom (H,G)\ge \lambda^{e}|G|^{\,v-e}
\end{equation}
and 
\begin{equation}\label{eq:edge-spectral-sid}
  \hom (H,G)\ \ge\
  \lambda (G)^{2e-v}(2m)^{v-e}.
\end{equation}
Inequality \eqref{eq:edge-spectral-sid} was used in \cite{LLLZ-Sidorenko} to
obtain edge-spectral supersaturation results for $K_{t,t}$ and $C_{2t}$.

All three are lower bounds on $\hom (H,G)$ in terms of $\lambda (G)$ and the
order or the size of $G$, but they differ in hypothesis and in range:
\eqref{eq-bound-tree-hom} needs $G$ to be $F$-free and concerns trees, where
$v=e+1$, whereas \eqref{eq:vertex-spectral-sid} and
\eqref{eq:edge-spectral-sid} need no forbidden subgraph but require $v\le e$.
The former is calibrated against complete multipartite graphs, the latter two
against quasirandom graphs. We now show that at $v=e+1$ the former calibration
is the correct one.

A tree $T$ on $\ell$ vertices has $e=v-1$. 
Thus, trees lie outside the range of
\eqref{eq:edge-spectral-sid}, where the
right-hand side would read $2m\,\lambda^{\ell-2}$. 
Inequality \eqref{eq:edge-spectral-sid} fails at  $H=P_3$ and $G=P_6$, where
$\hom (P_3,G)=18$ while
$2m\,\lambda (G)=10\cdot 2\cos \frac{\pi}{7}
\approx 18.019$.   
We now determine what replaces it at $v=e+1$.

\begin{theorem}[Chao--Yu \cite{CY2026}]  
\label{thm:tree-Chao-Yu} 
Let $T$ be an arbitrary tree on $\ell $ vertices and let
$G$ be a graph with clique number $\omega \ge 2$. Then
\begin{equation*} 
\hom (T,G) \ \ge \ \frac{\omega}{\omega -1} \, 
\lambda^{\ell}(G), 
\end{equation*}
with equality when $G$ is a regular complete
$\omega$-partite graph.
\end{theorem}

This result was first proved by Chao and Yu \cite{CY2026} using an entropic Tur\'{a}n theorem (\ref{eq-Chao-Yu-entropic}). In the following, we give an alternative proof based on the entropy-Perron bridge of Theorem \ref{thm-tree-bridge}.

\begin{proof}[Alternative proof]
Let $C$ be a component with $\lambda (C)=\lambda (G)$. 
Then  $2\le \omega (C)\le \omega (G)$. Since $s\mapsto \frac{s}{s-1}$ is decreasing on
$s\in [2,\infty )$ and $\hom (T,C)\le \hom (T,G)$, we see that the assertion for $C$ implies
the assertion for $G$. 
Without loss of generality, 
we may assume that $G$ is connected. 
Let $\bm{x}$ be its unit Perron vector and $\pi_v=x_v^2$. Theorem  \ref{thm-tree-bridge} gives
$$\lambda^{\ell}(G)\le \Phi_G(\pi ) \cdot \hom (T,G).$$ 
The Motzkin--Straus theorem gives
$\Phi_G(\pi )\le 1-\frac{1}{\omega}.$ 
Combining the two bounds yields
$\hom (T,G)\ge \frac{\omega}{\omega -1}\lambda^{\ell}(G)$, as
required.  
When $G$ is a regular complete $\omega$-partite graph, say  $G=K_{\omega}[s]$, one has $\lambda =(\omega -1)s$
and $\hom (T,G) =\omega s ((\omega -1)s )^{\ell -1}$, so
the desired bound is an equality.
\end{proof}

At $v=e+1$, the term  $2m\,\lambda^{\ell-2}$ of (\ref{eq:edge-spectral-sid}) should be replaced by $\frac{\omega }{\omega -1}\,
  \lambda^{\ell}(G)$.   
In fact,  Nikiforov's bound \eqref{eq-Nik-2002} gives
$\lambda^2(G)\le \frac{\omega-1}{\omega }2m$, which yields  $  \frac{\omega }{\omega -1}\,
  \lambda^{\ell}(G) \le 2m\,\lambda^{\ell -2}(G)$ 
for every graph $G$.

\smallskip 
We now prove the following extension to forests without isolated vertices. 

\begin{proposition} 
\label{prop:forest-conj}
Let $H$ be a forest on $v$ vertices with $k$ components and without
isolated vertices. Then for every $r\ge 2$ and every $K_{r+1}$-free
graph $G$,
\[
  \lambda^{v}(G) \ \le\
  \Bigl(1-\frac1r\Bigr)^{\! k} \hom (H,G),
\]
with equality when $G$ is a regular complete $r$-partite graph.
\end{proposition}

\begin{proof}
We may assume that $G$ has at least one edge, so 
 $\omega (G)\ge 2$, which is
the hypothesis of Theorem \ref{thm:tree-Chao-Yu}, and
$\omega (G)\le r$ since $G$ is $K_{r+1}$-free. 
Write $H=T_1\cup \cdots \cup T_k$, where each $T_i$ is a tree on $v_i\ge 2$ vertices and $v_1+\cdots +v_k =v$. Homomorphism counts are
multiplicative over the components, so
$\hom (H,G)=\prod_{i=1}^{k}\hom (T_i,G)$. 
Theorem \ref{thm:tree-Chao-Yu} applied to each component $T_i$ gives
\[
  \lambda^{v}(G)=\prod_{i=1}^{k}\lambda^{v_i}(G)
  \ \le\ \prod_{i=1}^{k}\Bigl(1-\frac{1}{\omega (G)}\Bigr)\hom (T_i,G)
  \ \le\ \Bigl(1-\frac1r\Bigr)^{\! k}\hom (H,G).
\]
For $G=K_r[s]$ one has $\lambda (G)=(r-1)s$ and
$\hom (T_i,G)=rs\bigl((r-1)s\bigr)^{v_i-1}$, so 
the equality holds.
\end{proof}

\smallskip 
A natural question can be asked with
$2m$ replaced by the walk count $w_\ell (G)$,
which interpolates between the vertex-spectral
form ($\ell =1$) and the
edge-spectral form ($\ell =2$). 
To start with, we need to introduce a spectral
regularization of Li, Lin, Liu and Zhang \cite[Theorem 1.2]{LLLZ-Sidorenko}.

\begin{theorem}[Li--Lin--Liu--Zhang \cite{LLLZ-Sidorenko}]
\label{lem:spectral-regularization}
Let $G$ be a graph with at least one edge, and put
$\lambda =\lambda (G)$. There are a constant
$C=C(G)>0$ and infinitely many integers $k$ such
that the tensor power $G^{\otimes k}$ contains a
regular subgraph $F_k$ of degree $d_k$ satisfying 
$\lambda^{k}k^{-C}\le d_k\le \lambda^{k}$. 
\end{theorem}

Now, we give a unified extension of both (\ref{eq:vertex-spectral-sid}) and (\ref{eq:edge-spectral-sid}). 

\begin{theorem}[Walk-spectral Sidorenko inequality]
\label{thm:walk-sidorenko}
Let $H$ be a bipartite graph with $v$ vertices and
$e$ edges, without isolated vertices, and put 
$\delta :=e-v\ge 0$. Let $\ell \ge 1$ be an integer. Then
\begin{equation}\label{eq:walk-sidorenko}
  \hom (H,G)\ \ge\
  \lambda (G)^{\,e+(\ell -1)\delta}\,
  w_\ell (G)^{-\delta} 
\end{equation}
holds for every graph $G$ 
 if and only if $H$ satisfies Sidorenko's
inequality.
\end{theorem}

\begin{proof}
Suppose first that \eqref{eq:walk-sidorenko}
holds. Let $n:=|G|$ and let $\bm{u}_1, \ldots ,\bm{u}_n$ be an orthonormal eigenvector basis. 
We write $\bm{1}=\sum_{i=1}^{n} c_i\bm{u}_i$ for some values $c_1, \ldots ,c_n$. Then 
$$w_\ell (G)=\sum_{i=1}^n c_i^2\lambda_i^{\ell -1}
 \le \lambda^{\ell -1}\sum_{i=1}^n  c_i^{2}
 = n \lambda^{\ell -1}.$$
  Since $\delta \ge 0$, the
right-hand side of \eqref{eq:walk-sidorenko} is at least $\lambda^{e} n^{v-e}$; 
this is the vertex-spectral version in (\ref{eq:vertex-spectral-sid}). 
Using $\lambda \ge 2e(G)/ n$, this is at least
$(2e(G))^{e} n^{v-2e}$, which is Sidorenko's
inequality.  

Conversely, suppose that $H$ is Sidorenko and let
$G$ be a graph with at least one edge. 
By the spectral regularization in Theorem \ref{lem:spectral-regularization}, 
there are $C=C(G)>0$ and
infinitely many $k$ for which $G^{\otimes k}$
contains a $d_k$-regular subgraph $F_k$ with
$\lambda^{k}k^{-C}\le d_k\le \lambda^{k}$. Put
$n_k:=|F_k|$. Since $A(G^{\otimes k})
 =A(G)^{\otimes k}$ we have
$w_\ell (G^{\otimes k})=w_\ell (G)^{k}$, and
$w_\ell$ is monotone under taking subgraphs, so
\[
  n_k\,d_k^{\ell -1}=w_\ell (F_k)
  \le w_\ell (G^{\otimes k})=w_\ell (G)^{k},
\]
the first equality because $F_k$ is $d_k$-regular.
Applying Sidorenko's inequality to $F_k$, and
using $2e(F_k)=n_kd_k$ together with
$n_k^{-\delta}\ge
 \bigl(w_\ell (G)^{k}\bigr)^{-\delta}
 d_k^{(\ell -1)\delta}$,  we get
\[
  \hom (H,F_k)\ \ge\ n_k^{\,v-2e} \big(2e(F_k) \big)^{e}
  =n_k^{-\delta}d_k^{\,e}
  \ \ge\ \frac{d_k^{\,e+(\ell -1)\delta}}
              {w_\ell (G)^{k\delta}} .
\]
Finally $\hom (H,G)^{k}
 =\hom (H,G^{\otimes k})\ge \hom (H,F_k)$, so
$d_k\ge \lambda^{k}k^{-C}$ gives
\[
  \hom (H,G)\ \ge\
  \frac{\lambda^{\,e+(\ell -1)\delta}}
       {w_\ell (G)^{\delta}}\,
  k^{-C(e+(\ell -1)\delta)/k},
\]
and letting $k\to \infty$ along the infinite set
above yields \eqref{eq:walk-sidorenko}.
\end{proof}

Although each of these inequalities is equivalent to Sidorenko's inequality for $H$, they are different inequalities for different $\ell$, and letting $\ell \to \infty$ gives a further strengthening.

\begin{corollary}\label{cor:perron-sidorenko}
Let $H$ be a bipartite graph satisfying Sidorenko's
inequality, with $v\le e$ and no isolated vertices. 
Then every connected non-bipartite graph $G$ with unit Perron
vector $\bm{x}$ satisfies
\[
  \hom (H,G)\ \ge\ \lambda (G)^{e}\,
  \lVert \bm{x}\rVert_1^{\,2(v-e)} .
\]
\end{corollary}

\begin{proof}
Since $G$ is connected and non-bipartite, we have $|\lambda_i|<\lambda$ for
every $i\ge 2$, and hence
$w_\ell (G)=\sum_i c_i^{2}\lambda_i^{\ell -1}
 =\lVert \bm{x}\rVert_1^{2}\lambda^{\ell -1}\bigl(1+o(1)\bigr)$
as $\ell \to \infty$. Writing $\delta =e-v\ge 0$ and applying
\eqref{eq:walk-sidorenko} for every $\ell \ge 1$, we obtain
\[
  \hom (H,G)\ \ge\ \lambda^{e}
  \left(\frac{\lambda^{\ell -1}}{w_\ell (G)}\right)^{\!\delta}
  =\lambda^{e}\bigl(\lVert \bm{x}\rVert_1^{-2}+o(1)\bigr)^{\delta}.
\]
Letting $\ell \to \infty$ gives
$\hom (H,G)\ge \lambda^{e}\lVert \bm{x}\rVert_1^{-2\delta}$,
which is the assertion.
\end{proof}

Corollary \ref{cor:perron-sidorenko} improves the
vertex-spectral form in (\ref{eq:vertex-spectral-sid}) since $\lVert \bm{x}\rVert_1^{2}\le |G|$. 
Observe that the proof of the sufficiency direction of 
Theorem \ref{thm:walk-sidorenko} uses only three properties of $w_\ell$: 
the multiplicativity under tensor powers, 
the monotonicity under subgraphs,
and the value $|G|\,d^{\ell -1}$ on $d$-regular graphs. 
Tree homomorphism counts share all three. 
Hence, we obtain a tree analogue: 
for every tree $T$ on $\ell$ vertices and every
 bipartite graph $H$ with $\delta =e-v\ge 0$ and no
isolated vertices, if $H$ satisfies the Sidorenko inequality, then 
$  \hom (H,G) \ge \lambda (G)^{\,e+(\ell -1)\delta}\,
  \hom (T,G)^{-\delta}$
 for every graph $G$.

\section{Concluding remarks}
\label{sec:remarks}

Theorem \ref{thm:large-lambda} determines the exact spectral constant for
counting walks (homomorphisms of paths) in $F$-free graphs with $F$ color-critical, and Theorem \ref{thm:spectral-tree} extends it to unbalanced trees. 
One natural Tur\'{a}n-type problem remains: 
for a general forbidden graph $F$, 
what does the constant become when walks are replaced by homomorphism counts of a general bipartite graph $H$?

\subsection{A general limit superior problem}

\begin{problem}  \label{prob-limsup}
Let $F$ be a graph and let $H$ be a bipartite graph. Determine 
\begin{equation*}
   \limsup_{\lambda (G)\to \infty} \ \frac{\lambda^{v(H)}(G)}{\hom (H,G)}, 
   \end{equation*}
where $G$ runs over all $F$-free graphs. 
\end{problem}

In this language, some known results read as follows.
When $F$ is a color-critical graph with
$\chi (F)=r+1\ge 3$, and $H=P_\ell$ with $(r,\ell) \neq (2,2)$, Theorem \ref{thm:large-lambda} says that the limit superior in
Problem \ref{prob-limsup} equals $1-\frac{1}{r}$, and this value is attained by
the regular complete $r$-partite graphs. Moreover, Theorem \ref{thm:spectral-tree} gives the same conclusion when $H$ is an unbalanced tree.  
A related gap concerns the case of balanced trees in Theorem \ref{thm:spectral-tree}, and
it is open for every $r\ge 2$. When $2c(T)=\ell$, Lemma \ref{lem:flat-tree}
gives no decay, and Lemma \ref{lem:perron-flatness} is stated for walks. What is missing is exactly the following
tree analogue: whether
$\lambda^{\ell}(G)\ge \bigl(\frac12+\delta \bigr)\hom (T,G)$ forces
$\max_{v\in V(G)}x_v^{2}=O_{\delta ,T}\bigl(\lambda^{-1}(G)\bigr)$. 
A positive answer would extend Theorem \ref{thm:spectral-tree} to all trees when
$\chi (F)=r+1 \ge 4$.

\smallskip 
 It is interesting to determine the limit superior for other pairs $(H, F)$. In particular, for $F=K_{r+1}$, 
one may conjecture that for a fixed bipartite graph $H$, if $r$ is sufficiently large, then the regular complete $r$-partite graphs attain the limit superior of $\lambda^{v(H)}(G)/ \hom (H,G)$ among all $K_{r+1}$-free graphs $G$.  
 This conjectured behavior is motivated by an interesting result \cite{MNNRW2023} from generalized Tur\'{a}n theory.  
When $H$ is a forest, the answer is affirmative; this is
Proposition \ref{prop:forest-conj}. When $H$
contains a cycle, the answer is negative: the regular complete $r$-partite graphs need not be extremal, 
even for large $r$.  
Examples \ref{ex:C4-not-turan} and \ref{ex:K23-not-turan} exhibit two different ways in which this fails.

\begin{example}[Even cycle] \label{ex:C4-not-turan} 
Note that $\hom (C_4,G) =\sum_i\lambda_i^4$, and the eigenvalues of $K_r[s]$ are  
$$\left\{ (r-1)s,\,(-s)^{(r-1)},\,0^{(r(s-1))} \right\}. $$
Then  we have 
\[  \frac{\lambda^{4}(K_r[s])}{\hom (C_4,K_r[s])}
  =\frac{(r-1)^{4}}{(r-1)^{4}+(r-1)}
  \ <\ 1 .
\]
On the other hand, Alon \cite[Theorem 2.1]{Alon1994} constructed a sequence of triangle-free $d$-regular graphs $G_n$ on
$n$ vertices with $d=\Theta (n^{2/3})$ and
$\theta :=\max_{i\ge 2}|\lambda_i|=O(n^{1/3})$. Then we have 
\[ 0 <  
  \hom (C_4,G_n)-\lambda^{4}(G_n)
  =\sum_{i\ge 2}\lambda_i^{4}
  \le \theta^{2}\bigl(nd-d^{2}\bigr) = O(n^{7/3}) 
  =O\bigl(n^{-1/3}\bigr)\lambda^{4}(G_n).
\]
Thus, we get $\lambda^4(G_n)/ \hom (C_4,G_n) \to 1$ and $\lambda(G_n) \to \infty$ as $n \to \infty$. 
Since triangle-free graphs are $K_{r+1}$-free for every $r\ge 2$, the limit superior in Problem \ref{prob-limsup} equals $1$ for the pair $(H, F)=(C_4, K_{r+1})$. 
It exceeds the value attained by the regular
complete $r$-partite graphs by a fixed factor. 
\end{example}

The tensor product
$G_1\otimes G_2$ is the graph with vertex set  $V(G_1)\times V(G_2)$ in which two pairs $(u_1,u_2)$ and
$(v_1,v_2)$ are adjacent whenever $u_iv_i\in E(G_i)$
for $i=1,2$. It is easy to see that 
$\lambda (G_1\otimes G_2)
 =\lambda (G_1)\lambda (G_2)$ and 
$\hom (H,G_1\otimes G_2)
 =\hom (H,G_1)\hom (H,G_2)$ for every $H$.

\begin{example}[Complete bipartite graph]
\label{ex:K23-not-turan}
 A homomorphism from
$K_{2,3}$ to $G$ is determined by an ordered pair $(u,v)$ of
vertices of $G$ together with three vertices of
$N_G(u)\cap N_G(v)$, so
\[
  \hom (K_{2,3},G) \ \  =\sum_{u,v\in V(G)}
  \bigl|N_G(u)\cap N_G(v)\bigr|^{3}.
\]
Taking $G=K_r$ gives
$\hom (K_{2,3},K_r)=r(r-1)^{3}+r(r-1)(r-2)^{3}$. 
Then for every $r\ge 3$, 
\[   \frac{\lambda^{5}(K_r[s])}{\hom (K_{2,3},K_r[s])} =  
\frac{\lambda^{5}(K_r )}{\hom (K_{2,3},K_r)} = 
 \frac{(r-1)^5}{r(r-1)^{3}+r(r-1)(r-2)^{3}} >1. \]
 Now write $G_0:=K_r[s]$, and let $G:=G_0^{\otimes k}$ be the $k$-th tensor power of $G_0$. 
 Then $G$ is again $K_{r+1}$-free. Since $\lambda (\cdot )$ and $\hom (K_{2,3},\cdot )$ are
multiplicative under tensor powers,  
we get 
\[   \frac{\lambda^{5}(G)}{\hom (K_{2,3},G)} = 
 \bigg(\frac{\lambda^{5}(G_0)}{\hom (K_{2,3},G_0)}\bigg)^{\! k}  \to \infty \quad \text{as $k\to \infty$}, \]
 while $\lambda (G)=\bigl((r-1)s\bigr)^{k}\to \infty$ as well. Hence, each tensor power of $G_0$ has a strictly larger ratio than $G_0$ itself. For the pair 
 $(H,F)=(K_{2,3},K_{r+1})$ with every $r\ge 3$, the limit superior in Problem
\ref{prob-limsup} is infinite. 
In particular, it is not attained at the regular complete $r$-partite graphs. 
\end{example}

\section*{Acknowledgments}
This work was supported by the Postdoctoral 
Science Foundation of China (No. 2026M793384) and 
the Postdoctoral Fellowship Program of China Postdoctoral Science Foundation (No. GZC20261656). 
The author would like to thank Hongzhang Chen and Yongjiang Wu for their valuable assistance in examining the possible problems. In particular, the author thanks Guorong Gao for discussing the proof of Lemma \ref{lem:large-perron}, which addresses the missing case of the original version of this paper.  The author developed the arguments, wrote and verified all proofs. During an early exploratory stage, the author additionally used language-model-based tools to brainstorm candidate proof strategies. Any suggestions arising from these tools served only as informal inspiration.

\appendix 
\section{Edge-spectral stability via weighted stability}
\label{sec:appendix-stability}

In Subsection \ref{subsec:stability-app}, we sketched how the weighted stability theorem yields the edge-spectral stability theorem of Li, Liu and Zhang \cite[Theorem 1.6]{LLZ-ESS}, and
we isolated the one point: the transition from the Perron weights $\pi_v=x_v^2$ to the vertex
counts and edge counts. This
appendix carries out that transition in full. 
We fix notation for the whole appendix. We
write $\kappa :=1-\frac1r$, and $a=(1\pm \theta )b$ means $|a-b|\le \theta b$. For graphs $G$ and $H$, we write
$d(G,H)=|E(G)\, \triangle \, E(H)|$ for the edit distance
between $G$ and $H$. For disjoint vertex sets $A$ and $B$, we write $K_{A,B}$ for the complete bipartite graph with parts $A$ and $B$. For a vertex set $K$, we write $T_{K,r}$ for an $r$-partite Tur\'{a}n  graph on $K$.

\begin{theorem}[Li--Liu--Zhang \cite{LLZ-ESS}]
\label{thm:edge-spectral-stability}
Let $F$ be a graph with $\chi (F)=r+1$. 
For every $\varepsilon > 0 $, there are  
$\delta >0$ and $m_0$ such that  
if $G$ is $F$-free with $m\ge m_0$ edges and
$  \lambda^2(G) \ge  (\kappa -\delta )2m$, 
then
\begin{itemize}
\item[\rm (a)] 
 for $r=2$, there exist disjoint vertex sets $A,B\subset V(G)$ such that $d(G,K_{A,B}) \le \varepsilon m$.
 
\item[\rm (b)]  for $r\ge 3$, there exists a vertex set $K\subseteq V(G)$ such that 
$d\bigl(G,T_{K,r}\bigr)\le \varepsilon m$. 
\end{itemize}
\end{theorem}

We deduce part (b) of Theorem
\ref{thm:edge-spectral-stability} from the weighted stability of Theorem \ref{thm:weighted-stability}. The only spectral ingredient is the flatness estimate of 
Lemma \ref{lem:perron-flatness}. 
This also explains why part (a) is not treated here. 
Indeed, Lemma
\ref{lem:perron-flatness} requires $\lambda^2(G)\ge \left( \frac12+\delta \right) w_2(G)$, and the hypothesis
$\lambda^{2}(G)\ge (\kappa -\delta )2m$ supplies this only
when $r\ge 3$. For $r=2$, we therefore quote part (a) from \cite{LLZ-ESS}.

\smallskip
The hypothesis $\lambda^{2}\ge (\kappa -\delta )2m$ says that
the Cauchy--Schwarz inequality 
$\lambda^{2}\le 2m\Phi_G(\pi )$ is
almost tight. Equality in this step forces the products
$x_ux_v$ to be constant over all edges of $G$, so
near-equality should force them to be nearly constant. Such a
near-constancy is exactly what converts a statement about the
weights $\pi_v$ into a statement about vertex and edge counts.
Lemma \ref{lem:good-edges} below makes it quantitative.
We fix a nonnegative unit Perron vector
$\bm{x}=(x_v)_{v\in V(G)}$ of $A(G)$, and we abbreviate
$\pi_v:=x_v^{2}$, $\Phi :=\Phi_G(\pi )$ and
$\lambda :=\lambda (G)$. The Cauchy--Schwarz inequality gives
\[
  \lambda^{2}
  =\biggl( 2\!\!\sum_{uv\in E(G)}\! x_ux_v \biggr)^{2}
  \ \le\ 2m\,\Phi ,
\]
with equality if and only if $x_ux_v= \frac{\lambda}{2m} = \frac{\Phi}{\lambda}$
for every $uv\in E(G)$.

\begin{lemma} \label{lem:good-edges}
Let $G$ be a graph with $m\ge 1$ edges, and let $\tau :=\Phi /\lambda$. Then
\[
  \mathcal{D}\ :=\!\!\sum_{uv\in E(G)}\!\!
  (x_ux_v-\tau )^{2}
  \ =\ \frac{\Phi}{2\lambda^{2}}
       \bigl(2m\Phi -\lambda^{2}\bigr)\ \ge\ 0 .
\]
Let $\gamma \in (0, 0.1)$, call an edge $uv\in E(G)$ \emph{good} if $|x_ux_v-\tau |\le \gamma \tau$ and \emph{bad} otherwise, and let $\mathcal{B}$ be the set of bad edges. 
If $  \Phi  \le \kappa +\sigma $ 
and  $ \lambda^{2} \ge (\kappa -\delta )2m$ 
for some $\sigma ,\delta \ge 0$ with
$\sigma +\delta \le \frac{1}{2}\gamma^{3}$, then
\[
  \sum_{uv\in \mathcal{B}}\pi_u\pi_v\ \le\ 4\gamma .
\]
\end{lemma}

\begin{proof}
Since $\sum_{uv\in E}x_ux_v 
=\frac{\lambda}{2}$ and
$\sum_{uv\in E}(x_ux_v)^{2}=\frac{\Phi}{2}$, expanding the
square gives
\[
  \mathcal{D}=\frac{\Phi}{2}-\lambda \tau +m\tau^{2}
  =\frac{m\Phi^{2}}{\lambda^{2}}-\frac{\Phi}{2},
\]
which is the required identity. Since $\mathcal{D}\ge 0$, the identity also yields $\lambda^{2}\le 2m\Phi$.

Suppose now that $\Phi \le \kappa +\sigma$ and
$\lambda^{2}\ge (\kappa -\delta )2m$. Then
\[
  2m\Phi -\lambda^{2}\le 2m(\sigma +\delta )
  \qquad \text{and} \qquad
  \frac{m}{\lambda^{2}}\le \frac{1}{2(\kappa -\delta )} .
\]
Here $\delta \le \sigma +\delta \le \frac12\gamma^{3}<0.1$
and $\kappa =1-\frac1r\ge \frac12$, so that
$\kappa -\delta >\frac13$. Moreover
$\Phi \le \kappa +\sigma \le \frac43$. Substituting these
bounds into the identity, we obtain
\[
  \mathcal{D}\ \le\ \frac{\Phi}{2\lambda^{2}}
  \cdot 2m (\sigma + \delta)
  \ \le\ \frac{\Phi }{2(\kappa -\delta )} \cdot
  (\sigma + \delta) \
  \le\ 2 (\sigma + \delta) .
\] 
Every bad edge contributes more than $\gamma^{2}\tau^{2}$ to
$\mathcal{D}$, so
$|\mathcal{B}|\,\tau^{2}\le \mathcal{D}/\gamma^{2}$. Since
$\pi_u\pi_v=(x_ux_v)^{2}\le 2(x_ux_v-\tau )^{2}+2\tau^{2}$
and $\gamma \le 1$, we conclude that
\[
  \sum_{uv\in \mathcal{B}}\pi_u\pi_v
  \ \le\ 2\mathcal{D}+2|\mathcal{B}|\tau^{2}
  \ \le\ \frac{4\mathcal{D}}{\gamma^{2}}
  \ \le\ \frac{8(\sigma + \delta)}{\gamma^{2}}\ \le\ 4\gamma .
  \qedhere
\]
\end{proof}

The proof of Theorem \ref{thm:edge-spectral-stability}(b) has three steps. 
Step 1 verifies the hypotheses of Theorem
\ref{thm:weighted-stability} and records the partition it supplies. 
Step 2 uses Lemma \ref{lem:good-edges} to pin down $\pi$ on a set of almost full measure. Step 3 turns the resulting flatness into counts of pairs, and then into the edit distance.

\begin{proof}[{\bf Proof of Theorem
\ref{thm:edge-spectral-stability}(b)}] 
The proof depends on a small parameter
$0< \sigma < (10r)^{-10} $, whose value will be determined later. Given  $\sigma$, let $\eta$ be the
smaller of the two constants $\eta (F,\sigma )$ supplied by
Theorems \ref{thm-weighted-ESS} and
\ref{thm:weighted-stability}, and let $\delta_1$ be the
constant $\delta (F,\sigma )$ of Theorem
\ref{thm:weighted-stability}. 
 Constants implicit in $O(\cdot )$ depend on $r$ alone. 
Put $\gamma :=(4\sigma )^{1/3}$ and $\delta :=\min \{\delta_1,\sigma \}$. Let
$m_0=m_0(F,\varepsilon )$ be such that
\begin{equation}\label{eq:es-m0}
  \frac{20^{8}}{\sqrt{m_0}}\ \le\ \eta
  \qquad\text{and}\qquad
  \sqrt{m_0}\ \ge\ \frac{2r}{\gamma}.
\end{equation}
Suppose that $G$ is an $F$-free graph with $m\ge m_0$ edges and
$\lambda^{2}\ge (\kappa -\delta )2m$.

\medskip\noindent
\textit{Step 1. An $r$-partition with small weight.}\
First of all, we claim that 
\begin{equation}\label{eq:es-step1}
  \lVert \pi \rVert_\infty \le \eta
  \qquad\text{and}\qquad
  \kappa -\delta \ \le\ \Phi \ \le\ \kappa +\sigma .
\end{equation} 
We begin with the first estimate in \eqref{eq:es-step1}. The
hypothesis reads $\lambda^{2}\ge (\kappa -\delta )w_2(G)$,
and here $r\ge 3$ enters for the first time: it gives
$\kappa -\delta \ge \frac23-\sigma >\frac12+\frac1{20}$, so
that Lemma \ref{lem:perron-flatness} applies with $\ell =2$
and with $\frac{1}{20}$ in place of $\delta$. It yields
$\lVert \pi \rVert_\infty <20^{8}/\lambda \le
20^{8}/\sqrt{m}\le \eta$, the last step by \eqref{eq:es-m0}.

The flatness just proved is precisely the hypothesis of
Theorem \ref{thm-weighted-ESS}. Applying that theorem to $G$
with the weight vector $\pi$, we obtain the upper bound
$\Phi \le \kappa +\sigma$. The matching lower bound is
immediate from Lemma \ref{lem:good-edges}, which gives
$\Phi \ge \lambda^{2}/(2m)\ge \kappa -\delta$. This proves  \eqref{eq:es-step1}, as claimed.

We now find a partition with small missing weight. 
Since $\Phi \ge \kappa -\delta$ and $\delta \le \delta_1$, Theorem
\ref{thm:weighted-stability} applies to $G$ and $\pi$ and
produces a partition $V(G)=U_1\cup \cdots \cup U_r$ with
\begin{equation}\label{eq:es-partition}
  \Bigl|\pi (U_i)-\frac1r\Bigr|\le \sigma \quad 
  \text{for every}~i\in [r],
  \quad \text{and} \quad 
  M:=M(G,\pi ,U_1,\ldots ,U_r)\le \sigma .
\end{equation} 
Write $a_i:=\pi (U_i)$ and $b_i:= a_i - \frac{1}{r} $, 
so that $|b_i| \le \sigma$ for every $i\in [r]$. Since  $\sum_{i=1}^r a_i =1$, we have $\sum_{i=1}^r b_i =0$, and therefore   
$\sum_{i=1}^r a_i^{2} = \sum_{i=1}^r (\frac{1}{r^2} + \frac{2b_i}{r} + b_i^2) \le \frac1r+r\sigma^{2}$.  
Split $\Phi$ along the partition as
$\Phi =\Phi_{\mathrm{cr}}+\Phi_{\mathrm{in}}$, where
$\Phi_{\mathrm{cr}}$ collects the edges of $G$ joining two
distinct parts and $\Phi_{\mathrm{in}}$ collects the edges
lying inside a part. Every cross pair is either an edge of
$G$ or a missing cross pair counted by $M$, so
\eqref{eq:es-partition} gives $\Phi_{\rm cr} = (\sum_{v\in V} \pi_v)^2 -\sum_{i=1}^r a_i^{2}-2M \ge 1- \frac{1}{r} - r\sigma^2 - 2M \ge \kappa - 3\sigma$. 
Combining this with the upper bound in \eqref{eq:es-step1},
we bound the inside part 
$  \Phi_{\mathrm{in}} =  \Phi - \Phi_{\rm cr} \le (\kappa + \sigma) - (\kappa -3\sigma) \le 4\sigma$.

\smallskip\noindent
\textit{Step 2. A large core of nearly equal weights.}\
By \eqref{eq:es-step1} and
$\sigma + \delta \le 2\sigma \le \frac{1}{2} \gamma^3$, the
hypotheses of Lemma \ref{lem:good-edges} hold for $(G,\pi )$.
Call $K:=\{u\in V(G):\ \pi_u=(1\pm 4\gamma )\tau \}$ the
\emph{core} of $G$, 
and let us prove that $\pi (K)=1-O(\gamma )$.
The argument proceeds in two stages. We first show that almost all the weight sits on ``typical'' vertices, which are joined by good edges to all but a small fraction of the weight of the other parts, and we then show that every typical vertex belongs to $K$.

For $u\in V(G)$, let $i=i(u)$ be the index with $u\in U_{i}$, and let $\mathrm{def}(u)$ be the total weight of those
vertices of $V(G)\setminus U_{i}$ that are not joined to $u$ by a good edge. Every vertex counted by $\mathrm{def}(u)$ is either a non-neighbor of $u$ or a neighbor joined to $u$ by a bad edge, whence Lemma \ref{lem:good-edges} gives
\[
  \sum_{u\in V}\pi_u\,\mathrm{def} (u)
  \ \le\ 2 M+ 2 \sum_{uv\in \mathcal{B}}\! \pi_u\pi_v
  \ \le\ 2 \sigma +8\gamma \ \le\ 10\gamma .
\]
Call a vertex $u$ \emph{typical} if
$\mathrm{def}(u)\le \frac{1}{50r}$, and let $W$ be the set of
typical vertices of $G$. Since every vertex $u$ outside $W$ satisfies $\mathrm{def}(u) > \frac{1}{50r}$, we get
$\pi (V(G) \setminus W) \le 50r \cdot 10\gamma =O(\gamma)$.

Next, we claim that every typical vertex lies in the core.
Let $u$ be a typical vertex, 
and relabel the parts so that $u\in U_1$.
For $j\in \{2,3\}$, which is where $r\ge 3$ enters for the
second time, put
\[
  A_j:=\{v\in U_j:\ uv\ \text{is a good edge of}\ G\}.
\]
Every vertex of $U_j\setminus A_j$ is counted by
$\mathrm{def}(u)$, so \eqref{eq:es-partition} gives
$\pi (A_j)\ge \pi (U_j)-\mathrm{def}(u)\ge
\frac1r-\sigma -\frac{1}{50r}\ge \frac{9}{10r}$. 
 Draw $v\in A_2$ and $w\in A_3$ independently, with probabilities
${\pi_v}/{\pi (A_2)}$ and ${\pi_w}/{\pi (A_3)}$. 
Since $A_2$ and $A_3$
lie in different parts, the probability that $vw$ fails to be
a good edge is at most
\[
  \frac{M+\sum_{uv\in \mathcal{B}}\pi_u\pi_v}
       {\pi (A_2)\,\pi (A_3)}
  \ \le\ \frac{100\,r^{2}}{81}\cdot 5\gamma
  \ \le\ 7r^{2}\gamma \ <\ 1 .
\]
So we may fix $v$ and $w$ for which $uv$, $uw$ and $vw$ are
all good. It follows that 
\[
  \pi_u\ =\ \frac{(x_ux_v)(x_ux_w)}{x_vx_w}
  \ \in\ \Bigl[\tfrac{(1-\gamma )^{2}}{1+\gamma}\,\tau ,\
  \tfrac{(1+\gamma )^{2}}{1-\gamma}\,\tau \Bigr]
  \ \subseteq\ \bigl[(1-4\gamma )\tau ,\,
  (1+4\gamma )\tau \bigr]. 
\]
Hence $u\in K$, and so $W\subseteq K$. 
Since $\pi (V(G)\setminus W) = O(\gamma)$, we conclude that  
\begin{equation} \label{eq-large-core}
  \pi (V(G)\setminus K) = O(\gamma) 
  \quad \text{and} \quad 
\pi (K)= 1- O(\gamma). 
\end{equation}  
\textit{Step 3. From Perron weights to edge counts.} 
Steps 1 and 2 concern only the weights $\pi$,  
and we now convert them into edge counts. 
Write $n:=|K|$ and $K_i:=K\cap U_i$. Summing
$\pi_u=(1\pm 4\gamma )\tau$ over all $u\in K$ gives 
$(1+ 4\gamma)\tau n \ge \sum_{u\in K} \pi_u  = \pi (K) \ge 1- O(\gamma)$ by (\ref{eq-large-core}), 
so $n =(1\pm O(\gamma )) \frac{1}{\tau}$.
 Since $\tau =\Phi /\lambda$ and $\Phi =(1\pm 2\sigma )\kappa$ by
\eqref{eq:es-step1}, we get 
$n=(1\pm O(\gamma ))\lambda /\kappa$. 
By \eqref{eq:es-partition} and Step 2, we have 
$\pi (K_i)= \pi (U_i) \pm O(\gamma) = \frac1r\pm O(\gamma )$. 
Together with
$(\kappa -\delta )2m\le \lambda^{2}\le 2m\Phi \le
(\kappa +\sigma )2m$, we obtain
\begin{equation}\label{eq:es-core}
  \pi_u= \big(1\pm O(\gamma ) \big) \frac{1}{n}\ \ (u\in K),
  \quad
  |K_i|= \big(1\pm O(\gamma ) \big)\frac{n}{r} \ \ (i\in [r]),
  \quad
  \kappa n^{2}=\big(1\pm O(\gamma ) \big) 2m ,
\end{equation}
and in particular
$n^{2}\le (1+O(\gamma ))\,2m/\kappa \le 4m$.

\smallskip 
We are now ready to count edges. 
Since $\pi_u= \big(1\pm O(\gamma ) \big) \frac{1}{n}$ for every $u\in K$, 
every cross pair of $K$ that is a
non-edge contributes at least $\frac{1}{2n^{2}}$ to $M$, and
every edge with both ends in one $K_i$ contributes at least
$\frac{1}{n^{2}}$ to the inside part $\Phi_{\mathrm{in}}$ of $\Phi$. 
Since $M\le \sigma$ and $\Phi_{\mathrm{in}} \le 4\sigma$, 
there are at most $2\sigma n^{2}\le 8\gamma \, m$ cross non-edges of $K$, 
and at most $4\sigma n^{2}\le 16\gamma \, m$ edges inside the sets $K_i$. Let $T_0$ be the complete $r$-partite graph on $K$
with vertex parts $K_1,K_2,\ldots ,K_r$. 
By \eqref{eq:es-core}, 
we get  
$$ e(T_0)=\tfrac12\bigl(n^{2}- {\textstyle \sum_{i}} |K_i|^{2}\bigr) 
= \big(1\pm O(\gamma ) \big) \tfrac{1}{2} \Big(n^2 - r(\tfrac{n^2}{r^2}) \Big) 
= \big(1\pm O(\gamma ) \big)\tfrac{\kappa n^{2}}{2}
= \big(1\pm O(\gamma ) \big)m.$$ 
Since all edges of $T_0$ but at most
$8\gamma m$ of them are edges of $G$, we have 
$e_G(K) \ge e(T_0) - 8\gamma m \ge (1-O(\gamma ))m$. An edge of $G$ outside $T_0$
 either has an end outside $K$, of which there are
$m-e_G(K)=O(\gamma )m$, or lies inside some $K_i$, of which
there are at most $16\gamma m$. 
On the other hand, 
an edge of $T_0$ outside
$G$ is a cross non-edge inside $K$, of which there are at
most $8\gamma m$. Hence $d(G,T_0)=O(\gamma )m$.

Finally, rebalancing $T_0$ into $T_{K,r}$ moves at most
$O(\gamma )n$ vertices by \eqref{eq:es-core}, each move
changing at most $n$ pairs. 
Then $d(T_0,T_{K,r}) = O(\gamma)n^2 =O(\gamma) m$, 
so $d (G,T_{K,r})\le  C(r) \gamma \, m$, where $C(r)$
depends on $r$ alone. Choosing
$\sigma :=\frac14 (\frac{\varepsilon}{C(r)})^{3}$
makes $\gamma \le \varepsilon /C(r)$, hence
$d\bigl(G,T_{K,r}\bigr)\le \varepsilon m$, as required.
\end{proof}

\end{document}